\documentclass{amsart}
\usepackage{amsmath,amsfonts,amssymb,amscd,verbatim,multicol,graphicx}
\usepackage[arrow,matrix,cmtip]{xy}
\usepackage{fullpage}

\usepackage{bbold}

\DeclareMathOperator{\spann}{span} \DeclareMathOperator{\coker}{coker}
\newcommand{\thra}{\twoheadrightarrow}
\newcommand{\sA}{\sh{A}}

\newcommand{\sF}{\sh{F}}
\newcommand{\sG}{\sh{G}}
\newcommand{\sH}{\sh{H}}
\newcommand{\sI}{\sh{I}}
\newcommand{\sJ}{\sh{J}}
\newcommand{\sK}{\sh{K}}
\newcommand{\sL}{\sh{L}}

\newcommand{\sN}{\sh{N}}
\newcommand{\sO}{\sh{O}}

\newcommand{\sR}{\sh{R}}
\newcommand{\sS}{\sh{S}}

\newcommand{\CM}{{\mathbf H}}
\newcommand{\KM}{{\mathbf K}}
\newcommand{\QM}{{\mathbf Q}}

\newtheorem{conjecture}{Conjecture}

\newcommand{\Div}{\operatorname{Div}}

\newcommand{\bpr}{\begin{proof}}
\newcommand{\epr}{\end{proof}}

\newcommand{\spec}{\operatorname{Spec}}

\newcommand{\D}{\displaystyle}
\newcommand{\mc}{\mathcal}
\newcommand{\mf}{\mathfrak}
\newcommand{\cha}{\operatorname{char}}
\newcommand{\mb}{\mathbb}

\newcommand{\dra}{\dashrightarrow}

\newcommand{\GK}{\operatorname{GKdim}}

\newcommand{\wt}{\widetilde}

\newcommand{\gldim}{\operatorname{gl.dim}}

\newcommand{\aut}{\operatorname{Aut}}

\DeclareMathOperator{\HB}{H}

\newcommand{\beq}{\begin{equation}}
\newcommand{\eeq}{\end{equation}}

\newcommand{\Hom}{{\rm Hom}}

\newcommand{\Ext}{{\rm Ext}}

\newcommand{\Aut}{{\rm Aut}}

\newcommand{\ang}[1]{\langle #1 \rangle}

\newcommand{\blank}{\mbox{$\underline{\makebox[10pt]{}}$}}

\newcommand{\st}{\left\vert\right.}

 \DeclareMathOperator{\rGr}{Gr-\!}

  \DeclareMathOperator{\lMod}{\!-Mod}  

\newcommand{\sh}{\mathcal}
\newcommand{\struct}{\mathcal{O}}
\renewcommand{\Lsh}{\mathcal{L}}
\newcommand{\shTor}{\mathcal{T}{\it or}}

\newcommand{\FF}{{\mathbb F}}
\newcommand{\ZZ}{{\mathbb Z}}

\newcommand{\PP}{{\mathbb P}}
\newcommand{\PGL}{\PP GL}
\newcommand{\NN}{\mathbb N}
\newcommand{\VV}{\mathbb V}
\newcommand{\kk}{\Bbbk} 
\newcommand{\GG}{\mb G}

\DeclareMathOperator{\im}{im} \DeclareMathOperator{\Pic}{Pic} \DeclareMathOperator{\pic}{Pic}

 \DeclareMathOperator{\pd}{pd} \DeclareMathOperator{\depth}{depth}

\numberwithin{equation}{section}
 \theoremstyle{plain}
\newtheorem{theorem}[equation]{Theorem}

\newtheorem{lemma}[equation]{Lemma}

\newtheorem{corollary}[equation]{Corollary}
\newtheorem{proposition}[equation]{Proposition}

\theoremstyle{definition}

\newtheorem{definition}[equation]{Definition}
\newtheorem{notation}[equation]{Notation}

\newtheorem{remark}[equation]{Remark}

\newtheorem{standing-hypothesis}[equation]{Standing Hypothesis}

\newtheorem*{defn}{Definition}

\newcommand{\Twiddle}{\widetilde{T}}
\newcommand{\coeff}{\binom{n+1}{2}}

 \DeclareMathOperator{\len}{len}

\newcommand{\Lfam}{\sF}
\newcommand{\Xfam}{H'}
\newcommand{\Yfam}{H}
\newcommand{\Ifam}{\sJ}
\newcommand{\Bfam}{ C}

\title{Some noncommutative projective surfaces of GK-dimension 4}
\author{D. Rogalski}
\address{ UCSD Department of Mathematics\\ 9500
Gilman Dr. \#0112\\ La Jolla, CA 92093-0112, USA}
\email{drogalsk@math.ucsd.edu}
\author{S. J. Sierra}
\address{Department of Mathematics\\Princeton University\\Fine Hall, Washington Rd.\\Princeton, NJ 08544, USA}
\email{ssierra@princeton.edu}
\date{\today}

\begin{document}
\begin{abstract}
We construct an interesting family of connected graded domains of GK-dimension 4, and show that the general member of this family is noetherian.
The algebras we construct are Koszul and have global dimension 4.  They fail to be Artin-Schelter Gorenstein, however, showing that a theorem of Zhang and Stephenson for dimension 3 algebras does not extend to dimension 4.  The Auslander-Buchsbaum formula also fails to hold for these algebras.

The algebras we construct are  birational to $\PP^2$, and their existence  disproves a conjecture of the first author and Stafford.  The algebras can be obtained as global sections of a certain quasicoherent graded sheaf on $\PP^1 \times \PP^1$, and our key technique is to work with this sheaf.  In contrast to all previously known examples of birationally commutative graded domains, the graded pieces of the sheaf fail to be ample in the sense of Van den Bergh.   Our results thus require significantly new techniques.
\end{abstract}

\maketitle

\tableofcontents

\section{Introduction}\label{INTRO}
Let $\kk$ be an uncountable algebraically closed field, and let $R $ be an $\NN$-graded $\kk$-algebra.   We assume also that $R$ is {\em connected graded}: i.e. $R_0 = \kk$ and ${\rm dim }_{\kk} R_n < \infty$ for all $n$.
Recall that $R$ is {\em Artin-Schelter (AS) regular} if $R$ has global dimension $d < \infty$, finite GK-dimension, and $R\Hom_R(\kk, R) = \kk[\ell]$ for some $\ell$.   In general, AS-regular algebras of dimension $d$ which are also Koszul (in which case $\ell=d$) are considered to be the noncommutative analogues of polynomial rings in $d$ variables.

The  condition that $R\Hom_R(\kk, R) = \kk[\ell]$ is the {\em Artin-Schelter Gorenstein} condition, and it is the most mysterious of the three conditions for regularity.  It has been shown in many cases to imply that the algebra $R$ is well-behaved, and in fact all known examples of AS-regular algebras are noetherian domains.  Conversely, Stephenson and Zhang \cite{SteZh} proved that if a connected graded algebra $R$ is noetherian, Koszul, and has global dimension 3, it is AS-Gorenstein.  In fact, there has been speculation \cite[p.~195]{SV} that this is true if $R$ has arbitrary finite global dimension.

In this paper we show, surprisingly, that the results of \cite{SteZh} fail in dimension 4.   We prove:
\begin{theorem}\label{ithm-main2}
 There is a connected graded noetherian domain $R$ with $\GK R = \gldim R = 4$ such that $R$ is Koszul but not AS-Gorenstein.
\end{theorem}
The algebra $R$ above has other surprising homological properties.  In particular, the {\em Auslander-Buchsbaum equality} fails for $R$:  we have ${\rm depth } \  \kk + { \rm p. dim }\  \kk = 0 + 4 > {\rm depth }\  R = 2$.   (See Section~\ref{NOTATION} for definitions.)
It is worth noting that $R$ is not only Koszul, but the trivial module has a resolution of the form
\beq\label{res}
0 \to R[-4] \to R[-3]^4 \to R[-2]^6 \to R[-1]^4 \to R \to \kk \to 0,
\eeq
so that $R$ has many properties of a polynomial ring in 4 variables.

The algebra $R$ is interesting in another way, as well:  it is a counterexample to a conjecture in the classification of noncommutative projective surfaces.  This classification is
one of the most important open problems in the subject of noncommutative projective geometry, and is far from complete.  An important special case, however, is well-understood:  {\em birationally commutative} surfaces, defined here as connected graded noetherian domains $R$ whose graded quotient ring is isomorphic to $K[t, t^{-1}; \varphi]$ for some field $K$ of transcendence degree 2.

If such $R$ have GK-dimension 3 or 5, then they have been classified \cite{RS} \cite{S-surfclass} \cite{S-GK5}.  In fact, in \cite{RS} the first author and Stafford conjecture that this classification is complete:  {\em all} birationally commutative surfaces have GK-dimension 3 or 5.  The algebra $R$ above is a counterexample to this conjecture, and shows that there is  still unexplored territory in the study of birationally commutative surfaces.

Let us explain the conjecture of \cite{RS} further.  We begin by summarizing their results on birationally commutative graded algebras.
 If $S$ is a connected graded noetherian domain (or Ore domain, more generally) we may invert the homogeneous elements of $S$ to obtain the {\em graded quotient ring} $Q_{\rm gr}(S)$.  We have  $Q_{\rm gr}(S) \cong D[t, t^{-1}; \varphi]$, where $D$ is a  division ring  and $\varphi \in {\rm Aut}_{\kk}(D)$.  If $D $ is commutative, and thus $D \cong \kk(X)$ for some projective variety $X$, then we say that $S$ is {\em birational to $X$}, or more generally {\em birationally commutative}. If $\varphi$ is induced from an automorphism of $X$, we say that $S$ is {\em geometric}.

Subject to the condition that $S$ is geometric and generated in degree 1, \cite{RS} classifies birationally commutative projective surfaces: in other words those graded algebras birational to a commutative surface.  The prototypical example is a \emph{twisted homogeneous coordinate ring} $B(X, \mc{L}, \sigma) = \bigoplus_{n \in \NN} \HB^0(X, \mc{L} \otimes \sigma^* \mc{L} \otimes \dots \otimes (\sigma^{n-1})^*\mc{L})$ for some projective surface $X$ with automorphism $\sigma: X \to X$ and invertible sheaf $\mc{L}$ on $X$ (see \cite{AV} for more details about this construction). They show:
\begin{theorem} (\cite[Theorem~1.1]{RS})
\label{thm-RS} Let $S$ be a connected graded noetherian domain that is generated in degree 1, birational to a commutative surface, and geometric. Then up to a finite-dimensional vector space, $S$ is either a twisted homogeneous coordinate ring $B(X, \mc{L}, \sigma)$ or else a closely related subring $R(X, Z, \mc{L}, \sigma)$ of such (called a \emph{na\"ive blowup algebra}).
\end{theorem}
More generally, the second author has obtained in \cite{S-surfclass}, \cite{S-GK5} a similar classification result without the generation in degree $1$ hypothesis.  Possibly after passing to a Veronese subring, all such algebras are again (special kinds of) subrings of twisted homogeneous coordinate rings.  In all cases the rings involved may be written explicitly in terms of commutative geometric data.

The first author and Stafford conjectured \cite[p. 6]{RS} that the conclusions of Theorem~\ref{thm-RS} hold without the assumption that $S$ is geometric.  The conjecture was motivated by the following theorem of the first author, using work of Diller and Favre \cite{DF2001}.

\begin{theorem}\label{thm-R}
(\cite{R-GK}) Let $Q := K[t, t^{-1}; \varphi]$, where $K$ is a finitely generated field extension of $\kk$ of transcendence degree 2. Then every connected graded Ore domain
 $S$ with graded quotient ring $Q$ has the same GK-dimension $d \in \{ 3, 4, 5, \infty\}$.
If $d=\infty$, then $S$ is not noetherian.  If $d < \infty$, then $d \in \{ 3, 5\}$ if and only if $S$ is geometric.
\end{theorem}
It follows from this result that   the hypotheses of Theorem~\ref{thm-RS} imply restrictions on the GK-dimension of $S$.
Thus we may restate the first author and Stafford's conjecture as
\begin{conjecture}\label{conj-RS}
[Rogalski-Stafford] Suppose that $S$ is a connected graded domain of GK-dimension 4, generated in degree 1, that is birational to a commutative surface.  Then $S$ is not noetherian.
\end{conjecture}

We note that it is not hard to write down examples of connected graded domains of GK-dimension 4 that are birational to commutative surfaces; some of these are  even finitely presented and  Koszul.  However, analyzing their ring-theoretic structure is nontrivial, and the  examples known until now are  not    noetherian.

We begin with one of these non-noetherian examples, and examine a family of (non-formal) deformations of it.  Namely, let $K := \kk(u,v)$ be a rational function field and define $\sigma: K \to K$ by $\sigma(u) = uv, \sigma(v) = v$.  Setting $E := \kk  + \kk u + \kk v + \kk uv$, the $\kk$-subalgebra $A$ of $K[t, t^{-1}; \sigma]$
generated by $Et$ is a non-noetherian Koszul algebra of GK-dimension 4; it has appeared before in the literature, for example in \cite[Proposition~7.6]{YZ2006}.

To deform $A$, we perturb $\sigma$ by a 2-parameter family of automorphisms of $K$. Given $\rho, \theta \in \kk^* $, let
 $\tau = \tau(\rho, \theta): K \to K$ be given by
\[
\tau(u)  = \frac{(\rho +1) u + (\rho -1)}{(\rho -1) u + (\rho + 1)}, \ \ \ \tau(v) = \frac{(\theta + 1) v + (\theta -1)}{(\theta -1) v + (\theta + 1)}. \] Let $\GG := \{\tau(\rho, \theta) \st \rho, \theta \in \kk^*\}$ be the subgroup of $\Aut(\PP^1 \times \PP^1)$ that fixes each of the four points $[ \pm 1: 1][\pm 1: 1]$. Given $\tau \in \GG$, define  $\varphi := \sigma \circ \tau$.  Let
\begin{equation}\label{foo}
 R(\tau) := \kk \langle Et \rangle \subseteq  k(u,v)[t, t^{-1}; \varphi].
\end{equation}
The family of algebras of interest is then
$\{ R(\tau) \st \tau \in \GG\}$.

We may now state our main result.
\begin{theorem}\label{ithm-main}
Let $R = R(\tau)$ be as in \eqref{foo}, where $\tau = \tau(\rho, \theta)\in \GG$.  For any pair $(\rho, \theta)$ which is algebraically independent over the prime subfield of $\kk$, the algebra $R(\tau)$ is a noetherian domain of GK-dimension 4 that is birational to $\mathbb{P}^2$.  Further, the results of Theorem~\ref{ithm-main2} hold for this $R$.
\end{theorem}
It is not  hard to show using the growth criterion underlying Theorem~\ref{thm-R} that the rings $R(\tau)$ have GK-dimension $4$ for the very general choices of $(\rho, \theta)$ in the theorem (interestingly, though, the GK-dimension of $R(\tau)$ is $3$ instead for some sporadic non-general choices.)  It is also fairly straightforward that they have global dimension 4 and are Koszul but not AS-Gorenstein.  The fact that the rings $R(\tau)$ are noetherian for general $\tau$ is the deeper content of the result.  The proof of this requires  techniques that differ substantially from past work on birationally commutative algebras and occupies most of the second half of the paper.

The ring $R=R(\tau)$ in Theorem~\ref{ithm-main} must be non-geometric  by Theorem~\ref{thm-R}.  Thus there is no rational projective surface $X$
 for which $\varphi: K \to K$ corresponds to an automorphism $\phi: X \to X$.  This certainly precludes the possibility that $R$ has the same graded quotient ring as any twisted homogeneous coordinate ring, so Theorem~\ref{thm-RS} fails completely in this case. We conjecture that the augmentation ideal $R_+$ is the only nontrivial graded prime ideal of $R$ and thus that $R$ has no nontrivial map to any twisted homogeneous coordinate ring.

To conclude the introduction, we give an overview of the ideas behind the proofs.  Let $R = R(\tau)$ where $\tau = \tau (\rho, \theta)$ for a pair $(\rho, \theta)$ algebraically independent over the prime subfield of $\kk$.  As in past work on birationally commutative algebras, we would like to work as far as possible with  sheaves on a projective variety, rather than $R$-modules.  We work on $T := \mathbb{P}^1\times\mathbb{P}^1$.    Let $\phi:  T\dra T$ be the birational self-map induced by $\varphi = \sigma\tau$.  We have $R\subseteq \kk(T)[t, t^{-1}; \varphi]$.  Let $\mc{R}_n$ be the subsheaf of the constant sheaf of rational functions generated by $R_nt^{-n} \subseteq \kk(T)$.
Much of the first half of the paper is devoted to proving that $R = \bigoplus H^0(T, \sR_n)$ for general $\tau$.  The proof requires a careful analysis of how the properties of the birational maps $\phi^n$ change as $\tau$ varies.   The proof that $\kk_R$ has a resolution as in \eqref{res} is intertwined and must be done simultaneously.

Writing the ring $R$ as a ring of sections is crucial to the proof of the noetherian property, but working with the sheaves $\mathcal{R}_n$ is quite delicate.  In all past work on birationally commutative algebras, the main idea has been to show that the sequence of sheaves $\mathcal{R}_n$ is an \emph{ample sequence} in the sense of \cite{VdB1996}.  There is then a purely formal category equivalence between the category of tails of $R$-modules and the category of tails of (appropriately defined) $\mc{R}$-modules. This allows one to translate questions about properties of $R$ (including the noetherian property) to more tractable geometric questions about $\mc{R}$. Unfortunately, in our case the sheaves $\mathcal{R}_n$ do not form an ample sequence.  This requires us to develop new methods to show the ring $R$ is noetherian.
More specifically, if $\sF$ is a  coherent sheaf on $T$, then $\bigoplus_n H^1(T, \mc{F} \otimes \mc{R}_n)$ may not be finite-dimensional; however,  this graded vector space carries a natural $R$-action and a key step in our proof is  showing that such an $R$-module, which we call a \emph{cohomology module}, is noetherian.

{\bf Acknowledgments.} We thank Mark Gross, Peter J{\o}rgensen, Jessica Sidman, Toby Stafford, Michel Van den Bergh, and James Zhang for helpful discussions and comments.
The first author was supported by NSF grant DMS-0900981, and the second author was supported  by NSF grant DMS-0802935.

\section{Review of pullback by a birational map}
\label{PULLBACK} Throughout the paper, we will work with an automorphism $\varphi$ of $K = \kk(u,v)$ and the induced birational self-map $\phi$ of $T = \PP^1 \times \PP^1$.  Explicitly, $\varphi$ is equal to the pullback action of $\phi$ on $K = \kk(T)$.  In this section, we review  basic facts about pullback of divisors under a birational self-map of a surface. See \cite{R-GK} for a longer discussion of this, which closely follows the ideas in \cite{DF2001}. In this paper we will only need a few selected results.

Fix an algebraically closed base field $\kk$.  Suppose that $\psi: U \to T$ is a regular morphism of nonsingular varieties over $\kk$. Then there is a standard pullback map of (Weil) divisors $\psi^*: \Div T \to \Div U$.
Now let $\psi: S \dra T$ be a birational map of nonsingular varieties.  Then the domain of definition of $\psi$ is an open set $U = S \smallsetminus Z$, where $Z$ has codimension at least $2$ \cite[Lemma V.5.1]{Ha}. Since $\Div S = \Div U$ \cite[Proposition II.6.5]{Ha} we get a pullback map $\psi^*: \Div T \to \Div U = \Div S$.  We define $\psi_*: \Div S \to \Div T$ as $(\psi^{-1})^*$.

On $S$ and $T$, there is a one-to-one correspondence between Weil divisors and invertible subsheaves of $\mc{K}$, where $\mc{K}$ is the constant sheaf of rational functions.  Thus the pullback map defined above also induces a pullback map on invertible subsheaves of $\mc{K}$. If $\mc{M} \subseteq \mc{K}$ we write $\mc{M}^{\psi} \subseteq \mc{K}$ for the pulled back invertible sheaf.  Note that $\mc{M}^{\psi}$ is the unique invertible subsheaf $\mc{N}$ of $\mc{K}$ such that $\mc{N} \vert_U = (\psi \vert_U)^*(\mc{M})$. Let $(f)$ be the principal divisor associated to a rational function $f \in \kk(T)$.
We often use the notation $f^{\psi}: =f \circ \psi$.  It follows directly from the definitions that $ (f^{\psi}) = \psi^* ( f)$, and thus pullback of divisors also induces a pullback map $\psi^*: \pic T \to \pic S$ on the Picard groups.  This also shows that given any invertible sheaf $\mc{M}$ (without a framing inside $\mc{K}$), there is an invertible sheaf $\mc{M}^{\psi}$ which is well-defined up to isomorphism.

One thing that makes pullback by a birational map a subtle operation is that  it can behave poorly with respect to composition. In particular, when one attempts to iterate pullback by a birational self-map $\psi: S \dra S$ of a nonsingular surface, one may have $(\psi^n)^* \neq (\psi^*)^n$.  The simplest example of this is the Cremona transformation $\psi: \mb{P}^2 \dra \mb{P}^2$ defined by $[x : y  : z] \mapsto [yz: xz: xy]$.  At the level of the Picard group $\Pic \mb{P}^2 = \mb{Z}$, the map $\psi^*: \mb{Z} \to \mb{Z}$ is multiplication by $2$, while $\psi^2$ is the identity map and so $(\psi^2)^* \neq (\psi^*)^2$.   The following gives a sufficient condition to avoid such pathologies.
\begin{lemma} (\cite[Lemma 2.4]{R-GK})
\label{comp-lem} Let $\phi, \psi: S \dra S$ be birational self-maps of a nonsingular surface $S$.  If there does not exist a curve $C$ on $S$ such that $\phi$ contracts $C$ to a fundamental point of $\psi$, then  $(\psi \circ \phi)^* = \phi^* \circ \psi^*$ as maps $\Div S \to \Div S$.
\end{lemma}
Note that the images of curves contracting under $\phi$ are precisely the fundamental points for the map $\phi^{-1}$, by Zariski's main theorem. So an equivalent formulation of the hypothesis of the lemma is to assume that $\phi^{-1}$ and $\psi$ have no common fundamental points.
\begin{definition} (\cite[Definition 2.6]{R-GK})
\label{stable-def} A birational map $\psi: S \dra S$ of a nonsingular projective surface $S$ is called \emph{stable} if for all $n \geq 1$, there is no curve $C$ such that $\psi^n$ contracts $C$ to a fundamental point of $\psi$ (equivalently, if for all $n \in \NN$, $\psi^{-n}$ and $\psi$ have no common fundamental points.)
\end{definition}
The definition follows \cite{DF2001}, where the term analytically stable is used.  By Lemma~\ref{comp-lem}, a stable birational map has $(\psi^n)^* = (\psi^*)^n$ for all $n \in \NN$.  Note that our convention is that $\NN = \{ 0, 1, \ldots \}$.

\section{Establishing notation}
\label{NOTATION} For the remainder of the paper, let $\kk$ be an algebraically closed, uncountable base field. Let $\FF$ be the prime subfield of $\kk$. Let $T:= \PP^1 \times \PP^1$, with coordinates $[x:y][z:w]$.  Let $u := x/y$ and let $v := z/w$, and let $K := \kk(T) = \kk(u,v)$. We define a birational self-map
\begin{align*}
\sigma: T & \dra T \\
[x:y][z:w] & \mapsto [xz:yw][z:w].
\end{align*}
We note the pullback action of $\sigma$ on rational functions:  we have
\begin{equation}\label{eq-sigma}
 u^{\sigma} = uv, \ \ \ v^{\sigma} = v, \ \ \
u^{\sigma^{-1}} = uv^{-1}, \ \ \ v^{\sigma^{-1}} = v.
\end{equation}

We establish some notation for subvarieties of $T$:  let
\begin{gather*}
X := \VV(x) = [0:1] \times \PP^1, \ \ \ \ Y := \VV(y) = [1:0] \times \PP^1, \\
Z := \VV(z) = \PP^1 \times [0:1], \ \ \ \ W := \VV(w) = \PP^1 \times [1:0].
\end{gather*}
Note that $u$ and $v$ give coordinates on $T \smallsetminus (Y \cup W) \cong \mathbb{A}^2$. We also fix names for the four intersection points:
\begin{gather*}
P := Z \cap X = [0:1][0:1], \ \ \ \ Q := Z \cap Y = [1:0][0:1], \\
F := W \cap X = [0:1][1:0], \ \ \ \ G := W \cap Y = [1:0][1:0].
\end{gather*}
Then the fundamental points of $\sigma$ are $Q$ and $F$; further, $\sigma(Z\smallsetminus Q) = P$ and $\sigma(W \smallsetminus F) = G$. The inverse map $\sigma^{-1}$ is given by the formula $[x:y][z:w] \mapsto [xw:yz][z:w]$ and the fundamental points of $\sigma^{-1}$ are $P$ and $G$, while $\sigma^{-1}(Z\smallsetminus G) = Q$ and $\sigma^{-1}(W\smallsetminus P) = F$.

It will be helpful to factor $\sigma$ as a composition of monoidal transformations and their inverses \cite[Theorem~V.5.5]{Ha}.  Let $\alpha: \Twiddle \to T$ be the blowup of $T$ at $Q$ and $F$.  We denote the six $(-1)$ curves on $\Twiddle$ by $L_X, L_F, L_W, L_Y, L_Q, L_Z$; these are arranged in a hexagon in this order.    The morphism $\alpha$ contracts the divisors $L_Q$ and $L_F$ to the points $Q$ and $F$, and maps $(L_X, L_W, L_Y, L_Z)$ to $(X, W, Y, Z)$.    There is also a morphism $\beta: \Twiddle \to T$ so that the diagram
\[ \xymatrix{
& \Twiddle \ar[ld]_{\alpha} \ar[rd]^{\beta} & \\
T \ar@{-->}[rr]_{\sigma} && T }\] commutes.  The morphism $\beta$ contracts $L_Z$ to $P$ and $L_W$ to $G$.  It maps $(L_X, L_F, L_Y, L_Q)$ to $(X, W, Y, Z)$. See Figure~1.
\begin{figure}
\includegraphics[scale=.75]{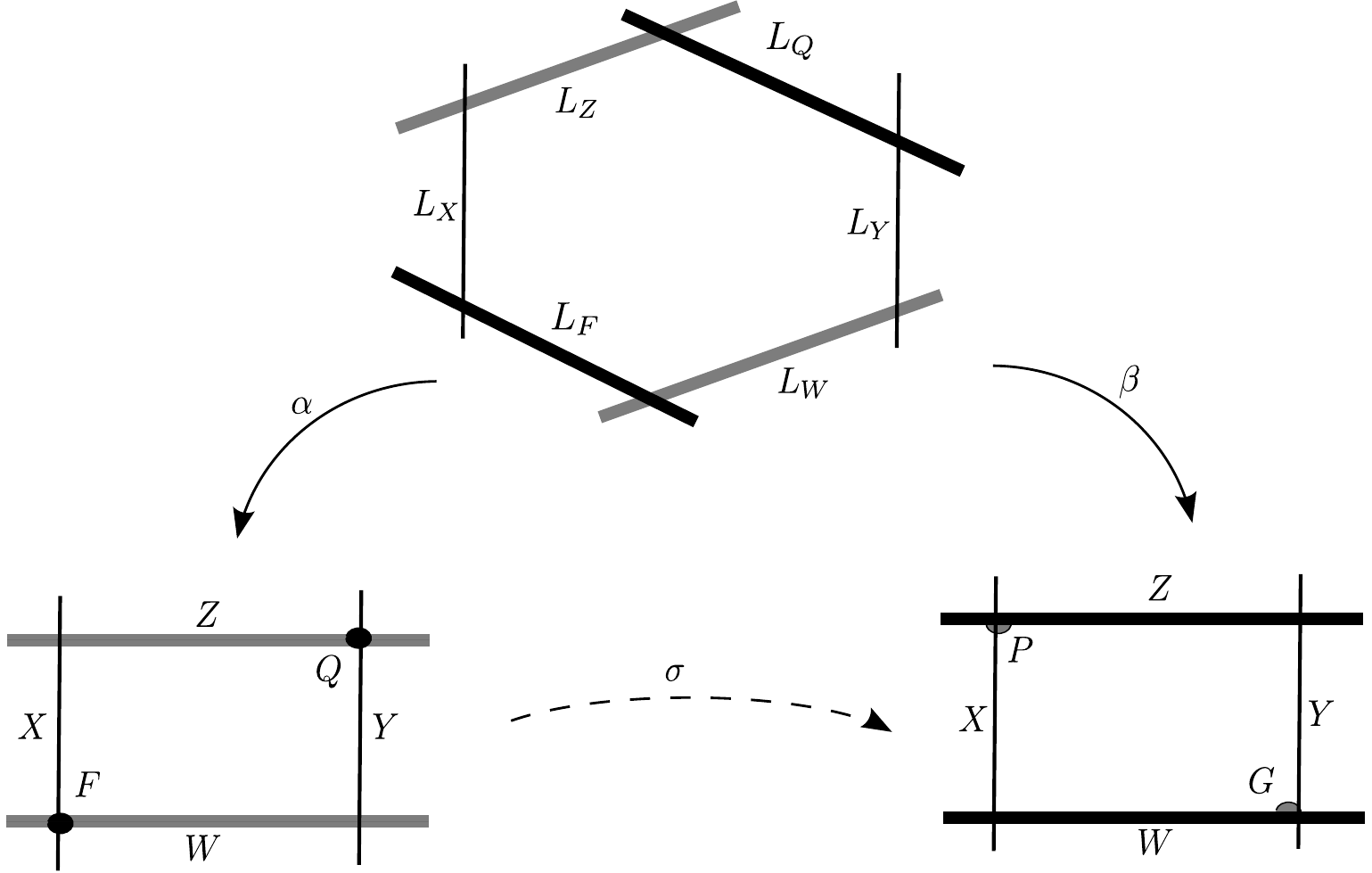}
\caption{The birational self-map $\sigma$.}
\end{figure}

We will want to understand the how divisors on $T$ pull back along $\sigma = \beta\alpha^{-1}$.
Since $\alpha$ and $\beta$ have no fundamental points, by Lemma~\ref{comp-lem} we have:
\begin{gather}
\sigma^* X  = \alpha_* \beta^* X = \alpha_* (L_X+L_Z) = X+Z, \ \ \ \
\sigma^*Y = \alpha_* \beta^* Y = \alpha_* (L_Y+L_W)  = Y+W, \\
\sigma^* Z = \alpha_* (L_Z+L_Q) = Z, \ \ \ \  \sigma^* W = \alpha_* (L_W+L_F) =  W.
\end{gather}
In particular,  this calculation shows that in the Picard group, the pullback map is  $\mc{O}(a,b)^{\sigma} \cong \mc{O}(a, a + b)$.

We are interested in deforming $\sigma$ by composing it with a automorphism of $T$.  For any $\tau \in \Aut^o(T) = \PGL_2 \times \PGL_2$, we define $\phi := \tau \circ \sigma : T \dra T$.  Since $\tau$ is an automorphism acting as the identity on the Picard group of $T$, note that we also have
\begin{equation}
\label{phi-action-eq} \mc{O}(a,b)^{\phi} \cong \mc{O}(a, a + b).
\end{equation}

In fact, it will be convenient to work only with $\tau$ coming from a more restricted (2-dimensional) algebraic subgroup of automorphisms. Let
\[ \mb{G}: = \{ \tau \in \Aut(T) \st \mbox{ $\tau$ fixes each of the points $[\pm 1:1][\pm1:1]$} \} \subseteq \Aut^o(T).\]
It is easy to see that $\mb{G} \cong \kk^* \times \kk^*$ as algebraic groups (an explicit isomorphism is given later in this section), and we
 write $\tau(\rho, \theta)$ for the element of $\mb{G}$ which corresponds to $(\rho, \theta) \in \kk^* \times \kk^*$.

We now define the rings that are the main subject of the paper. Let
\[ \Lsh := \struct_T(Y+W) \cong \struct(1,1) \ \ \text{and}\ \  E := \kk+ \kk u + \kk v + \kk uv \subseteq K = \kk(u,v). \]
We canonically identify $E$ with $H^0(T,\Lsh)$.
Now for given $\tau \in \mb{G}$, we denote the pullback action of $\phi = \tau \sigma: T \dra T$ on rational functions by $\varphi$, so $\varphi(f) = f^\phi$, and we form the skew-Laurent ring $K[t, t^{-1}; \varphi]$.   Let
\[ R:= R(\tau) : = \kk \ang{Et} \subseteq K[t, t^{-1}; \varphi] \]
be the $\kk$-subalgebra generated in degree $1$ by $Et$.

Let us say that $\tau$ is a {\em general} element of  $\mb{G}$ if $\tau $ lies in the complement of a countable union of proper closed subvarieties of $\mb{G}$.  The goal of this paper is to understand the algebras $R(\tau) = R$, at least for general $\tau$. We recall  here the definitions of some homological properties of interest.  The connected graded $\kk$-algebra $R$ is {\em Artin-Schelter (AS) Gorenstein} if $R$ has finite left and right injective dimension $d$ and we have
\[ \Ext^i_R(\kk, R) = \begin{cases}
                     0 & \mbox{if $i \neq d$} \\
        \kk[\ell] & \mbox{if $i = d$}
                    \end{cases} \]
(for some shift $\ell$), where $\kk = R/R_{\geq 1}$ is the trivial graded module.  If in addition $R$ has finite global dimension $d$ and finite GK-dimension, then $R$ is {\em AS-regular}.  The algebra $R$  satisfies (right) $\chi_i$ if $\dim_\kk \Ext_R^j(\kk, M) < \infty$ for all $j \leq i$ and for all finitely generated $M_R$, and $R$ satisfies $\chi$ if it satisfies $\chi_i$ for all $i \geq 0$.

Recall \cite{Jorgensen1998} that the {\em depth} of a graded $R$-module $M$ is defined to be $\min \{ i \st \Ext^i_R(\kk, M) \neq 0 \}$. For commutative noetherian graded rings  and for some classes of noncommutative graded rings, including AS-Gorenstein algebras \cite[Theorem~3.2]{Jorgensen1998},  the {\em Auslander-Buchsbaum formula} holds:  that is, for a graded module $M$ of finite projective dimension
we have $\pd M + \depth M = \depth R$.   Recall also that if $R = \bigoplus_{n \in \NN} R_n$ is a connected graded algebra, then the {\em Hilbert series} of $R$ is the formal power series $h_R(s) = \bigoplus_{n \in \NN} (\dim_{\kk} R_n) s^n$.

The main goal of the remainder of the paper will be to prove the following result:
\begin{theorem}\label{thm-main}
Let the pair $(\rho, \theta)$ be algebraically independent over the prime subfield $\FF$ of $\kk$, and let $R = R(\tau(\rho, \theta))$.
\begin{enumerate}
\item $R$ is  noetherian. \item $R$ has a presentation $k\langle x_1, x_2, x_3, x_4 \rangle/(f_1, \ldots, f_6)$, where
\begin{gather*}
f_1 =  x_1(c x_1 -  x_3) + x_3( x_1 - c x_3), \ \
f_2 =  x_1(c x_2 -  x_4) + x_3 ( x_2 - c x_4), \\
f_3 = x_2(c x_1 -  x_3) + x_4( x_1 - c x_3), \ \
f_4  = x_2(c x_2 -  x_4) + x_4( x_2 - c x_4), \\
f_5 =  x_1(d x_1 -  x_2) + x_4( x_1 - d x_2), \ \ f_6 = x_1(d x_3 -  x_4) + x_4( x_3 - d x_4),
\end{gather*}
for $c = \frac{\theta-1}{\theta+1}, d= \frac{\rho-1}{\rho+1}$. \item The trivial module $\kk$ has a free resolution of the form
\[
0 \to R[-4] \to R[-3]^4 \to R[-2]^6 \to R[-1]^4 \to R \to \kk \to 0.
\]
\item $R$ is Koszul of global dimension 4 and has Hilbert series $h_R(s) = 1/(1-s)^4$.  However, $R$ is not AS-Gorenstein and
    therefore
not AS-regular, and the Auslander-Buchsbaum equality fails for $R$.
\end{enumerate}
\end{theorem}

We use the notation $\mathbb 1 := \tau(1, 1)$ for the identity map of $\mb{G}$, and
we give the special case of $R(\tau)$ where $\tau = \mathbb 1$
its own name:
\[
A: = R(\mathbb 1) = \kk \langle t, ut, vt, uvt \rangle \subseteq K[t, t^{-1}; \sigma].
\]
(Note that here we denote $\sigma:  T\dra T$ and the induced pullback action on $K$ by the same symbol; we will do likewise for $\tau$. We hope this will not induce confusion.) The ring $A$ has appeared in the literature before (see \cite[Proposition 7.6]{YZ2006} ) and it has certain bad properties.  Most notably, it is not noetherian on either side (we sketch the simple proof in  \ref{rem-not-noeth} below).   However, several other properties we desire to prove for $R$, such as finite global dimension, hold for $A$ and the main strategy of our proofs in these cases is to show that these properties deform to hold also for $R(\tau)$ for  general $\tau$.

In the remainder of this section, we give some further formulas and subsidiary results which will be useful in the sequel.  Although the coordinate system we have been using is the one in which $\sigma$ is simplest, occasionally we will want to change coordinates so that the automorphisms $\tau \in \mb{G}$ are diagonalized.   We use round brackets $( \ : \ )$ for the coordinate system of $\mb{P}^1$ which is related to the original one by the change of coordinate formulas $[x:y] = (x-y:x+y)$ and $(a:b)= [a+b:-a+b]$.  Then $[1:1] = (0:1)$ and $[1:-1] = (1:0)$, and the group of automorphisms of $\mb{P}^1$ fixing both of these points is isomorphic to $\kk^*$, where we let $\rho \in \kk^*$ correspond to the diagonal automorphism $(a:b) \mapsto (\rho^{-1} a : b)$.   An automorphism in $\mb{G}$ has the form $\mu \times \nu$ where $\mu, \nu \in \aut (\mb{P}^1)$ correspond to elements $\rho, \theta \in \kk^*$ respectively; this makes explicit the isomorphism $\tau:  \kk^* \times \kk^* \to \mb{G}$ already mentioned.

We define the useful abbreviations
\[
\gamma := \rho + 1, \ \ \ \delta := \rho -1, \ \ \ \epsilon := \theta + 1, \ \ \ \zeta := \theta -1.
\]
Then in our two coordinate systems we have the following formulas for $\tau(\rho, \theta): T \to T$:
\begin{align*}
(a:b)(c:d)  & \mapsto (\rho^{-1}a:b)(\theta^{-1}c:d), \\
[x:y][z:w] & \mapsto [\gamma x + \delta y: \gamma y + \delta x][\epsilon z + \zeta w: \epsilon w + \zeta z].
\end{align*}
In terms of the action on rational functions, we record the formulas
\begin{equation}
\label{eq-tau} u^\tau = \frac{\gamma u + \delta}{\delta u + \gamma}, \ \ \ v^\tau = \frac{\epsilon v + \zeta}{\zeta v + \epsilon}, \ \ \ u^{\tau^{-1}} = \frac{\gamma u - \delta}{ - \delta u + \gamma}, \ \ \ v^{\tau^{-1}} = \frac{\epsilon v - \zeta}{-\zeta v + \epsilon}.
\end{equation}

For future reference we also record the following formulas for the action of $\phi$ and its inverse on rational functions.
Recalling that we write $\varphi(f) = f^\phi$, we have:
\begin{equation}
\label{eq-phi} \varphi(u) = \frac{\gamma uv + \delta}{\delta uv + \gamma}, \ \ \ \varphi(v) = \frac{\epsilon v + \zeta}{\zeta v + \epsilon}, \ \ \ \varphi^{-1}(u) = \frac{(\gamma u - \delta)}{( - \delta u + \gamma)}\frac{(-\zeta v + \epsilon)}{(\epsilon v - \zeta)}, \ \ \ \varphi^{-1}(v) = \frac{\epsilon v - \zeta}{-\zeta v + \epsilon}.
\end{equation}

We note here the following symmetry in our main setup.
\begin{lemma}
\label{sym-lem} Let $\psi := \tau(-1, -1)$, so $\psi([x:y][z:w]) = [y:x][w:z]$ is the automorphism of $T$ that switches the coordinates in both copies of $\mb{P}^1$.  Then $\psi$ commutes with $\phi = \tau \sigma$ for all $\tau \in \GG$.  Also, $\psi$ interchanges $X$ and $Y$ and interchanges the points $F$ and $Q$. \qed
\end{lemma}
We close this section with an analysis of the opposite ring of $R(\tau)$.  We show that the opposite ring lives in the same family of examples, and so we will be able to focus our attention on right ideals and right modules.

\begin{proposition}\label{prop-2}
Let $\tau \in \mb{G}$ and let $R = R(\tau)$. Then $R^{op} \cong R(\tau^{-1})$.
\end{proposition}
\begin{proof}
We define an automorphism $\omega: K \to K$ by $\omega(u) =  u, \omega( v) =  v^{-1}$. In other words, $\omega $ is pullback by $ \tau(1,-1) \in \mb{G}.$
  Note that $\omega^2 = \mathbb{1}$.

All compositions in this proof are compositions of automorphisms of $K$.
Now, $R$ is the subalgebra of $K[t; \varphi]$ generated by $H^0(T, \struct(1,1))\cdot t$. The map $K[t; \varphi] \to K[t; \varphi^{-1}]$ given by $\sum a_i t^i \mapsto \sum a_i^{\phi^{-i+1}} t^i$ is an anti-isomorphism, so the subalgebra of $K[t; \varphi^{-1}]$ generated by $H^0(T, \struct(1,1)) \cdot t$ is isomorphic to $R^{op}$. Next, the map $\sum a_i t^i \mapsto \sum \omega(a_i^\tau) t^i$ gives an isomorphism $K[t; \varphi^{-1}] \to K[t; \omega \tau \varphi^{-1} \tau^{-1} \omega^{-1}]$, so $R^{op}$ is isomorphic to the subalgebra of $K[t; \omega \tau \varphi^{-1} \tau^{-1} \omega^{-1}]$ generated by $\omega \tau \bigl(H^0(T, \struct(1,1))\bigr)\cdot t$.  Since $\omega \tau: K \to K$ is induced by an automorphism of $T = \mb{P}^1 \times \mb{P}^1$ which fixes linear equivalence classes of divisors, we have $\omega \tau\bigl(H^0(T, \struct(1,1))\bigr) = a H^0(T, \struct(1,1))$ for some $0 \neq a \in K$.
Finally, putting $t' = at$, we have $K[t;  \omega \tau \varphi^{-1} \tau^{-1} \omega^{-1}] = K[t'; \omega \tau \varphi^{-1} \tau^{-1} \omega^{-1}]$, and now $R^{op}$ is isomorphic to the subalgebra of $K[t'; \omega \tau \varphi^{-1} \tau^{-1} \omega^{-1}]$ generated by $H^0(T, \struct(1,1)) \cdot t'$.

We consider the automorphism $\omega\tau\varphi^{-1} \tau^{-1} \omega^{-1}$ of $K$.  Since  $\sigma$ and $\tau$ act via pullback on $K$, we have $\phi^* = \varphi = \sigma\tau \in \Aut(K).$ Thus $\omega \tau \varphi^{-1} \tau^{-1} \omega^{-1} = \omega \sigma^{-1} \tau^{-1} \omega^{-1} = \omega \sigma^{-1} \omega \tau^{-1}$, as $\GG$ is abelian and $\omega^2 = \mathbb{1}$.  A trivial computation shows that $\omega \sigma^{-1} \omega = \sigma$, and so we conclude that $\omega \tau \varphi^{-1} \tau^{-1} \omega^{-1} = \sigma \tau^{-1}$.

It follows that $R(\tau)^{op} \cong R(\tau^{-1})$, as claimed.
\end{proof}

\begin{corollary}\label{cor-Aop}
The algebras $A$ and $A^{op}$ are isomorphic.  \hfill $\Box$
\end{corollary}

\section{Geometrizing $R$}
\label{RandR}

We want to describe the rings $R(\tau)$ defined in the previous section in a way analogous to the construction of a twisted homogeneous coordinate ring or a na\"ive blowup.  In particular, we would like to show that each graded piece of $R$ can be identified with the global sections of a certain sheaf on $T$. This will require using pullback by a birational map instead of an automorphism, and so some sensitive calculations will be needed.

All of the notation developed in the previous section will be in force in this section. Fix $\tau \in \GG$.  Recall that we defined $\phi = \tau \sigma$ as a birational self-map of $T$, and we put $E := \HB^0(T, \mc{L}) \subseteq K = \kk(T)$ where $\mc{L} = \mc{O}_T(W + Y) \subseteq \mc{K}$. For all $n, m \in \NN$ we define an invertible sheaf
\[
\mc{L}^{\phi^m}_n := \mc{L}^{\phi^m} \otimes \mc{L}^{\phi^{m+1}} \otimes  \dots \otimes \mc{L}^{\phi^{m+n-1}} \subseteq \mc{K}.
\]
In particular, when $m = 0$ this defines the invertible sheaf $\mc{L}_n := \mc{L} \otimes \mc{L}^{\phi} \otimes \dots \otimes \mc{L}^{\phi^{n-1}}$.  We warn the reader, however, that $\mc{L}^{\phi^m}_n$ is meant to be one notational unit, and is not necessarily equal to the pullback of $\mc{L}_n$ by the birational map $\phi^m$; though if $\phi$ is a stable birational map then this is true.  (In all of our later applications, we will in fact choose $\tau$ so that $\phi$ is stable.)
It is clear that for all $n \geq 0$, $E^{\phi^n} \subseteq \HB^0(T, \mc{L}^{\phi^n}) \subseteq K$, and moreover that $E^{\phi^n}$ generates $\mc{L}^{\phi^n}$ except possibly at the fundamental points of $\phi^n$.  Since we defined $R = R(\tau) = \kk\langle Et \rangle \subseteq K[t, t^{-1}; \varphi]$, we have
\[
R_n = (Et)^n = E E^{\phi} \cdots E^{\phi^{n-1}}t^n \subseteq \HB^0(T, \mc{L}_n)t^n \subseteq K t^n
\]
for all $n \geq 0$.

For all $m, n \in \NN$, we define $\sh{R}^{\phi^m}_n$ to be the sheaf generated by
\[(R_n t^{-n})^{\phi^m} =
E^{\phi^m}E^{\phi^{m+1}} \cdots E^{\phi^{m+n-1}} \subseteq H^0(T, \sL_n^{\phi^m}).\]
  Again,
$\sh{R}^{\phi^m}_n$ is not to be construed as the pullback of $\sh{R}_n$ by $\phi^m$.
We use special notation for the case $A = R(\mathbb 1)$ where $\tau$ is the identity map. In this case we write $\mc{A}_n^{\sigma^m}$ instead of  $\mc{R}_n^{\phi^m}$.

Let $\sh{I}_n^{\phi^m}$ be the base ideal of the  sections in $(R_n t^{-n})^{\phi^m}$, so that $\sh{R}^{\phi^m}_n = \sh{I}_n^{\phi^m} \Lsh^{\phi^m}_n \subseteq \mc{K}$. Let $B_n^m = B_n^m(\tau)$ be the subscheme of $T$ defined by $\sh{I}_n^{\phi^m}$. Now,  $B_1^m$ is a $0$-dimensional subscheme, supported at the fundamental points of $\phi^m$.  So $B_n^m$ is $0$-dimensional for all $m,n$ also.
One of the main goals of the next few sections will be to show that for general choice of $\tau$, then $(R_nt^{-n})^{\phi^m}$  is precisely equal to the global sections of the sheaf $\mc{R}^{\phi^m}_n$.  In particular, taking $m = 0$ this will show that indeed the graded pieces of the ring $R$ have a geometric description as the global sections of certain sheaves.   To move towards this goal, we will need to study how the sheaves $\mc{L}_n$ and the ideal sheaves $\mc{I}_n$
depend on the choice of $\tau$.

It is useful in this section to write $\tau(\rho, \theta) = \mu \times \nu$, where $\mu = \mu(\rho)$ and $\nu = \nu(\theta)$ are automorphisms of $\mb{P}^1$.
In the next several results, we concentrate on gaining an understanding of $\sh{I}_1^{\phi^m}$ for $m \geq 1$. Recall that this was defined as the base ideal of the sections
\[ \kk + \kk u^{\phi^m} + \kk v^{\phi^m} + \kk (uv)^{\phi^m} = E^{\phi^m} \subseteq H^0(T, \mc{L}^{\phi^m}).\]
To compute this, we first compute the Weil divisors associated to these rational functions.  We obtain:
\begin{gather*} (1) = 0, \  (u^{\phi^m}) = X^{\phi^m} - Y^{\phi^m}, \
(v^{\phi^m}) = Z^{\phi^m} - W^{\phi^m}, \  \text{and}\
 ((uv)^{\phi^m}) =
X^{\phi^m} + Z^{\phi^m} - Y^{\phi^m} - W^{\phi^m}.
\end{gather*}
Thus we seek the intersection of the four effective Weil divisors:
\[ (Y^{\phi^m} + W^{\phi^m}) \cap  (Y^{\phi^m} + Z^{\phi^m}) \cap (X^{\phi^m} + W^{\phi^m}) \cap (X^{\phi^m} + Z^{\phi^m}).\]
Notice that a general horizontal line $D = \mb{P}^1 \times [z :w]$ will not contain any fundamental points of $\phi^{-m}$, and thus $D^{\phi^m} = \mb{P}^1 \times \nu^{-m}([z:w])$ since $\phi$ simply acts as $\nu$ in the second copy of $\mb{P}^1$.   Now for any horizontal line $D$ whatsoever, since pullback respects linear equivalence, $D^{\phi^m}$ will be another $(0,1)$-curve; this forces $D^{\phi^m} = \mb{P}^1 \times \nu^{-m}([z,w])$ in all cases.  In particular, $W^{\phi^m}$ and $Z^{\phi^m}$ are distinct horizontal lines and $W^{\phi^m} \cap Z^{\phi^m} = \emptyset$.  So the base locus $B_1^m$
is equal to the scheme-theoretic intersection $Y^{\phi^m} \cap X^{\phi^m}$.

The more specific calculation of $B_1^m$ is straightforward in case $\nu$ has infinite order, but messier otherwise. For given $m$, we want to understand $B_1^m$ at least for an open set of $\tau$ that includes the identity map $\mathbb 1$.  So we will have to carefully analyze the $\tau$ with $\nu$ the identity, but will only need to analyze the case of other finite order $\nu$ in passing.

We denote the order of the automorphism $\nu$ by $o(\nu) \in \{1, 2, \ldots \} \cup \{ \infty \}$.
The $\nu$-orbit of any point in $\PP^1 \smallsetminus \{[\pm 1:1] \}$ has size $o(\nu)$.
\begin{lemma}
\label{easy-orbit-lem} Let $\tau = \tau(\rho, \theta) = \mu \times \nu \in \mb{G}$.  If $o(\nu) = 2k$ for some $k \geq 1$, then $\nu^k[0:1] = [1:0]$ and $\nu^k[1:0] = [0:1]$.  If $o(\nu) = \infty$ or $o(\nu)$ is odd, then the $\nu$-orbits of $[0:1]$ and $[1:0]$ are disjoint.
\end{lemma}
\begin{proof}
In the other coordinate system $( \ : \ )$ for $\mb{P}^1$ introduced in Section~\ref{NOTATION},  $\nu$ is the diagonal map $(a:b) \mapsto (\theta^{-1} a : b)$.  In these coordinates, we have $[0:1] = (-1 : 1)$ and $[1:0] = (1: 1)$.  It follows that if $o(\theta) = 2k$ for some $k \geq 1$, then $\theta^k = -1$ and the first case occurs.  Otherwise, clearly $[0:1]$ and $[1:0]$ do not lie on the same orbit, so  the second case occurs.
\end{proof}

\begin{definition}
Let $j \geq 1$. Let $V(j) \subseteq \mb{G}$ be defined by
\[ V(j) := \{ \tau(\rho,\theta) \st o(\theta) > 2j \} \cup \{ \tau(\rho, 1) \st
\rho \neq -1 \}.  \]
\end{definition}

We note that $V(j)$ (or, more properly, its complement) is defined over $\FF$, the prime subfield of $\kk$.

\begin{proposition}
\label{prop-fund-pts} Fix $j \geq 1$ and let $\tau = \tau(\rho, \theta) \in V(j)$.  Then
\begin{enumerate}
\item  For all $0 \leq i \leq j$, $\{ \mbox{fundamental points of } \phi^{-i}\} \cap \{ \mbox{fundamental points of } \phi \} = \emptyset$.
Thus $\phi^{-i}$ is defined at both $F$ and $Q$. \item For all $0 \leq m \leq j+1$, the divisors $X^{\phi^m}$ and $Y^{\phi^m}$ are $(1,m)$-curves, and the scheme $B_1^m$ has length $2m$. Also, $\mc{L}^{\phi^m} \cong \mc{O}(1, m+1)$. \item For all $1 \leq n \leq j+2$ and $0 \leq m \leq j+2-n$, we have $\mc{L}_n^{\phi^m} \cong \mc{O}(n,k)$, where $k =
    \binom{n+m+1}{2}
- \binom{m+1}{2}$. \item $\{F, Q \} \subseteq X^{\phi^m} \cap Y^{\phi^m}$ for all $1 \leq m \leq j$.  In particular, $B_1^1 = \{F, Q \}$ with reduced subscheme structure.
\end{enumerate}
\end{proposition}
\begin{proof}
(1). The fundamental points of $\phi = \tau \sigma$ are precisely the fundamental points $F=[0:1][1:0]$ and $Q = [1:0][0:1]$ of $\sigma$, so we only have to prove the first statement.  Similarly, the fundamental points of $\phi^{-1} = \sigma^{-1} \tau^{-1}$ are $\tau(P)$ and $\tau(G)$, since $P = [0:1][0:1]$ and $G = [1:0][1:0]$ are the fundamental points of $\sigma^{-1}$.   Since $\sigma^{-1}$ leaves the second copy of $\mb{P}^1$ alone, an inductive argument shows that the fundamental points of $\phi^{-i}$ are contained in $\PP^1 \times \{\nu[0:1], \ldots, \nu^{i}[0:1], \nu[1:0], \ldots, \nu^{i}[1:0]\}$. It is clear from Lemma~\ref{easy-orbit-lem} that if $o(\nu) > 2j \geq 2i$, then the set above is disjoint from $\{F, Q\}$.

Thus we have only to consider the case $\theta = 1$, $\rho \neq -1$, and $i \geq 1$.  We claim that in this case the set of fundamental points of $\phi^{-i}$ is precisely $\{ \tau(P), \tau(G)\}$.  Now, $\phi$ contracts  only the two curves $Z$ and $W$, which it contracts to the points $\tau(P)$, $\tau(G)$ respectively;  moreover, $\phi$ is defined at $\tau(P), \tau(G)$ since the condition $\theta = 1, \rho \neq -1$ forces $\{\tau(P), \tau(G) \} \cap \{F, Q\} = \emptyset$.  Since $\theta = 1$, we have $\tau(P) \in Z$, $\tau(G) \in W$.  Thus necessarily $\phi(\tau(P))  = \tau(P)$, $\phi(\tau(G)) = \tau(G)$.  One may check that $Z$ and $W$ are not the images of any curves under $\phi$, so an inductive argument shows that $Z$ and $W$ are also the only curves that $\phi^i$ contracts.  Then $\{ \tau(P), \tau(G) \}$ are the only images of curves contracted by $\phi^i$, and thus these are the only points which are fundamental points for $\phi^{-i}$, proving the claim.  These fundamental points $\{ \tau(P), \tau(G)\}$ of $\phi^{-i}$ are disjoint from the fundamental points $\{F, Q \}$ of $\phi$, as we have already noted.

(2). By part (1), for any $0 \leq i \leq j$, $\phi^{-i}$ and $\phi$ have no common fundamental points, and thus $ \phi^* (\phi^i)^* = (\phi^{i+1})^*$ by Lemma~\ref{comp-lem}. By induction we see that $(\phi^*)^m = (\phi^m)^*$ for all $0 \leq m \leq j + 1$. In particular, since we calculated earlier in \eqref{phi-action-eq} that $\phi^*$ acts on $\Pic T \cong \mb{Z}^2$ by $\mc{O}(a,b)^{\phi} = \mc{O}(a,a+b)$, we see that $\mc{O}(1,0)^{\phi^m} \cong \mc{O}(1,m)$.  Since $X$ and $Y$ are $(1,0)$-curves, $X^{\phi^m}$ and $Y^{\phi^m}$ are $(1,m)$-curves, for all $m \leq j + 1$.  The length of the $0$-dimensional scheme $B_1^m = X^{\phi^m} \cap Y^{\phi^m}$ is equal to the intersection product $(X^{\phi^m}. Y^{\phi^m}) = (1,m).(1, m) = 2m$ \cite[Proposition V.1.4]{Ha}.  Finally, since $\mc{L} \cong \mc{O}(1,1)$, we also get $\mc{L}^{\phi^m} \cong \mc{O}(1,m+1)$ for all such $m$.

(3). Since by definition  $\mc{L}_n^{\phi^m} = \mc{L}^{\phi^m} \otimes \mc{L}^{\phi^{m+1}} \otimes \dots \otimes \mc{L}^{\phi^{m+n-1}}$, and $\mc{L}^{\phi^i} \cong \mc{O}(1, i+1)$ for all $0 \leq i \leq j+1$ by part (2), it follows as long as $m + n - 1 \leq j + 1$ that $\mc{L}_n^{\phi^m} \cong \mc{O}(n,k)$ where $k = \sum_{i =m}^{m+n-1}( i+1) = \binom{n+m+1}{2} - \binom{m+1}{2}$.

(4). Let $C$ be an $(a,b)$-divisor on $T$ with $a > 0, b \geq 0$.  Then $(\tau^{-1}(C).Z) = (\tau^{-1}(C).W) = a > 0$, so $\tau^{-1}(C)$ must meet both $Z$ and $W$ nontrivially.  Recall the decomposition $\sigma = \beta \alpha^{-1}$ and its associated notation from Section~\ref{NOTATION}.  Then $\beta^*(\tau^{-1}(C)) = \wt{C} + E$ where $\wt{C}$ is the proper transform of $\tau^{-1}(C)$ and $E$ is a sum of exceptional curves for $\beta$.  If $\tau^{-1}(C) \cap Z \neq \{ P \}$, then $\wt{C}$ must meet (but not be equal to) the proper transform $L_Q$ of $Z$.   If instead $\tau^{-1}(C) \cap Z = \{ P \}$, then the exceptional curve $L_Z$ lying over $P$ must appear in $E$.   In either case, we see that $Q$ lies on some curve appearing in $C^{\phi} = \alpha_*(\wt{C}) +  \alpha_*(E)$. An analogous argument considering intersections with $W$ shows that $F \in C^{\phi}$ as well.   Thus $\{F, Q \}$ is contained in the support of $C^{\phi}$.

This argument shows, in particular, that $\{ F, Q \} \subseteq B_1^m = X^{\phi^m} \cap Y^{\phi^m}$ for all $m \geq 1$.  When $m = 1$ then $B_1^1$ has length $2$ also by part (2), and so $B_1^1$ has no choice but to be the reduced subscheme supported at $\{F, Q \}$.
\end{proof}

\begin{corollary}
\label{cor-stable} Let $\tau  \in \bigcap_{j \geq 1} V(j)$.  Then $\phi = \tau \sigma$ is a stable birational map.
\end{corollary}
\begin{proof}
This is immediate from Definition~\ref{stable-def} and part (1) of the proposition.
\end{proof}

\begin{remark}
It is easy to find examples of $\tau \in \GG$ for which $\phi$ is not stable.  The simplest of these is $\tau = \tau(-1,1)$.  For this $\tau$ one has $\tau(u) = u^{-1}, \tau(v) = v$ and so $u^\phi = u^{-1}v, v^\phi = v$. Then $\phi^2 = \mb 1$, but $\phi^*$ acts on the Picard group as calculated in \eqref{phi-action-eq}; in particular, $(\phi^*)^2$ is certainly not the identity. Correspondingly, one can see that the ring $R = R(\tau(-1,1))$ behaves very differently from the case of a general $\tau$.  In fact, since $\phi^2 = 1$, the $2$-Veronese $R^{(2)}$ is a commutative ring with graded quotient ring $\kk(u,v)[t^2]$, and so this $R$ is a PI ring with $\GK R = 3$.  We have not attempted to fully characterize which $\tau$ lead to a stable $\phi = \tau \sigma$.
\end{remark}

Let $\tau = \tau(\rho, \theta) \in V(j)$ for some $j \geq 1$. By definition, the automorphisms in $V(j)$ come in two types, those with $o(\theta) > 2j$, and those with $\theta = 1$.   We now analyze the structure of the schemes $B_1^m$ for $m \leq j$ in each case separately. We begin with the easier case where $o(\theta)
> 2j$.
\begin{definition}
\label{def-fq} Let $\tau = \tau(\rho, \theta) \in V(j)$ with $o(\theta) > 2j$ for some $j$. Proposition~\ref{prop-fund-pts}(1) shows that $\phi^{-i}$ is defined at $F$ and $Q$ for $1 \leq i \leq j$.  We set $F_i := \phi^{-i}(F)$ and $Q_i := \phi^{-i}(Q)$, for $0 \leq i \leq j$.
\end{definition}

\begin{lemma}
\label{lem-pointchains1} Let $j \geq 1$ and let $\tau = \tau(\rho, \theta) \in V(j)$ with $o(\theta) > 2j$. For each $0 \leq m \leq j+1$, $B_1^m$ is the reduced scheme consisting of the $2m$ distinct points $\{F_0, \ldots, F_{m-1}, Q_0, \ldots, Q_{m-1}\}$.
\end{lemma}
\begin{proof}
We claim first that $\phi^{-1}$ is defined and a local isomorphism at each $F_i$ and $Q_i$ with $0 \leq i \leq j$.  To see this, note that $\phi^{-1} = \sigma^{-1} \tau^{-1}$ is defined and a local isomorphism at any point which does not lie on $\tau(Z) = \mb{P}^1 \times \nu([0:1])$ or $\tau(W) = \mb{P}^1 \times \nu([1:0])$.  Since $F_i \in \mb{P}^1 \times \nu^{-i}([1:0])$ and $Q_i \in \mb{P}^1 \times \nu^{-i}([0:1])$ for $1 \leq i \leq j$, the claim follows, using the hypothesis that $o(\nu) > 2j \geq 2i$ and Lemma~\ref{easy-orbit-lem}.  It is also clear from this calculation that the points $\{F_0, \ldots, F_j, Q_0, \ldots, Q_j \}$ are distinct.

Next, we prove by induction that $\{ F_0, Q_0, \ldots, F_{m-1}, Q_{m-1}\} \subseteq X^{\phi^m} \cap Y^{\phi^m}$ for all $1 \leq m \leq j+1$. The case $m=1$ is immediate from Proposition~\ref{prop-fund-pts}(4). Suppose we have proven that $\{ F_0, Q_0, \ldots, F_{m-1}, Q_{m-1}\} \subseteq X^{\phi^m} \cap Y^{\phi^m}$ for some $m < j+1$.  Since $\phi^{-1}$ is a local isomorphism at each of these points, it is clear that $\phi^{-1}$ of each of these points lies on $(X^{\phi^m})^{\phi} = X^{\phi^{m+1}}$ as well as  $(Y^{\phi^m})^{\phi} = Y^{\phi^{m+1}}$ (using stability of $\phi$).  So $\{ F_1, Q_1, \ldots, F_m, Q_m\} \subseteq X^{\phi^{m+1}} \cap Y^{\phi^{m+1}}$. However, $X^{\phi^{m+1}} $ and $ Y^{\phi^{m+1}}$ also both vanish at $F$ and $Q$, as we saw in Proposition~\ref{prop-fund-pts}(4).  So $X^{\phi^{m+1}} \cap Y^{\phi^{m+1}} = B_1^{m+1}$ is supported at least at the set $\{ F_0, Q_0, \ldots, F_m, Q_m\}$.

Now for any $m$ with $0 \leq m \leq j +1$, $B_1^m$ is a scheme of length $2m$ by Proposition~\ref{prop-fund-pts}(2), so it must be exactly the reduced scheme supported at the points  $\{ F_0, Q_0, \ldots, F_{m-1}, Q_{m-1}\}$.
\end{proof}

Now we calculate $B_1^m$ in the case where $\theta = 1$.  This is more complicated because the scheme is concentrated at two points.
\begin{lemma}
 \label{lem-pointchains2} Let $\tau = \tau(\rho,
1) \in \mb{G}$ with $\rho \neq -1$.
\begin{enumerate}
\item If $\rho = 1$ then $X^{\phi^n} = X + nZ$ and $Y^{\phi^n} = Y + nW$ for all $n \geq 0$, while if $\rho \not \in \{ 1, -1\}$ then
$X^{\phi^n}$ and $Y^{\phi^n}$ are irreducible curves for all $n \geq 0$.
\item For all $m \geq 1$, $B_1^m$ is supported at $\{F, Q \}$. Further, there are local coordinates $a,b$ at $F$ so that for $1 \leq n \leq m$, $B_1^n$ is defined locally at $F$ by $(a, b^n)$, and similarly for $Q$.
\end{enumerate}
\end{lemma}
\begin{proof}
(1).  If $\rho = 1$, then the calculation $X^{\phi} = X + Z$, $Z^{\phi} = Z$ was done in Section 1, and the result $X^{\phi^n} = X + nZ$ follows by induction, because $\phi$ is stable by Corollary~\ref{cor-stable}.  The calculation of $Y^{\phi^n}$ is similar.

Now assume that $\rho \not\in \{ 1, -1\}$. We prove that $X^{\phi^n}$ is irreducible by induction on $n$, the case $n = 0$ being immediate. Suppose that $C = X^{\phi^n}$ has been proven irreducible for some $n \geq 0$, and let us prove that $X^{\phi^{n+1}} = C^{\phi}$ (where we use that $\phi$ is stable) is irreducible. The curve $C^{\phi}$ will be irreducible (and equal to $\phi^{-1}(C)$) unless $C^{\tau} = \tau^{-1}(C)$ contains one of the points $\{P, G \}$ which are the images of the curves $Z, W$ which $\sigma$ contracts. We know that $(\tau^{-1}(C). W) = (\tau^{-1}(C).Z) = 1$, since $\deg \tau^{-1}(C) = \deg C = (1, n)$ by Proposition~\ref{prop-fund-pts}(2). But $\{F, P\} \subseteq X$ by definition, while $\{F, Q\} \subseteq C = X^{\phi^n}$ if $n \geq 1$ by Proposition~\ref{prop-fund-pts}(4). Thus $\tau^{-1}(F)$ must be the unique point in $\tau^{-1}(C) \cap W$ , while if $n \geq 1$ (respectively, if $n = 0$) then $\tau^{-1}(Q)$ (or $\tau^{-1}(P)$) is the unique point in $\tau^{-1}(C) \cap Z$. None of these points is equal to $P$ or $G$ since $\rho \not\in \{ 1, -1\}$, so $C^{\phi}$ is irreducible, completing the induction step. The proof that $Y^{\phi^n}$ is irreducible follows from the symmetry given by Lemma~\ref{sym-lem}.

(2). If $\rho = \theta = 1$, then $X^{\phi^n} = X + nZ$ and $Y^{\phi^n} = Y + nW$ for all $n \geq 0$, by part (1). As  $X$ and $W$ intersect at $F$ and $Y$ and $Z$ intersect at $Q$, the local structure of $X^{\phi^n} \cap Y^{\phi^n}$ is immediate.

Suppose then that $\theta=1$ and $\rho \not \in \{ -1, 1\}$. Note that the set $U:= T \smallsetminus (Z \cup W)$ is $\phi$-stable; in fact, $\phi|_U$ is an isomorphism.  We have $\{F, Q \} \subseteq B_1^m$ for all $m \geq 1$, and $B_1^1 = X^\phi \cap Y^\phi = \{F, Q \}$, by Proposition~\ref{prop-fund-pts}(4).   Thus $X^{\phi} \cap Y^{\phi} \cap U = \emptyset$, and by induction since $\phi \vert_U$ is an isomorphism we must have $X^{\phi^n} \cap Y^{\phi^n} \cap U = \emptyset$ for all $n \geq 1$. Thus $B_1^m$ is supported on $Z \cup W$.  Now, $(X^{\phi^m}.Z) = (X^{\phi^m}.W) = 1$, so set-theoretically $X^{\phi^m} \cap Z = \{Q\}$ and $X^{\phi^m} \cap W = \{F\}$.  By a similar argument, $Y^{\phi^m} \cap (Z \cup W) = \{ F, Q\}$.  Thus $B_1^m$ is supported on $\{F, Q\}$.  $B_1^m$ has length $2m$ by Proposition~\ref{prop-fund-pts}(2), and by the symmetry in Lemma~\ref{sym-lem},
necessarily $B_1^m$ has length $m$ locally at $F$ and length $m$ locally at $Q$.

Let $x f(z,w) + y g(z,w)$ be the $(1,m)$-form defining $X^{\phi^m}$. Since $F=[0:1][1:0] \in X^{\phi^m}$, it must be  that $w | g$. Since $X^{\phi^m}$ is irreducible, $w \! \not | f$. In the local ring $\struct_{T,F}$, let  $u = x/y$ and let $b := v^{-1} = w/z$. The curve $X^{\phi^m}$ is locally defined at $F$ by $a:= u + b^k \alpha$, for some unit $\alpha$ and some $ k \geq 1$.  Likewise, $Y^{\phi^m}$ is locally defined at $F$ by $u + b^j \beta = a + b^{\ell} \beta'$, where $j, \ell \geq 1$ and $\beta, \beta'$ are units.  Thus $B_1^m$ is locally defined at $F$ by the ideal $(a, b^{\ell})$. Since, as we have seen, $B_1^m$ has length $m$ at $F$, we must have $\ell = m$.

We show now that if $n < m$, then $B_1^n \subseteq B_1^m$. It is enough to prove this for $n =m-1$.  As we have already seen, $B_1^1 = \{ F, Q \}$ with reduced structure, and so $\sh{R}_1^{\phi}(0,-1) \cong \sh{I}_F \sh{I}_Q \struct(1,1)$, which has a two-dimensional space of  global sections.  Note that $X^{\phi}$, $Y^{\phi}$, $X+Z$, and $Y+W$ are global sections of this sheaf. Thus both $\{ X^{\phi}, Y^{\phi} \}$ and $\{ X + Z, Y + W \}$ span the same linear system $\mf{d} \subseteq \mb{P}\HB^0(T, \mc{O}(1,1))$. Pulling back by $\phi^{m-1}$ and using stability, we obtain that $\{ (X+Z)^{\phi^{m-1}}, (Y+W)^{\phi^{m-1}} \}$ and $\{ X^{\phi^m}, Y^{\phi^m} \}$ span the same linear system inside $\mb{P}\HB^0(T, \mc{O}(1,1)^{\phi^{m-1}}) = \mb{P}\HB^0(T, \mc{O}(1,m))$.  This implies that the scheme-theoretic intersections $B_1^m = X^{\phi^m} \cap Y^{\phi^m}$ and $(X+Z)^{\phi^{m-1}} \cap (Y+W)^{\phi^{m-1}}$ are the same, both being the base locus of this linear system.  But the latter intersection trivially contains $X^{\phi^{m-1}} \cap Y^{\phi^{m-1}} = B_{1}^{m-1}$.

Thus for $n < m$, $(B_1^n)_F$ is  defined by an ideal of codimension $n$ that contains $(a, b^m)$.  There is a unique such ideal, namely $(a, b^n)$.  By Lemma~\ref{sym-lem}, the local structure at $Q$ is symmetric.
\end{proof}

\begin{proposition}
\label{prop-0dim} Let $\tau = \tau(\rho, \theta) \in V(m+n)$, and recall that $B_n^m$ is the subscheme defined by $\sh{I}_n^{\phi^m}$. Then $B_n^m$ is a 0-dimensional subscheme of length $2 (m \binom{n+1}{2} + \binom{n+1}{3})$.  In case $o(\theta)> 2(n+m)$, then in the notation of Definition~\ref{def-fq}, $B_n^m$ is the following subscheme of fat points:
\begin{multline*}
B_n^m = n F_0 + \cdots + nF_{m-1} + (n-1)F_m + \cdots + F_{n+m-2} + \\
    n Q_0 + \cdots + nQ_{m-1} + (n-1)Q_m + \cdots + Q_{n+m-2}.
\end{multline*}
\end{proposition}
\begin{proof}
Note in all cases that $\sh{I}_n^{\phi^m} = \prod_{j=m}^{m+n-1} \sh{I}_1^{\phi^j}$. Thus this proposition will follow from our calculations in Lemmas~\ref{lem-pointchains1}, \ref{lem-pointchains2} above of the $n = 1$ case.

Suppose first that $o(\theta) > 2(n+m)$.  For all $i$ in the range $m \leq i \leq m+n-1$, the scheme defined by $\mc{I}_1^{\phi^i}$ is the reduced subscheme supported at the points $\{F_0, Q_0, \dots F_{i-1}, Q_{i-1} \}$, as we saw in Lemma~\ref{lem-pointchains1}. Then $\sh{I}_n^{\phi^m}$ defines the scheme of fat points of multiplicity $n$ at the points $F_0, Q_0, \dots F_{m-1}, Q_{m-1}$, multiplicity $n-1$ at the points $F_m, Q_m$, and so on, with multiplicity $1$ at $F_{n+m-2}, Q_{n+m-2}$. Since a fat point of multiplicity $a$ has length $\binom{a+1}{2}$, we calculate the length of this scheme to be
\[
2 \bigg[m \binom{n+1}{2} + \binom{n}{2} + \binom{n-1}{2} + \dots + \binom{2}{2} \bigg] = 2 m \binom{n+1}{2} + 2 \binom{n+1}{3}
\]
as claimed.

Now suppose that $\theta = 1$.  In this case Lemma~\ref{lem-pointchains2} applies, and shows that there is some choice of coordinates $a$ and $b$ in the local ring $S = \mc{O}_{T, F} \cong \kk[a,b]_{(a,b)}$ such that $\sh{I}_1^{\phi^i}$ is equal locally at $F$ to $(a, b^i)$ for all $i \leq m+n-1$.  Thus the sheaf $\mc{I}_n^{\phi^m}$ is equal locally at $F$ to $I = (a, b^m) \cdot (a, b^{m+1}) \cdots (a, b^{m+n-1})$.
Then $S/I$ has a basis consisting of the (images of the) monomials $a^i b^j$ for certain $i,j$. It is easy to see that if $\sum_{\ell=1}^{k-1} (m + \ell -1) \leq j < \sum_{\ell=1}^k (m + \ell -1)$, the monomial $a^i b^j$ occurs for all $0 \leq i < n-k +1$.   We see that $S/I$ has dimension $n  m + (n-1) (m+1) + \cdots + 1 (m+n-1)$.
Since $B_n^m$ must have the same length locally at $Q$ because of Lemma~\ref{sym-lem},   $B_n^m$ is supported at the two points $\{F, Q \}$ and once again has length
\[ 2 \sum_{i = 0}^{n-1} (n-i)(m+i) = 2m \binom{n+1}{2} + 2
\binom{n+1}{3},  \]
as claimed.
\end{proof}

We have been referring loosely to the sheaves $\sR^{\phi^m}_n$, the schemes $B^m_n$, etc., as families depending on $\tau$.  In the next result we make this explicit.

\begin{proposition}\label{prop-family}
 Let $n, m \in \NN$ and let $V := V(n+m)$.
\begin{enumerate}
 \item For all $0 \leq j \leq n+m$, there are closed subschemes $\Yfam_j, \Xfam_j$ of $T \times V$, flat over $V$, so that $\Yfam_j
 |_{T\times\tau}= Y^{\phi^j}$ and $\Xfam_j|_{T\times \tau} = X^{\phi^j}$ for $\tau \in V$.
 \item For all $0 \leq i \leq n$, there is an invertible sheaf $\Lfam^m_i$ on $T \times V$ so that $\Lfam^m_i|_{T \times \tau} =
\sL^{\phi^m}_i$
 for $\tau \in V$.
 \item For all $0 \leq i \leq n$ and $0 \leq j \leq m$ there is an ideal sheaf $\Ifam^j_i$ on $T \times V$, flat over $V$, so that
 $\Ifam^j_i|_{T\times \tau} = \sI^{\phi^j}_i$ for $\tau \in V$.
 \item For all $0 \leq i \leq n$ and $0 \leq j \leq m$ there is a closed subscheme $\Bfam^j_i$ of $T \times V$, flat over $V$, so
    that
$\Bfam^j_i|_{T\times \tau} = B^j_i(\tau)$ for $\tau \in V$.
\end{enumerate}
All these sheaves and subschemes are defined over $\FF$.
\end{proposition}
\begin{proof}
Let $h: T \times V  \to T \times V$ be given by the formula $(x, \tau)  \mapsto (\tau(x), \tau)$, and
let $\pi:  T \times V \to T $ be projection on the first factor.  Let $\wt{\sigma}: T \times V \dra T \times V$ be the birational map given by the formula $(x, \tau) \mapsto (\sigma(x), \tau)$ for $x$ in the domain of definition $U$ of $\sigma$. Note that $h$, $\pi$, and $\wt{\sigma}$ are defined over $\FF$.

$(1)$. For $0 \leq j \leq n+m$, we define   ideal sheaves $\sG_j$ on $T \times V$.  For $j=0$, put $\sG_0 := \pi^* \sO_T(-Y)$.  Then $\sG_0$ is invertible and defines $H_0$, which is the constant family $Y \times V$.

Suppose that $1 \leq j \leq n+m$, and that we have defined an invertible ideal sheaf $\sG_{j-1}$ on $T \times V$ so that $\sG_{j-1}|_{T \times \tau} = \sO_T(-Y^{\phi^{j-1}})$ for $\tau \in V$. Put $\sG_j := \wt{\sigma}^* h^* \sG_{j-1}.$  Here, $\wt{\sigma}^*$ is pullback by the birational map $\wt{\sigma}$.  We claim that $\sG_j|_{T \times \tau} = \sO_T(-Y^{\phi^{j}})$.  Certainly $h^* \sG_{j-1}|_{T \times \tau} = \sO_T(-(Y^{\phi^{j-1}})^{\tau})$ since $h$ restricts to the automorphism $\tau$ in each fiber $T \times \tau$ over $V$.  The verification that $\sG_j|_{T \times \tau} = \sO_T(-((Y^{\phi^{j-1}})^{\tau})^{\sigma})$ is then not much different.  One needs only to check that since the domain of definition $U \times V$ of $\wt{\sigma}$ intersects each fiber $T \times \tau$ in the open set $U \times \tau$, whose complement has codimension at least $2$, then each fiber of a pullback by $\wt{\sigma}$ is equal to the pullback by $\sigma$ of that fiber.  This follows directly from the definitions in Section~\ref{PULLBACK}.  By induction on $j$, for all $1 \leq j \leq n+m$ we get
\[
\sG_j|_{T \times \tau} = \sO_T(-(Y^{\phi^{j-1}})^{\phi}) = \sO_T(-Y^{\phi^{j}}) \cong \mc{O}(-1, -j),
\]
since $(Y^{\phi^{j-1}})^{\phi} = Y^{\phi^j}$ holds for all $\tau \in V \subseteq V(j)$ as we saw in the proof of Proposition~\ref{prop-fund-pts}.  This proves the claim, and defining the subscheme $\Yfam_j$ by the ideal sheaf $\sG_j$, it will have the required property that $\Yfam_j|_{T \times \tau} = Y^{\phi^{j}}$ for $\tau \in V$. Since each fiber of $\Yfam_j$ is a $(1,j)$-curve and all $(1,j)$-curves on $T$ have the same Hilbert series, by \cite[Theorem~III.9.9]{Ha}, $\Yfam_j$ is flat over $V$.     Since $\wt{\sigma}$, $h$, and $\pi$ are defined over $\FF$, so are $\sG_j$ and $\Yfam_j$.  By symmetry, $\Xfam_j$ exists as described for $0 \leq j \leq n+m$.

$(2)$.  An analogous argument to $(1)$ shows that we may find  an invertible sheaf $\sH_j$ on $T \times V$, defined over $\FF$, so that $\sH_j|_{T \times \tau} = \sO_T(-W^{\phi^j}-Y^{\phi^j})$ for $\tau \in V$.
  Let $0 \leq i \leq n$ and let $\Lfam_i^m := \sH_m^{-1} \sH_{m+1}^{-1} \cdots \sH_{m+i-1}^{-1}$.

$(3)$, $(4)$. If $i =0$ then the result is trivial.  Let $\Bfam^j_1:= \Yfam_j \cap \Xfam_j$.  Note that
\[ \Bfam^j_1 |_{T \times \tau} = \Yfam_j \times_{T\times V} \Xfam_j \times_{T \times V} (T \times \tau)
  = \Yfam_j |_{T\times \tau} \times_{T\times \tau} \Xfam_j|_{T\times \tau} = Y^{\phi^j} \cap X^{\phi^j} = B^j_1(\tau)
\]
for $\tau \in V$. Let $\Ifam^j_1$ be the ideal sheaf defining $\Bfam^j_1$. For $1 < i \leq n$,  let $\Ifam^j_i := \Ifam_1^j \Ifam_1^{j+1} \cdots \Ifam_1^{j+i-1}$. Let $\Bfam^j_i$ be the subscheme defined by $\Ifam^j_i$.

Now, $\Ifam^j_1 \cdot \sO_{T\times\tau} = \sI_1^{\phi^j}$.  Thus $\Ifam^j_n\cdot \sO_{T\times\tau} = \sI_1^{\phi^j}\cdots\sI_1^{\phi^{j+n-1}} = \sI_n^{\phi^j}$.  From the exact sequence
\[ 0 \to \Ifam^j_n \cdot \sO_{T\times \tau} \to \sO_{T\times \tau} \to (\sO_{\Bfam^j_n})|_{T\times\tau} \to 0 \]
we see that $\Bfam^j_n |_{T \times \tau}  = B^j_n(\tau)$.
By Proposition~\ref{prop-0dim}, for fixed $n, j$ the length of $B^j_n(\tau)$ is constant for $\tau \in V$.  Thus by \cite[Theorem~III.9.9]{Ha} $\Bfam^j_n$ and therefore $\Ifam^j_n$ are flat over $V$.

Since $T\times V$ is flat over $V$, by flat base change for Tor \cite[Prop. 3.2.9]{Weibel} we have
\[ \shTor^{T\times V}_1(\sO_{\Bfam^j_n}, \sO_{T\times\tau}) \cong \shTor_1^{T\times V}(\sO_{\Bfam^j_n}, \sO_{T\times V} \otimes_V \sO_\tau)
\cong \shTor^V_1(\sO_{\Bfam^j_n}, \sO_\tau).\]
This vanishes because $\Bfam^j_n$ is flat over $V$.  Thus for $\tau \in V$ we have $\Ifam^j_n|_{T \times \tau} = \Ifam^j_n \cdot \sO_{T\times \tau} = \sI^{\phi^j}_n$.
Finally, since $\Yfam_j$ and $\Xfam_j$ are defined over $\FF$, by construction $\Ifam^j_i$ and $\Bfam^j_i$ are defined over $\FF$.
\end{proof}

\begin{corollary}
\label{cor-semicon} Fix $n, m \in \NN$ and $a,b \in \ZZ$.  There is a dense open subset $U \subseteq V(m+n) \subseteq \mb{G}$, with $\mb{1} \in U$, so that if $\tau \in U$, then $h^i(T, \sh{R}_n^{\phi^m}(a,b)) \leq h^i(T, \sh{A}_n^{\sigma^m}(a,b))$ for $i = 0, 1,2$.  Further, $\GG \smallsetminus U$ is defined over $\FF$.
\end{corollary}
\begin{proof}
Let $V:=V(n+m)$.  Let $\pi:  T\times V\to T $ be projection on the first factor; let $\Ifam^m_n$ and $\Lfam^m_n$ be the sheaves on $T \times V$ defined in Proposition~\ref{prop-family}.  Let $\sN:= \Ifam^m_n\otimes \Lfam^m_n \otimes \pi^* \sO(a,b)$. This sheaf is defined over $\FF$ and flat over $V$, and for any $\tau \in V$ we have $\sN|_{T\times\tau} = \sR^{\phi^m}_n (a,b)$.  We now apply upper semi-continuity \cite[Theorem III.12.8]{Ha} to obtain a open neighborhood $U$ of $\mb 1 \in V$ so that the statement holds.
\end{proof}

To apply the corollary, we study in the next result the cohomology of the sheaves $\mc{A}_n^{\sigma^m}$, and their relation to the ring $A$.
\begin{lemma}
\label{lem-A-HS}
Let $m,n \in \NN$.
\begin{enumerate}
\item
$(A_nt^{-n})^{\sigma^m} \subseteq K$ has a $\kk$-basis consisting of all monomials
\[
\{ u^iv^j | 0 \leq i \leq n, a(i) \leq j \leq b(i) \}, \text{where}\ a(i) = im + \binom{i}{2}, \ \ \ b(i) = im + \binom{n+1}{2} - \binom{n-i}{2}.
\]
In particular, $\dim_\kk A_n = \binom{n+3}{3}$ for all $n \geq 0$ and $A$ has Hilbert series $h_A(s) = 1/(1-s)^4$.
\item $(A_nt^{-n})^{\sigma^m} = H^0(T, \sh{A}_n^{\sigma^m})$ and $H^1(T, \sh{A}_n^{\sigma^m}) = 0$.
\end{enumerate}
\end{lemma}
\begin{proof}
(1). Recall that we write $E = \kk +\kk u+ \kk  v + \kk uv $, so $A = \kk \langle Et \rangle \subseteq \kk(u,v)[t; \sigma]$.  We need to calculate $E_n^{\sigma^m} := (A_nt^{-n})^{\sigma^m} = E^{\sigma^m} E^{\sigma^{m+1}} \cdots E^{\sigma^{m+n-1}}$.  It is enough to prove the case $m = 0$, for $E_n^{\sigma^m}$ is formed by sending each monomial $u^iv^j$ occurring in $E_n$ to $(u^iv^j)^{\sigma^m}= u^iv^{j+im}$, and  the bounds $a(i)$ and $b(i)$ simply adjust by $im$ for each $i$.

Now $E^{\sigma^i}$ is the $\kk$-span of $\{1,  uv^i, v, uv^{i+1} \}$ for each $i \geq 0$, so $E_n$ is spanned by all possible products of $n$ monomials, one from each $E^{\sigma^i}$ with $0 \leq i \leq n-1$.  Consider which monomials $u^iv^j$ are in this spanning set for a given fixed $i$ with $0 \leq i \leq n$.  Clearly the smallest value of $j$ occurs when one chooses $u, uv, uv^2, \dots uv^{i-1}$ from $E, E^{\sigma}, \dots, E^{\sigma^{i-1}}$ respectively, and $1$ from the remaining spaces $E^{\sigma^i}, \dots, E^{\sigma^{n-1}}$; while the largest value of $j$ occurs by taking $v$ from each of the spaces $E, E^{\sigma}, \dots, E^{\sigma^{n-i-1}}$, and then $uv^{n-i+1}, \dots, uv^n$ from the spaces $E^{\sigma^{n-i}}, \dots, E^{\sigma^{n-1}}$, respectively.  Thus $1 + 2 + \dots + i-1 = \binom{i}{2} \leq j \leq (n-i) + (n-i+1) + \dots + n = \binom{n+1}{2}-\binom{n-i}{2}$, and it is easy to see that every $j$ in this range actually occurs.  Thus $E_n$ has the claimed basis, and as a consequence
\[
\dim_{\kk} A_n = \dim_{\kk} E_n = \sum_{i = 0}^n \bigg[ \binom{n+1}{2}-\binom{n-i}{2} - \binom{i}{2} + 1 \bigg]= \binom{n+3}{3}.
\]

(2). $T$ is covered by the four open sets
\[
U^+_+ := \spec \kk[u,v], \ \ U_-^+ := \spec \kk[u, v^{-1}], \ \ U_-^- := \spec \kk[u^{-1}, v^{-1}], \  \ U_+^- := \spec \kk[u^{-1}, v],
\]
and by definition $\sh{A}_n^{\sigma^m}$ is the sheaf globally generated on $T$ by the sections in $W := E_n^{\sigma^m}$.  On any affine open set $U$ of the cover, $\sh{A}_n^{\sigma^m}(U) =  W \mc{O}_T(U)$ has an easily calculable $\kk$-basis of monomials $u^iv^j$.  In particular,
\[
W \mc{O}_T(U_-^+) = \kk \{ u^iv^j | i \geq p, j \leq q,\ \text{some}\ u^pv^q \in W \}, \ \ W \mc{O}_T(U_+^-) = \kk \{ u^iv^j | i \leq p, j \geq q,\ \text{some}\ u^pv^q \in W \}.
\]
Now, $H^0(T, \sh{A}_n^{\sigma^m}) = W \mc{O}_T(U_+^+) \cap W \mc{O}_T(U_-^+) \cap W \mc{O}_T(U_-^-) \cap W \mc{O}_T(U_+^-).$ If $u^iv^j \in W \mc{O}_T(U_-^+) \cap W \mc{O}_T(U_+^-)$, then (i) $i \geq p, j \leq q$ for some $u^pv^q \in W$ and (ii) $i \leq p', j \geq q'$ for some $u^{p'}v^{q'} \in W$.  Note that (i) forces $i \geq 0$ and (ii) forces $i \leq n$, so $0 \leq i \leq n$. Then since $a(i)$ and $b(i)$ are increasing functions of $i$ for $0 \leq i \leq n$, (i) forces $j \leq b(i)$ and (ii) forces $j \geq a(i)$.  It follows that $W \mc{O}_T(U_-^+) \cap W \mc{O}_T(U_+^-) = W$ already, so $H^0(T, \sh{A}_n^{\sigma^m}) = W = (A_nt^{-n})^{\sigma^m}$.

In particular, we obtain from the above that $\binom{n+3}{3} = \dim_\kk A_n = \dim_\kk E_n^{\sigma^m} = h^0(T, \sh{A}_n^{\sigma^m})$. Recall that $\sh{A}_n^{\sigma^m} = \sh{I}_n^{\sigma^m} \sh{L}_n^{\sigma^m}$, where  $\sh{I}_n^{\sigma^m}$ is the ideal sheaf defining the scheme $B_n^m(\mathbb 1)$.  By Proposition~\ref{prop-fund-pts}(3), we have $\sh{L}_n^{\sigma^m} \cong \struct(n, \binom{n+m+1}{2} - \binom{m+1}{2})$. By the K\"unneth formula, $H^1(T, \struct(n, \binom{n+m+1}{2} - \binom{m+1}{2})) = 0$.  Consider the exact sequence
\[ 0 \to \sh{A}_n^{\sigma^m} \to \struct(n, \binom{n+m+1}{2} - \binom{m+1}{2}) \to \struct_{B_n^m(\mb 1)} \to 0.\]
 Proposition~\ref{prop-0dim} and the associated long exact cohomology sequence give us that
\begin{multline*}
 h^1(T, \sh{A}^{\sigma^m}_n) = -h^0(T, \struct(n,\binom{n+m+1}{2} - \binom{m+1}{2})) + h^0(T, \sh{A}^{\sigma^m}_n)
 + \len( \struct_{B_n^m(\mb 1)})  \\
= -(n+1)(\binom{n+m+1}{2} - \binom{m+1}{2}+1) + \binom{n+3}{3} + 2 \bigg(m \binom{n+1}{2} + \binom{n+1}{3}\bigg) = 0.
 \end{multline*}
 Therefore, $H^1(T, \sh{A}_n^{\sigma^m}) = 0$ for all $n, m
 \geq 0$.
\end{proof}

\begin{remark}
\label{rem-not-noeth} It easily follows from the explicit basis given in the preceding lemma that $A$ is not noetherian, since the right ideal $\sum_{n \geq 1} uv^{2n-1}t^n A$ is infinitely generated.
\end{remark}

Recall that we say that $U \subseteq \GG$ is a {\em general subset} if it is the complement of a countable union of proper closed subvarieties.
\begin{proposition}\label{prop-1E}
There is a  general subset $U$ of $\mb{G}$ such that for all $\tau \in U$ and for all $m \geq 0$, the Hilbert series of $\bigoplus_{n \in \NN} H^0(T, \sh{R}_n^{\phi^m})$ is $\D \frac{1}{(1-s)^4}$ and $H^1(T, \sh{R}_n^{\phi^m}) = 0$ for all $n \geq 0$. In particular, for $\tau \in U$ we have $\dim_\kk R_n \leq \binom{n+3}{3}$.  Further, $U$ contains  $\tau(\rho, \theta)$ for all pairs $(\rho, \theta)$ that are algebraically independent over $\FF$.
\end{proposition}
\begin{proof}
Fix $n, m \geq 0$. We have computed in Lemma~\ref{lem-A-HS} that $H^1(T, \sh{A}_n^{\sigma^m}) = 0$.  Therefore, by Proposition~\ref{cor-semicon}, there is a nonempty open subset $U(n,m) \subseteq V(n+m) \subseteq \mb{G}$ such that $H^1(T, \sh{R}_n^{\phi^m}(\tau)) = 0$ for $\tau \in U(n,m)$.  Since the complement of $U(m,n)$ is defined over $\FF$, we have $\tau(\rho, \theta) \in U(n,m)$ for all algebraically independent pairs $(\rho, \theta)$.

Take $U = \bigcap_{n, m \geq 0} U(n,m)$, and let $\tau \in U$.   We have that $\sh{L}_n^{\phi^m} \cong \struct(n,\binom{n+m+1}{2} - \binom{m+1}{2})$ by Proposition~\ref{prop-fund-pts}(3).  The scheme $B_n^m(\tau)$ defined by $\mc{I}_n^{\phi^m}$ has length $2(m \binom{n+1}{2} + \binom{n+1}{3})$, by Proposition~\ref{prop-0dim}. Since $H^1(T, \sh{R}_n^{\phi^m}(\tau)) = 0$, we may deduce from the long exact cohomology sequence associated to $\sh{R}_n^{\phi^m} \subseteq \struct(n,\binom{n+m+1}{2} - \binom{m+1}{2} )$  that
\[
  h^0(T, \sh{R}^{\phi^m}_n) = (n+1)(\binom{n+m+1}{2}
- \binom{m+1}{2}+1) - 2(m \binom{n+1}{2} + \binom{n+1}{3}) = \binom{n+3}{3}.
\]
Since $R_nt^{-n} \subseteq H^0(T, \sh{R}_n)$, all statements are now immediate.
\end{proof}

We will see in the next section that for general $(\rho, \theta)$ (in particular for a pair algebraically independent over $\FF$) then  we have $\dim_\kk R_n = \binom{n+3}{3}$ for all $n \in \NN$.

\begin{remark}
There is a fair amount of literature on regularity of fat point schemes on multiprojective spaces.  However, the cohomology vanishing in Proposition~\ref{prop-1E} does not seem to be given by these results.  In particular,
it follows from \cite[Theorem~5.1]{SVT2006} that, if $\tau$ is general, then
$H^1(T, \sI_n (i,j)) = 0$ for any $i, j$ with $i, j \geq n-2$ and $i+j \geq 2\sum_{k=1}^{n-1} k = 2 \binom{n}{2}$.  For Proposition~\ref{prop-1E}, however, we need  $(i,j)=(n, \binom{n+1}{2})$.
\end{remark}

\section{Presentation, Hilbert series and free resolution of $\kk$}\label{HILBERT}
In this section, we analyze the resolution of the trivial module $\kk_R$, and more specifically the presentation of $R$ by generators and relations.  We show that there is a uniform description of the resolution of $\kk$ for  general $\tau$, and compute the Hilbert series and some homological properties of (general) $R(\tau)$.  Our main technique is to prove these results for $A$ and analyze their behavior under deformation.

We will rely heavily on the notation and formulas established in Section~\ref{NOTATION}.   In particular, recall that for given $\tau = \tau(\rho, \theta)$ we set $\gamma = \rho + 1, \delta = \rho -1, \epsilon = \theta + 1,$ and $\zeta = \theta -1$, as these expressions simplify the formulas for $u^\phi$ and $v^\phi$ as in \eqref{eq-phi}. Write $r_1 = t, r_2 = ut, r_3 = vt, r_4 = uvt$, so that $R = R(\tau) = \kk \langle r_1, r_2, r_3, r_4 \rangle \subseteq \kk(u,v)[t; \varphi]$. It is easy to calculate some quadratic relations among the $r_i$.  For example, suppose that $t (ft) = (vt)(gt)$ for some $ft, gt \in R_1$, so $f, g \in E = \kk+ \kk u + \kk v + \kk uv$.  Then $f^\phi = v g^\phi$, or equivalently $f = v^{\phi^{-1}}g$.  Then using \eqref{eq-phi}, $(-\zeta v + \epsilon) f = (\epsilon v - \zeta) g$ and there are two linearly independent solutions:  $f = \epsilon v - \zeta$, $g =-\zeta v + \epsilon$, and $f = u(\epsilon v - \zeta)$, $g = u(-\zeta v + \epsilon)$.  Similarly, one can find two relations of the form $r_2(ft) = r_4(gt)$ and two relations of the form $r_1(ft) = r_4(gt)$.
Let $\kk \ang{x_1, x_2, x_3, x_4} $ be the free algebra, and  consider the surjection $\pi: \kk \langle x_1, x_2, x_3, x_4 \rangle \to R; \ \ x_i \mapsto r_i.$  The process above produces the following six quadratic elements in the ideal of relations $J = \ker \pi$: \beq \label{rels} \begin{split} f_1 =  x_1(\zeta x_1 - \epsilon x_3) + x_3(\epsilon x_1 - \zeta x_3), \ \ \
f_2 =  x_1(\zeta x_2 - \epsilon x_4) + x_3 (\epsilon x_2 - \zeta x_4), \\
 f_3 = x_2(\zeta x_1 - \epsilon x_3) + x_4(\epsilon x_1 - \zeta x_3), \ \ \
 f_4  = x_2(\zeta x_2 - \epsilon x_4) + x_4(\epsilon x_2 - \zeta
x_4), \\
 f_5 =  x_1(\delta x_1 - \gamma x_2) + x_4(\gamma x_1 - \delta x_2), \ \ \
 f_6 = x_1(\delta x_3 - \gamma x_4) + x_4(\gamma x_3 - \delta x_4).
\end{split}
\eeq Since the coefficients in these relations depend only on $\tau$, we set $S(\tau) := \kk \langle x_1, x_2, x_3, x_4 \rangle/(f_1, f_2, \dots, f_6)$.  We shall see that for  general $\tau$, the surjection $S(\tau) \to R(\tau)$ is an isomorphism.  For now note that the relations $f_1$--$f_6$ give precisely the relations in Theorem~\ref{thm-main}(2) in case $\gamma \neq 0, \epsilon \neq 0$.

We set up some additional notation which will be useful throughout this section.
\begin{notation}
\label{not-z} It is convenient to name the following special elements in $R(\tau)_1$:
\begin{gather*}
z_1 = (\zeta - \epsilon v)t, \ z_2 = (\epsilon  - \zeta v)t, \ z_3 = (\zeta u - \epsilon uv)t, \ z_4 = (\epsilon u - \zeta uv)t, \ z_5 = (\delta - \gamma u)t, \ z_6 = (\gamma - \delta u)t, \\ z_7 = (\delta v - \gamma uv)t, \ z_8 = (\gamma v - \delta uv)t, \ z_9 = (\gamma u - \delta)(-\zeta v + \epsilon)t, \ z_{10} = (\delta u - \gamma)(\epsilon v - \zeta)t.
\end{gather*}
\end{notation}
\begin{lemma}
\label{lem-z} Assume Notation~\ref{not-z}.  The following relations hold in $R(\tau)$:
\begin{gather*}
r_1 z_1 + r_3z_2 = 0, \ r_1z_3 + r_3 z_4 = 0, \ r_2 z_1 + r_4 z_2 = 0, \ r_2 z_3 + r_4 z_4 = 0, \ r_1 z_5 + r_4 z_6 = 0, \ r_1z_7 + r_4
z_8 = 0, \\
z_5 z_1 + z_7z_2 = 0, \ z_5 z_3 + z_7 z_4=0 , \ z_6z_1 + z_8 z_2 = 0, \ z_6
z_3 + z_8 z_4 = 0, \\
z_9z_1 + z_{10}z_2 = 0, \ z_9z_3 + z_{10}z_4 = 0, \ z_1z_9 + z_3 z_{10} = 0, \ z_2 z_9 + z_4 z_{10} = 0.
\end{gather*}
\end{lemma}
\begin{proof} The first six relations are just $f_1$ through $f_6$.  The
others are checked easily, using \eqref{eq-phi}.
\end{proof}

The next result gives a complex which is a potential free resolution over $R(\tau)$ of the trivial module $\kk$.  We will prove later that this complex is exact for general $\tau$.  For notational purposes, we will think of the right module $R^n$ as a column vector.  An $R$-module map $M:  R[a]^n \to R[b]^m$ is therefore an $m\times n$ matrix of elements of $R_{b-a}$, acting by left multiplication.

\begin{proposition}\label{prop-resolution}
For any $\tau \in \mb{G}$, there is a complex of right $R = R(\tau)$-modules \beq\label{Koszul}
 0 \to R[-4] \overset{M}{\to} R[-3]^{\oplus 4} \overset{N}{\to} R[-2]^{\oplus 6} \overset{P}{\to} R[-1]^{\oplus 4}
\overset{Q}{\to} R \to \kk \to 0, \eeq where here
\begin{equation}
\label{eq-matrices} M = \begin{pmatrix}  0  \\ 0  \\  z_9
\\  z_{10} \end{pmatrix}, \
N = \begin{pmatrix} z_9 & 0 & 0 & 0 \\
z_{10} & 0  &  0 & 0 \\
0 & z_9 &  0 & 0 \\
0 & z_{10} & 0 & 0 \\
0 & 0 &  z_1 & z_3  \\
0 &  0 &  z_2 & z_4
\end{pmatrix}, \
P = \begin{pmatrix}
z_1 & z_3   & 0 &  0 & z_5 & z_7 \\
0 & 0 & z_1 & z_3 & 0 & 0 \\
z_2  & z_4& 0 & 0 & 0 & 0 \\
0 & 0 & z_2 & z_4 & z_6 & z_8
\end{pmatrix}, \
Q = \begin{pmatrix} r_1  & r_2 &  r_3 & r_4
\end{pmatrix}.
\end{equation}
\end{proposition}
\begin{proof}
The fact that this is a complex is equivalent to the matrix equations $QP = 0$, $PN = 0$, $NM = 0$, which are in turn equivalent to the relations computed in Lemma~\ref{lem-z}.
\end{proof}

In the case that $\tau = \mathbb 1$, we can analyze the Hilbert series of $R(\tau) = A$, and show that \eqref{Koszul} is exact, fairly directly. We do this in the next two results.  We note that such basic properties of $A$ were also obtained by Paul Smith and James Zhang in unpublished work \cite[Proposition~7.6]{YZ2006}.   Given $a, b \in A$, we use the notation $\operatorname{syz}_r(a,b)$
for the module of right syzygies between $a$ and $b$; in other words, $\operatorname{syz}_r(a,b) = \{(x, y) | ax + by = 0 \} \subseteq A^2$.
 Similarly, $\operatorname{syz}_{\ell}(a,b) = \{(x,y) | xa + yb = 0 \} \subseteq A^2$ is the module of left syzygies.

\begin{lemma}
\label{lem-A-prop}
Consider $A = R(\mathbb 1)$.
\begin{enumerate}
\item $A \cong S(\mb 1) = \kk \langle x_1, x_2, x_3, x_4 \rangle/(f_1, f_2, \dots, f_6)$; in particular, $h_A(s) = h_{S(\mb
1)}(s) = 1/(1-s)^4$.
\item $\operatorname{syz}_r(r_1, r_2) = (r_2, -r_3)A = \operatorname{syz}_r(r_3, r_4)$.
\item $\operatorname{syz}_r(r_1, r_3) = (r_3, -r_1)A + (r_4, -r_2)A  = \operatorname{syz}_r(r_2, r_4)$.
\item Dually to parts (2) and (3), we have $\operatorname{syz}_{\ell}(r_1, r_2) = A(r_4,
-r_1) = \operatorname{syz}_{\ell}(r_3, r_4)$,  and \\
$\operatorname{syz}_{\ell}(r_1, r_3) = A(r_3, -r_1) + A(r_4, -r_2) = \operatorname{syz}_{\ell}(r_2, r_4)$.
\end{enumerate}
\end{lemma}
\begin{proof}
(1). Take lexicographic order on the monomials in the $x_i$, with $x_2 < x_1 < x_3 < x_4$. In the case $\tau = \mathbb 1$ at hand, we have $\gamma = \epsilon = 1, \delta =\zeta = 0$, and so the relations of $S(\mb 1)$ become especially simple binomial relations:
\begin{gather*}
f_1 =  x_3x_1 - x_1x_3, \ f_2 =  x_3x_2 - x_1x_4, \ f_3 =  x_4x_1 - x_2x_3, \\
f_4  = x_4x_2 - x_2x_4, \ f'_5 =  x_1x_2 - x_2x_3, \  f_6 = x_4x_3 - x_1x_4.
\end{gather*}
Here, we have replaced $f_5$ by $f'_5 = f_3 - f_5$ so that the leading terms $x_3x_1, x_3x_2, x_4x_1, x_4x_2, x_1x_2, x_4x_3$ of the relations with respect to the order are distinct.  It is routine to check that all of the overlaps between these relations are resolvable, and so by Bergman's diamond lemma \cite{Bergman1978}  the set of irreducible words $\{ x_2^i x_1^j x_3^k x_4^{\ell} | i,j,k, \ell \geq 0 \}$ is a $\kk$-basis for $S(\mb 1)$.  The Hilbert series $h_{S(\mb 1)}(s) = 1/(1-s)^4$ is immediate. Since $A$ has the same Hilbert series by Lemma~\ref{lem-A-HS}, the surjection $\pi: S(\mb 1) \to A$ must be an isomorphism.

(2).  By part (1), we may work with the ring $S = S(\mb 1 )$ instead, which we do for the rest of the proof.  Consider the monomial ordering and basis of irreducible words given in part (1). Let $M := \operatorname{syz}_r(x_1, x_2) = \{ (f,g) \in S^2 | x_1f + x_2 g = 0 \}$. Obviously $(x_2, -x_3) \in M$ by relation $f'_5$. If $(f, g) \in M$ where $f$ is a linear combination of irreducible words, we can subtract an element in $(x_2, -x_3)S$ to yield $(f', g') \in M$ where $f'$ is a linear combination of irreducible words not containing $x_2$.

Define $Z_i$ to be the $\kk$-span of all irreducible words which begin with $x_i$; thus $S_{\geq 1} = Z_2 \oplus Z_1 \oplus Z_3 \oplus Z_4$ as vector spaces.  Now $x_1 f' \in Z_1$ by the previous paragraph, and rewriting $g'$ if necessary so that it is a linear combination of irreducible words, clearly $x_2 g' \in Z_2$.  So $x_1f' = -x_2g' \in Z_1 \cap Z_2 = 0$, which forces $f' = g' = 0$ since $S(\mb 1) = A$ is a domain. Thus $M = (x_2, -x_3)S$.  Now since $x_3 = vx_1$ and $x_4 = vx_2$, it is easy to see that $\operatorname{syz}_r(x_1, x_2) = \operatorname{syz}_r(x_3, x_4)$.

(3).  Maintain the notation of part (2).  First, an easy argument using
the basis of irreducible words shows that $x_1 S \subseteq Z_1 + Z_2$.  Then the result follows from a
similar argument to that in part (2), which we leave to the reader.

(4).  It is straightforward to check using the relations that the vector space bijection $S_1 \to S_1$ defined by $x_1 \mapsto x_3$, $x_2 \mapsto x_4$, $x_3 \mapsto x_1$, $x_4 \mapsto x_2$ extends to an anti-isomorphism $S \to S$.  We note (without proof) that this anti-isomorphism is the map given by  Corollary~\ref{cor-Aop}.
\end{proof}

\begin{proposition}
\label{prop-A-exact} As above, consider $A = R(\mathbb 1)$.
\begin{enumerate}
\item The complex \eqref{Koszul} is exact and so is a free resolution of $\kk_A$. \item $\Ext^1_A(\kk, A) = 0$.
\end{enumerate}
\end{proposition}
\begin{proof}
(1). In this case, the entries $z_i$ of the matrices in the complex \eqref{Koszul} simplify to be scalar multiples of the $r_i$.
Exactness of the complex at the $A[-1]^{\oplus 4}$ spot is now easily seen to be equivalent to the fact that $f_1, f_2, \dots, f_6$ generate the kernel of $\kk \ang{x_1, \dots, x_4} \to A$, as was proved in Lemma~\ref{lem-A-prop}(1).   Exactness at the $A[-4]$ spot follows because $A$ is a domain. Finally, it is straightforward to see that exactness of the complex in the remaining $A[-2]^{\oplus 6}$ and $A[-3]^{\oplus 4}$ spots requires precisely the right syzygy results proved in Lemma~\ref{lem-A-prop}(2)(3).  So \eqref{Koszul} is exact in this case.

(2).  Since \eqref{Koszul} is exact, we can compute $\Ext^i_A(\kk, A)$ as the $i$th homology of the complex \beq \label{dual}
 0 \to A \overset{Q^*}{\to} A[1]^{\oplus 4} \overset{P^*}{\to} A[2]^{\oplus 6} \overset{N^*}{\to} A[3]^{\oplus 4}
\overset{M^*}{\to} A[4] \to 0, \eeq given by applying $\Hom_A(\blank, A)$ to \eqref{Koszul}.  Since this is a complex of left modules, we will write the free modules as row vectors, and write $Q^*$, $P^*$, etc. as right multiplication by the matrices giving $Q$, $P$, etc.  Now to prove $\Ext^1_A(\kk, A) = 0$ we need to prove that $\ker P^* = \im Q^*$, where clearly $\im Q^* = A(r_1, r_2, r_3, r_4)$.  But an easy argument using the left syzygies computed in Lemma~\ref{lem-A-prop}(4) shows that $\ker P^* = A(r_1, r_2, r_3, r_4)$ as needed.
\end{proof}

Now we will study the complex \eqref{Koszul} as  $\tau \in \mb{G}$ varies.
\begin{lemma} \label{lem-Koszul}
Let $R = R(\tau)$.
\begin{enumerate}
\item For each $n \geq 1$, there is an open set $U_n \subseteq \mb{G}$, with $\mb 1 \in U_n$, such that $\dim_\kk R(\tau)_n = \binom{n+3}{3}$
for all $\tau \in U_n$.
\item For each $n \geq 1$, there is an open subset $V_n \subseteq U_n \subseteq \mb{G}$ such that $\eqref{Koszul}$ is exact in degree $n$ for
$\tau \in V_n$; moreover, $\mb 1 \in V_n$ and the complement of $V_n$ is defined over $\FF$.
\end{enumerate}
\end{lemma}
\begin{proof}
To save notation, let us identify $\GG$ with $\kk^* \times \kk^*$. We first fix some degree $n$ and consider the degree $n$ component of any one of the maps occurring in the complex $\eqref{Koszul}$; in more general notation, this looks like $\Omega(\tau): R[-m]^{\oplus i}_n \to R[-m+1]^{\oplus j}_n$, where the map $\Omega(\tau)$ is given by a matrix with entries in $R(\tau)_1$.  We think of $\rho, \theta$ as parameters now and note that the nonzero entries $z_i$ in the matrix are fixed elements in $\FF[\rho, \theta][u,v]t$.
We assume that $n \geq m$, since otherwise $\Omega(\tau) = 0$.   Writing down all possible words of degree $n-m$ in the elements $r_1, r_2, r_3, r_4$ and multiplying them out using \eqref{eq-phi}, one gets a $\kk$-spanning set $Y$ for $R_{n-m}$ consisting of $4^{n-m}$ elements of the form $f t^{n-m}$, where $f \in \mathbb{F}(\rho, \theta,u,v)$.  Then one gets a $\kk$-spanning set for $R[-m]^{\oplus i}_n$ consisting of $i \cdot 4^{n-m}$ $i$-tuples of such elements. Applying $\Omega(\tau)$, we get a $\kk$-spanning set for $\im \Omega(\tau)$ consisting of $j$-tuples of elements in $\mathbb{F}(\rho, \theta, u,v)t^{n-m+1}$.  We multiply by
$t^{-n+m-1}$, obtaining a set $Z \subseteq \mathbb{F}(\rho, \theta,u,v)^{\oplus j}$.

Following through the construction of $Z$, one may check that every element of $Z$ has a well-defined evaluation at any $(\rho, \theta) \in \GG$; in other words, each fraction appearing has a denominator which is not identically $0$ when evaluated at any $(\rho, \theta) \in \GG$. Indeed, this must be true since by construction, specializing $Z$ at any particular $(\rho, \theta)$ should give a set whose $\kk$-span (times $t^{n-m+1}$) is equal to $\im \Omega(\tau(\rho, \theta))$.
Now it is standard that $\dim_{\kk} Z(\rho, \theta)$ behaves lower-semicontinuously in $(\rho, \theta)$; in other words, for every $d \geq 0$, the condition $\dim_{\kk} Z \leq d$ is a closed condition on $(\rho, \theta) \in \mb{G}$; moreover, this closed set is cut out by the vanishing of polynomials in $\mathbb{F}[\rho, \theta]$. In conclusion, there is an open subset of $\mb{G}$, whose complement is defined over $\FF$, on which $\dim_{\kk} \im \Omega(\tau)$ achieves its maximum.

(1). We apply the argument in the preceding two paragraphs to the map $Q$ in \eqref{Koszul}. It shows that for all $n \geq 1$, there is an open set $U_n \subseteq \mb{G}$ of $\tau$ for which $\dim_\kk R(\tau)_n$ achieves a maximum value $d_n$.  Note that $d_n \geq \dim_\kk A_n = \binom{n+3}{3}$ by Lemma~\ref{lem-A-HS}(1).  By Proposition~\ref{prop-1E}, there is also a  general subset of $\mb{G}$ for which $\dim_\kk R(\tau)_n \leq \binom{n+3}{3}$.  This forces $d_n = \binom{n+3}{3}$ and $\mb 1 \in U_n$.

(2).  Fix $n \geq 1$ and define $U_n$ as above.  Then for all $\tau \in U_n$, $\dim_\kk R(\tau)_n = \binom{n+3}{3}$ is constant by part (1). Thus for all $\tau \in U'_n = U_n \cap U_{n-1} \cap \dots \cap U_{n-4}$, (omit any term $U_i$ with $i \leq 0$), the $\kk$-dimension in degree $n$ of each term in the complex \eqref{Koszul} is the same.  Let $\Omega(\tau)$ be
 any degree $n$ map  occurring in the complex, in the notation
of the first paragraph of the proof.  Since $\dim_{\kk} \im \Omega(\tau)$ behaves lower-semicontinuously for $\tau \in U'_n$, by the rank-nullity formula $\dim_{\kk} \ker \Omega(\tau)$ must behave upper-semicontinuously in $\tau \in U'_n$.  Then for any $i$, the $\kk$-dimension of the degree-$n$ piece of the $i$th homology of the complex \eqref{Koszul} will also behave upper-semicontinuously in $\tau \in U'_n$, and so will achieve a minimum along an open subset of $U'_n$.  But we saw in Proposition~\ref{prop-A-exact} that \eqref{Koszul} is exact when $\tau = \mathbb 1$.  Thus the minimum dimensions for the homology groups are $0$ in each degree; in other words,
 there is an open subset $V_n \subseteq U'_n \subseteq U_n \subseteq
\mb{G}$, with $\mb 1 \in V_n$, such that $\eqref{Koszul}$ is exact in degree $n$ as claimed.  The closed subset of $U'_n$  where each $\dim_{\kk} \ker \Omega(\tau)$ does not achieve its minimum is the same as the closed subset where $\dim_{\kk} \im \Omega(\tau)$ does not achieve its maximum, and we already saw in the first part of the proof that this closed subset is defined over $\FF$.
\end{proof}

The next result shows, among other things, that for general $\tau$ parts (2)-(4) of Theorem~\ref{thm-main} hold for $R(\tau)$.
\begin{proposition}
\label{prop-Koszul} There is a  general subset $U \subseteq \mb{G}$, with $\mb 1 \in U$ and $(\rho, \theta) \in U$ for any pair $(\rho, \theta)$ algebraically independent over $\FF$, such that $R = R(\tau)$ has the following properties for any $\tau \in U$:
\begin{enumerate}
\item The complex \eqref{Koszul} is exact, $R$ is Koszul of global dimension 4 with $h_R(s) = 1/(1-s)^4$, and $R \cong S(\tau) = \kk \langle
x_1, x_2, x_3, x_4 \rangle/(f_1, f_2, \dots, f_6)$;
\item $\Ext^1_R(\kk, R) = 0$;
\item $\dim_{\kk} \Ext^i_R(\kk, R) = \infty$ for $i = 2, 3, 4$;
\item $R$ is not AS-Gorenstein, and so $R$ is not a regular algebra; $R$ fails $\chi_2$ on the right; and $\depth R = 2$, so the
Auslander-Buchsbaum formula fails for the module $M := \kk_R$.
\end{enumerate}
\end{proposition}
\begin{proof}
Let $V_n \subseteq \mb{G}$ be the open subset occurring in Lemma~\ref{lem-Koszul}(2) for each $n \geq 1$, and let $V := \bigcap_{n \geq 1} V_n$.

(1).  Let $\tau \in V$.   The complex \eqref{Koszul} is exact by the construction of $V_n$ in Lemma~\ref{lem-Koszul}(2).  Then $R$ is Koszul, and $\gldim R = \pd \kk = 4$ by \cite{Li1996}.  The Hilbert series of $R(\tau)$ also follows immediately from the shape of the free resolution of $\kk$ in \eqref{Koszul}.  The fact that the kernel of the map $\kk \langle x_1, x_2, x_3, x_4 \rangle \to R(\tau)$ is generated as an ideal by $\{f_1, \dots, f_6 \}$ follows from the exactness of \eqref{Koszul} at the $R[-1]^{\oplus 4}$ spot.

(2).  As in Proposition~\ref{prop-A-exact}(2), we examine the complex \beq \label{dual2}
 0 \to R \overset{Q^*}{\to} R[1]^{\oplus 4} \overset{P^*}{\to} R[2]^{\oplus 6} \overset{N^*}{\to} R[3]^{\oplus 4}
\overset{M^*}{\to} R[4] \to 0, \eeq given by applying $\Hom_R(\blank, R)$ to \eqref{Koszul}, where the free modules are rows and the maps $M^*, N^*$ etc. are right multiplication by $M, N$, etc.  By $(1)$ we can calculate $\Ext^i_R(\kk, R)$ as the $i$th homology of \eqref{dual2} for $\tau \in V$.

But now analogous arguments as in Lemma~\ref{lem-Koszul} apply to the complex \eqref{dual2}.  In particular, the dimensions of the $n$th graded pieces of the part of this complex relevant to the calculation of $\Ext^1_R(\kk, R)$, namely
\[
R \overset{Q^*}{\to} R[1]^{\oplus 4} \overset{P^*}{\to} R[2]^{\oplus 6},
\]
are all constant for $\tau$ on an open set $U''_n := U_n \cap U_{n+1} \cap U_{n+2}$, in the notation of the proof of Lemma~\ref{lem-Koszul}(1). The same argument as in the proof of Lemma~\ref{lem-Koszul}(2) shows that the first
homology of \eqref{dual2} in any degree $n \geq -2$ is upper-semicontinuous for $\tau\in U_n''$. Since $\Ext^1_A(\kk, A) = 0$ by Proposition~\ref{prop-A-exact}, we conclude that for $n \geq 1$ there are open subsets $W_n \subseteq U''_n$, containing $\mb 1$ and all pairs $(\rho, \theta)$ which are algebraically independent over $\mathbb{F}$, such that \eqref{dual2} is exact in the $R[1]^{\oplus 4}$ spot (in degree $n$) for $\tau \in W_n$.  In particular, $\Ext^1_R(\kk, R) = 0$ for all $\tau \in U:= V \cap  \bigcap_{n \geq -2} W_n$.

(3).  Assume that $\tau \in U$ throughout this part.  Since $\tau \in V$, $\Ext^i_R(\kk, R)$ is the $i$th homology of \eqref{dual2}.  Clearly $\Ext^4_R(\kk, R) = R[4]/\im M^* = R/(Rz_9 + Rz_{10})[4]$.  There is a relation $z_1z_9 + z_3z_{10} = 0$ by Lemma~\ref{lem-z}.  Let $g(s)$ be the Hilbert series of $R/(Rz_9 + Rz_{10})$.  Then
\[
g(s) \geq 1/(1-s)^4 - \big[ 2s/(1-s)^4 - s^2/(1-s)^4 ] = 1/(1-s)^2,
\]
and in particular $\dim_\kk \Ext^4_R(\kk, R) = \infty$.

Similarly, we have $\Ext^3_R(\kk,R) = \ker M^*/\im N^*$, and obviously $\{(a, b, 0, 0) | a, b \in R \} \subseteq \ker M^*$.  On the other hand, $\im N^* \cap \{(a, b, 0, 0) | a, b \in R \} = \{(a, b, 0,0) | a, b \in Rz_9 + Rz_{10} \}$.  Thus $\Ext^3_R(\kk, R)$ has a subfactor isomorphic to $R/(Rz_9 + Rz_{10})^{\oplus 2}[3]$; in particular, $\dim_{\kk} \Ext^3_R(\kk,R) = \infty$ also.

It follows from the exactness of \eqref{Koszul} that
\[
h_R(s) - h_{R[1]^{\oplus 4}}(s) + h_{R[2]^{\oplus 6}}(s) - h_{R[3]^{\oplus 4}}(s) + h_{R[4]}(s) = s^{-4}.
\]
This  also gives the alternating sum of the Hilbert series of the homology groups of \eqref{dual2}; in other words we must have $\sum_{i=0}^4 (-1)^i h_{\Ext^i_R(\kk, R)}(s) = s^{-4}$. Then we have $\Ext^i_R(\kk, R) = 0$ for $i = 0,1$  (since $\tau \in W$), and by the calculations above, $h_{\Ext^4_R(\kk, R)}(s) = g(s)s^{-4}$ and $h_{\Ext^3_R(\kk, R)}(s) \geq  2g(s)s^{-3}$.  It follows that $h_{\Ext^2_R(\kk, R)}(s) \geq 2g(s)s^{-3} - g(s)s^{-4}$.
Write $g(s) = \sum g_i s^i$.  Now, if $\Ext^2_R(\kk, R)$ is finite-dimensional, for $i \gg 0$ we have  $0 < 2g_i \leq g_{i+1}$. Thus $R/(Rz_9+R z_{10})$ has  exponential growth; as we know $R$ has GK-dimension 4 this is impossible.

(4). For $\tau \in V$, the failure of the AS-Gorenstein property for $R$ follows from  part (3) above.  The failure of $\chi_2$ is also immediate from part (3).  Parts (2) and (3) show that $\depth R = 2$ for $\tau \in U$.  Therefore, for such $\tau$ the Auslander-Buchsbaum formula fails for the module $M := \kk$, since $\depth \kk = 0$ and $\pd \kk = 4$.
\end{proof}

We remark that \cite[Theorem~3.2]{Jorgensen1998} shows that if $R$ is connected graded and noetherian, and the smallest nonvanishing $\Ext^i_R(\kk, R)$ is finite-dimensional, then $R$ must satisfy the Auslander-Buchsbaum property.   Since we show later in the paper that $R$ is noetherian for general $\tau$, this shows that once we know that $\Ext^2_R(\kk, R) \neq 0$, it must necessarily be infinite-dimensional.
Note that \cite[Proposition~3.5]{Jorgensen1998} also shows that any noetherian connected graded algebra that satisfies $\chi$ and has finite global dimension must be Artin-Schelter regular.  The algebras $R$ show that the $\chi$ conditions are in some sense necessary for J{\o}rgensen's results.

\section{Critical density}\label{CD}

An infinite subset $C$ of a variety $S$ is called \emph{critically dense} if every infinite subset of $C$ is Zariski dense in $S$. This property arises naturally, among other places, in the study of the noetherian property for na\"ive  blowup algebras $R(S, c, \mc{L}, \sigma)$ as in \cite{KRS}:  a necessary condition for such a na\"ive blowup algebra to be noetherian is that the point $c \in S$ being blown up lies on a critically dense orbit of the automorphism $\sigma$.  Bell et. al. have studied this condition further in \cite{BGT} and have shown the that critical density of such an orbit of an automorphism is simply equivalent to density in case $\cha \kk = 0$. (In positive characteristic, on the other hand, there are easy examples of orbits of automorphisms, even of $\mb{P}^2$, which are dense but not critically dense.)

The aim in this section is to prove that the forward $\phi^{-1}$-orbits of the special points $F$ and $Q$ (when these are defined) are critically dense subsets of $T$, when $\tau$ is  general.  As we will see in the next section, this is a necessary condition for the ring $R(\tau)$ to be noetherian, by a similar argument as in the na\"ive blowup case.  Since critical density for an orbit of a birational map has not really been studied, we will prove critical density holds for general $\tau$ more or less from scratch, using a  method similar to that
 used in \cite{R-generic}.

In the proof of critical density, the alternative coordinate system $(\ : \ )$ for $\mb{P}^1$ introduced in Section~\ref{NOTATION} is especially useful, and we use it throughout this section. We begin with a simple computation that gives the general form of the points on the forward $\phi^{-1}$-orbit of $F = (1:-1)(1:1)$; the behavior of the orbit of $Q$ is symmetric.

\begin{lemma}
\label{lem-F-orb} Let $\tau = \tau(\rho, \theta) \in \mb{G}$ be a general element of $\mb{G}$, thinking of $\rho$ and $\theta$ as parameters. There are polynomials $p_n = p_n(\rho, \theta)$, $q_n = q_n(\rho, \theta)$ in $\FF[\rho, \theta]$ such that $\phi^{-n}(F)$ is defined and equal to $(p_n:q_n)(\theta^n:1)$
 for all $(\rho, \theta)$ such that $p_n(\rho, \theta), q_n(\rho, \theta)$ are not both zero.  Moreover, $p_n = \rho^n + \theta p'_n$ and $q_n =
 -1 + \theta q'_n$ for some polynomials $p'_n, q'_n \in \FF[\rho, \theta]$.
\end{lemma}
\begin{proof}
An easy calculation using the formulas in Section~\ref{NOTATION} shows that in terms of the coordinate system $( \ : \ )$, the formula for the birational map $\phi^{-1}$ is
\[
\phi^{-1}(a:b)(c:d) = \sigma^{-1} \tau^{-1}(a:b)(c:d) = \sigma^{-1}(\rho a:b)(\theta c:d) = (\rho ad - \theta bc: bd - \rho \theta ac)(\theta c:d).
\]
Let $p_0 := 1, q_0 := -1$, and inductively define $p_{n+1} := \rho p_n - \theta^{n+1} q_n$ and $q_{n+1} := q_n - \rho \theta^{n+1} p_n$. Induction on $n$ shows that  $\phi^{-n}(F) = (p_n:q_n)(\theta^n:1)$ for all $n \geq 0$ (for $(\rho, \theta)$ such that $p_n, q_n$ are not both zero).  Clearly $p_n, q_n \in \FF[\rho, \theta]$ for all $n$.   The last claim also follows easily by induction.
\end{proof}

\begin{proposition}\label{prop-CD}
There is a  general subset $U$ of $\mb{G}$, containing $\tau(\rho, \theta)$
 for all pairs $(\rho, \theta)$ which are algebraically independent
over  $\FF$, such that for  $\tau \in U$, the points $F_n = \phi^{-n}(F)$ and $Q_n = \phi^{-n}(Q)$ are defined for all $n \geq 0$ and $\{F_n \}_{n \geq 0}$ and $\{Q_n \}_{n \geq 0}$ are critically dense subsets of $T$.
\end{proposition}
\begin{proof}
By the usual symmetry argument using Lemma~\ref{sym-lem}, it is enough to prove the claims for the point $F$ and its $\phi^{-1}$-orbit. We use the notation and the result of Lemma~\ref{lem-F-orb}.

Suppose we are given any $m \geq 0, s \geq 0$, and an increasing sequence $0 \leq n_1 < n_2 < \dots < n_N$ of $N = (m+1)(s+1)$ nonnegative integers.    An arbitrary hypersurface of degree $(m,s)$ on $T$ is the vanishing of some nonzero multi-homogeneous form $\sum_{k=0}^m \sum_{\ell = 0}^s c_{k \ell} x^ky^{m-k} z^{\ell} w^{s-\ell}.$ We totally order monomials of bidegree $(m,s)$ as follows:  we set $x^iy^{m-i} z^j w^{s-j} < x^ky^{m-k} z^{\ell} w^{s-\ell}$ if $j > \ell$ or if $j = \ell$ and $i > k$.
Let $f_1 = x^mz^s < f_2 < \dots < f_N = y^m w^s$ be the enumeration of all monomials of degree $(m,s)$ in this order. Consider the $N \times N$ matrix
\begin{equation}
\label{eq-matrix} M_{m,s, \{n_j \}} = \big(a_{ij}\big)_{1 \leq i,j \leq N} =
 \big(f_i(p_{n_j},q_{n_j}, \theta^{n_j},1)\big)_{1 \leq i,j \leq N}
\end{equation}
which has entries in $\FF[\rho, \theta]$.  We claim that $\det M_{m,s, \{n_j \}}$ is a \emph{nonzero} polynomial in $\FF[\rho, \theta]$. Supposing we have proven this claim, then note that if the points $F_n$ happen to be defined for all $n \geq 0$, but the set $\{F_n | n \geq 0 \}$ is not critically dense in $T$, there will be some hypersurface $H$ of degree $(m,s)$ and some infinite subset of $\mb{N}$, say $\{ n_1, n_2, \dots \}$, such that $F_{n_j} \in H$ for all $n_j$.  This will force $(\rho, \theta)$ to be in the vanishing set of $\det M_{m,s, \{n_j \}_{j = 1}^N}$. Moreover, by Proposition~\ref{prop-fund-pts}(1) we already know that the points $\phi^{-n}(F)$ are well-defined for all $n \geq 0$ as long as $\theta$ does not have finite order.   In conclusion, $\{F_n | n \geq 0 \}$ is a well-defined critically dense set of points as long as $(\rho, \theta)$ is not in the vanishing set of any of the countably many polynomials $\det M_{m,s, \{n_j \}}$, or contained in the countably many horizontal lines where $\theta$ is a root of unity. Let $U \subseteq \mb{G}$  be the complement of these countably many proper closed subsets.  Since all of the removed closed sets are defined over  $\FF$, any point $(\rho, \theta)$ with coordinates algebraically independent over $\FF$ must belong to $U$.

It remains to prove the claim that $D = \det M_{m,s, \{n_j \}}$ is a nonzero polynomial. For this, think of $D$ as a sum of $N !$ signed products of entries of $ (a_{ij})$.  Order monomials in $\FF[\rho, \theta]$ lexicographically with $\theta < \rho$, so $\theta^i \rho^j < \theta^k \rho^{\ell}$ if $i < k$ or if $i = k$ and $j < \ell$.  We want to consider, for each such signed product, the smallest possible monomial in this ordering occurring with nonzero coefficient.  Since we have shown in Lemma~\ref{lem-F-orb} that $p_n$ has a single term $\rho^n$ of degree $0$ in $\theta$ and $q_n$ also has a single term $-1$ of degree $0$ in $\theta$, it follows that $a_{ij} = f_i(p_{n_j},q_{n_j},\theta^{n_j},1)$ has a unique term of lowest degree in $\theta$, namely $f_i(\rho^{n_j}, -1, \theta^{n_j}, 1)$. More specifically, if $f_i = x^ky^{m-k} z^{\ell} w^{s-\ell}$, where $k = k(i)$ and $\ell = \ell(i)$, then this is $(-1)^{m-k} \rho^{k n_j} \theta^{\ell n_j}$.  So clearly $\rho^{k n_j} \theta^{\ell n_j }$ is the smallest monomial occurring in $a_{ij}$.  Now if $\chi$ is any permutation of $\{1, 2, \dots, N \}$, and $P_{\chi} = \prod_{i =1}^{N} a_{i, \chi(i)}$ is one of the products occurring in the expansion of $D$, we may calculate the smallest monomial occurring in this product by multiplying the smallest monomials occurring in each factor.  The resulting smallest monomial in $P_{\chi}$ is
\[
L_{\chi} = \prod_{i = 1}^{N} \rho^{k(i) n_{\chi(i)}} \theta^{\ell(i) n_{\chi(i)} }.
\]
Let $\chi$ be any nonidentity permutation of $\{1, 2, \dots, N \}$; so there is $i_1 < i_2$ such that $\chi(i_1) > \chi(i_2)$. Define $\chi' = \chi \circ \tau$, where $\tau = (i_1, i_2)$ is the transposition interchanging $i_1$ and $i_2$.   We show that $L_{\chi'} < L_{\chi}$. Since only the $i_1, i_2$ terms in the products $L_{\chi}, L_{\chi'}$ differ, we just need to show that
\begin{equation}
\label{needed-eq} \rho^{k(i_1) n_{\chi(i_1)}} \theta^{\ell(i_1) n_{\chi(i_1)}} \rho^{k(i_2) n_{\chi(i_2)}} \theta^{\ell(i_2) n_{\chi(i_2)}}
> \rho^{k(i_1) n_{\chi(i_2)}} \theta^{\ell(i_1) n_{\chi(i_2)}} \rho^{k(i_2) n_{\chi(i_1)}} \theta^{\ell(i_2) n_{\chi(i_1)}}.
\end{equation}
By the way the $f_i$ were enumerated, since $i_1 < i_2$, we have $\ell(i_1) \geq \ell(i_2)$ and if $\ell(i_1) = \ell(i_2)$, then $k(i_1) > k(i_2)$.  It is then straightforward to verify that \eqref{needed-eq} holds.  In particular, this implies that $L_{e}$, where $e$ is the identity permutation, is strictly smaller than the smallest monomial occurring in $L_{\chi}$ for any non-identity $\chi$.  This finishes the proof that $D$ is not identically $0$, since $L_{e}$ cannot be canceled by any other term in the expansion of $D$.
\end{proof}

Let us pause and take stock of our progress so far.   We have shown that almost all of Theorem~\ref{thm-main} holds for  general $\tau \in \GG$, as well as proving a number of additional results about the map $\phi$ and the cohomology of the sheaves $\sR_n^{\phi^m}$.  More specifically, we have:

\begin{theorem}\label{thm-generic}
Let  $(\rho, \tau)$ be a pair algebraically independent over $\FF$, and let $\tau := \tau(\rho, \theta)$.  Then $R(\tau)$ and $\sR(\tau)$ satisfy the following properties:
\begin{enumerate}
\item For any $n, m \in \NN$ and $i, j \in \ZZ$, we have $h^1(T, \sh{R}_n^{\phi^m}(i,j)) \leq h^1(T, \sh{A}_n^{\sigma^m}(i,j))$.

\item $h^1(T, \mc{R}_n^{\phi^m}) = 0$ for all $m,n \in \NN$.

\item $(R_n t^{-n})^{\phi^m} = H^0(T, \mc{R}^{\phi^m}_n)$ for all $n,m \in \NN$.

\item For any $m \in \NN$, the rational map $\phi^{-m}$ is defined  at $F$ and $Q$ and the $\phi^{-1}$-orbits of $F$ and $Q$ are infinite.

\item The set $\{ \phi^{-m}  F\}_{m \geq 0} \cup \{ \phi^{-m} Q\}_{m \geq 0}$
is a critically dense subset of $T$.

\item $\dim_\kk R_n = \binom{n+3}{3}$ and $R \cong \kk \langle x_1, x_2, x_3, x_4 \rangle/(f_1, f_2, \dots, f_6)$ where the relations $f_1,
\dots, f_6$ are as in \eqref{rels}.

\item $R$ has  left and right global dimension 4, and \eqref{Koszul} is a free resolution of $\kk_R$. \item  $R$ fails left and right
$\chi_2$. \item The Auslander-Buchsbaum property fails for $R$ on the left and the right.
\end{enumerate}
\end{theorem}
\begin{proof}
Since $R(\tau)^{op} \cong R(\tau^{-1})$ by Proposition~\ref{prop-2}, it is enough to show that for such $\tau$ each of the properties claimed for $R(\tau)$ hold individually on the right. Then (1) is Corollary~\ref{cor-semicon}, (2) is Proposition~\ref{prop-1E}, (3) follows from Proposition~\ref{prop-1E} and Lemma~\ref{lem-Koszul}(1),  (4) and (5) are Proposition~\ref{prop-CD}, and the remaining properties are Proposition~\ref{prop-Koszul}.
\end{proof}

The ring-theoretic properties of a (general) $R(\tau)$ shown  in parts (6)-(9) of the theorem are not terribly surprising, since the pathological example $A$ has all of these properties.  What is less expected is that in the general case $R(\tau)$ becomes noetherian, unlike $A$.   Proving this is the goal of the remainder of the paper.

\section{Cohomology modules}\label{COHMODULES}
For the rest of the paper, we assume that $\tau= \tau(\rho, \theta)$ where the pair $(\rho, \theta)$ is algebraically independent over $\kk$. Thus $R = R(\tau)$ and $\sh{R} = \sh{R}(\tau)$ will satisfy all of the properties in Theorem~\ref{thm-generic}.

In this and the following two sections, we prove that $R$ is noetherian and thus complete the proof of Theorem~\ref{thm-main}.  We begin with some comments on the proof strategy.

There is a method of attack that has successfully shown that many classes of birationally commutative algebras are noetherian (c.f. \cite{KRS}, \cite{RS-0}, \cite{S-idealizer}, \cite{S-surfprop}). Ultimately, this goes back to Artin and Van den Bergh's original paper \cite{AV} on twisted homogeneous coordinate rings.  Suppose that one is interested in a graded algebra $S$,  given as global sections of some quasicoherent graded sheaf $\sS \cong \bigoplus \sS_n$ on a projective scheme $X$.  Roughly speaking,  the method is as follows.  First, one puts a multiplicative structure on $\sS$ that induces the multiplication on $S$; that is, one makes $\sS$ into a {\em bimodule algebra}, as in \cite{VdB1996}.  One shows that the bimodule algebra $\sS$ is noetherian; one may think of this as saying that $S$ is noetherian at the level of geometry.  Then one shows that the sheaves $\sS_n$ form an {\em ample sequence} in the sense of \cite{VdB1996}.  This forces certain cohomology groups to vanish, and  one then applies \cite[Theorem~5.2]{VdB1997} to show that $S$ itself is noetherian.

This method fails for the algebras $R(\tau)$. As we shall see in Remark~\ref{rem-notample}, the sheaves $\sR_n$ do not form an ample sequence, and thus one cannot force cohomology to vanish.  We will see, in fact, that there are infinite-dimensional {\em cohomology modules} over $R$ that form an extremely interesting class of objects.  In this section, we will define cohomology modules, and reduce the problem of showing that $R$ is noetherian to that of showing that (particular) cohomology modules are noetherian.

We begin, however, by showing that $R$ is noetherian at the level of geometry.  This amounts to showing that there is a well-behaved correspondence between graded right ideals of $R$ and  ideal sheaves on $T$.

\begin{proposition}\label{prop-3A}
Let $J^{(1)} \subseteq J^{(2)}\subseteq \cdots $ be an ascending chain of graded right ideals of $R$.
There are a number $k \in \NN$ and an ideal sheaf $\sJ \subseteq \sI_k$  so that the sections in $J^{(\ell)}_n$ generate $ \sh{J} \Lsh_k \cdot \sh{R}_{n-k}^{\phi^k}$ for $n \geq k, \ell \gg 0$.
\end{proposition}
\begin{proof}
Let $H$ be any graded right ideal of $R$.  Let $\mc{H}_n$ be the subsheaf of the constant sheaf $\mc{K}$ generated by $H_nt^{-n}$. That $H$ is a right ideal means  that $H_m R_n \subseteq H_{m+n} \subseteq Kt^{n+m}$, and so $\mc{H}_m \mc{R}_n^{\phi^m} \subseteq \mc{H}_{m+n}$ for all $m,n \in \NN$.  Now since $H_n \subseteq R_n = H^0(T, \mc{I}_n \otimes \mc{L}_n)$, we can write $\mc{H}_n = \mc{G}_n \otimes \mc{L}_n$, where $\mc{G}_n \subseteq \mc{I}_n$ is also an ideal sheaf.  Then the condition that $H$ is a right ideal becomes $\mc{G}_m \mc{I}_n^{\phi^m} \subseteq \mc{G}_{n+m}$ for all $m,n \in \NN$. Obviously this is equivalent to the conditions $\mc{G}_m \mc{I}_1^{\phi^m} \subseteq \mc{G}_{m+1}$ for all $m \geq 0$.

We call any sequence of ideal sheaves $\{ \mc{G}_m | m \geq 0 \}$ satisfying $\mc{G}_m \subseteq \mc{I}_m$ and $\mc{G}_m \mc{I}_1^{\phi^m} \subseteq \mc{G}_{m+1}$ for all $m \geq 0$ a \emph{standard sequence}.  It is enough to prove that for any standard sequence, we have $\mc{G}_m \mc{I}_1^{\phi^m} = \mc{G}_{m+1}$ for $m \gg 0$.  For supposing we have proved this, let $\{\mc{G}_m^{(i)}\}$ be the standard sequence associated to the ideal $J^{(i)}$.  For fixed $m$, the ascending chain $\mc{G}_m^{(1)} \subseteq \mc{G}_m^{(2)} \subseteq \cdots$ stabilizes to a fixed sheaf, call it $\mc{G}_m$. Clearly $\{ \mc{G}_m \}$ is again a standard sequence. We will have $\mc{G}_m \mc{I}_1^{\phi^m} = \mc{G}_{m+1}$ for $m \geq m_0$.  Let  $\mc{J} := \mc{G}_{m_0}$.  For $\ell \gg 0$, we have $\mc{G}_{m_0}^{(\ell)} = \mc{J}$. Then the proposition holds, with this $\mc{J}$ and  $k = m_0$.

Thus we must show that   $\mc{G}_m \mc{I}_1^{\phi^m} = \mc{G}_{m+1}$ for $m \gg 0$, for any standard sequence $\{ \mc{G}_m \}$. Let $m'$ be the smallest $m$ such that $\mc{G}_m \neq 0$. By redefining $\mc{G}_m = \mc{G}_{m'}$ for $m < m'$, we obtain another standard sequence, and it is enough to prove the claim for this sequence.  Thus we may assume that $\mc{G}_0$ defines a proper subscheme $C$ of $T$. Since  $D := \{\phi^{-n}(F) | n \geq 0 \} \cup \{ \phi^{-n}(Q) | n \geq 0 \}$ is a critically dense set by the hypothesis that $\tau$ is general, $S := C \cap D$ is a finite set of points.

Let $x := \phi^{-j}(F)$ for some $j \geq 0$. In this case, using Lemma~\ref{lem-pointchains1} we have that
\[
(\mc{I}_1^{\phi^m})_x = \begin{cases} \mf{m}_x\ \text{if}\ m \geq j + 1 \\ \mc{O}_{T, x}\ \text{if}\ m \leq j \end{cases}
\]
where we write $\mf{m}_x$ for the maximal ideal of $\mc{O}_{T, x}$. It follows similarly that $(\mc{I}_m)_x = \mf{m}_x^{m - j- 1}$ for $m \geq j + 1$, while $(\mc{I}_m)_x = \mc{O}_{T,x}$ for $m \leq j$. Similar formulas obviously hold if $x = \phi^{-j}(Q)$ for some $j \geq 0$. In particular, we deduce from the equation $\mc{G}_m \mc{I}_1^{\phi^m} \subseteq \mc{G}_{m+1}$ that the scheme defined by $\mc{G}_m$ is supported on the set $C \cup D$
 for all $m \geq 0$.

Let $y \in T$.  We  study the local behavior of the standard sequence at $y$.  We know that
\begin{equation}
\label{eq-stand} (\mc{G}_m)_y \subseteq (\mc{I}_m)_y, \ \ \ \text{and}\ (\mc{G}_m \mc{I}_1^{\phi^m})_y \subseteq (\mc{G}_{m+1})_y.
\end{equation}
We now consider cases.  Suppose that $y \not \in C \cup D$. In this case, \eqref{eq-stand} specializes to $(\mc{G}_m \mc{I}_1^{\phi^m})_y = \mc{O}_{T, y}= (\mc{G}_{m+1})_y$ for all $m \geq 0$.  Next, suppose that $y \in D \smallsetminus C$, say $y = \phi^{-j}(F)$ (the case of a point on the orbit of $Q$ is similar). Then \eqref{eq-stand} again specializes to $(\mc{G}_m)_y \mc{O}_{T,y} \subseteq (\mc{G}_{m+1})_y$ for $0 \leq m \leq j$, and it specializes to $(\mc{G}_m)_y \mf{m}_y \subseteq (\mc{G}_{m+1})_y \subseteq \mf{m}_y^{m-j}$ for $m \geq j+1$. Since $y \not \in C$, we will have  $(\mc{G}_0)_y = \mc{O}_{T, y}$.  Now induction on $m$ using the equations above shows that $(\mc{G}_m)_y = \mc{O}_{T, y}$ for $0 \leq m \leq j+1$, and $(\mc{G}_m)_y = \mf{m}^{m-j-1}$ for $m \geq j+2$.  In particular, all inclusions are equalities above and $(\mc{G}_m)_y (\mc{I}_1^{\phi^m})_y = (\mc{G}_{m+1})_y$ holds for all $m \geq 0$.

Next, let $y \in U := C \smallsetminus D$.  Since $C \cap D$ is finite, $U$ is an open subscheme of $C$.  Specializing \eqref{eq-stand}, we obtain
 $(\mc{G}_m)_y \mc{O}_{T,y} \subseteq
(\mc{G}_{m+1})_y$, since $y \not \in D$.  In other words, the sequence $\{ \mc{G}_m \vert_U \}$ gives an ascending chain of ideal sheaves on $U$. Since $U$ is a noetherian scheme, $(\mc{G}_m)_y (\mc{I}_1^{\phi^m})_y = (\mc{G}_m)_y = (\mc{G}_{m+1})_y$ for all $y \in U$ and $m \geq m_1$, for some $m_1$.

The finitely many points in  $ S=C \cap D$ are left. Suppose that $y \in C \cap D$, say $y = \phi^{-j}(F)$.
  Then for $m \geq j+1$, \eqref{eq-stand} says
$(\mc{G}_m)_y \subseteq  \mf{m}_y^{m - j- 1}$ and $(\mc{G}_m)_y \mf{m}_y \subseteq (\mc{G}_{m+1})_y$. These conditions can be reinterpreted as follows: $\bigoplus_{i = 0}^{\infty} (\mc{G}_{i + j + 1})_y$ is a graded ideal of the Rees algebra $\mc{O}_{T, y} \oplus \mf{m}_y \oplus \mf{m}_y^2 \oplus \dots$.  This Rees algebra is noetherian, so the ideal is finitely generated. This means precisely that $(\mc{G}_m \mc{I}_1^{\phi^m})_y = (\mc{G}_{m+1})_y$ for $m \gg 0$.  The case that $y = \phi^{-j}(Q)$ is similar. Repeating finitely many times,
 we conclude there is a single $m_2$ such that $(\mc{G}_m
\mc{I}_1^{\phi^m})_y = (\mc{G}_{m+1})_y$ for $m \geq m_2$, and for all $y \in S$.  Thus $\mc{G}_m \mc{I}_1^{\phi^m} = \mc{G}_{m+1}$ for $m \geq \max(m_1, m_2)$, and the claim is proved.
\end{proof}

It is not immediately clear from Proposition~\ref{prop-3A} that $R$ is right noetherian.  As we will see, the obstruction lies in the cohomology of sheaves of the form $\sh{F} \otimes \sh{R}_n^{\phi^m}$.   To begin to analyze this issue, we make some definitions.  For $m \geq 0$, let $\sR^{\phi^m} := \bigoplus_n \sR_n^{\phi^m}$.  For $m, n, \ell \in \NN$, let $\nu_{n, \ell}^m:  \sR_n^{\phi^m} \otimes \sR_{\ell}^{\phi^{m+n}} \to \sR_{n+\ell}^{\phi^m}$ be the natural multiplication on $\sR^{\phi^m}$ induced by the embeddings of these sheaves in the constant sheaf $\mc{K}$.

\begin{defn}
Suppose that $\sG:= \bigoplus_{n \in \NN} \sG_n$ is a quasicoherent sheaf on $T$, and that for all $n,  \ell \in \NN$ there are action maps $\mu_{n, \ell}: \sG_n \otimes \sR_{\ell}^{\phi^{n+m}} \to \sG_{n+\ell}$ so that the diagram \beq\label{assoc}
 \xymatrix{
 \sG_n \otimes \sR_{\ell}^{\phi^{n+m}} \otimes \sR_k^{\phi^{n+\ell+m}}
            \ar[r]_{\mu_{n, \ell}\otimes 1} \ar[d]_{1 \otimes \nu_{\ell, k}^{n+m}}
 & \sG_{n+\ell}\otimes \sR_k^{\phi^{n+\ell+m}} \ar[d]^{\mu_{n+\ell, k}} \\
\sG_n\otimes \sR_{\ell+k}^{\phi^{n+m}} \ar[r]_{\mu_{n,\ell+k}} &   \sG_{n+\ell+k}} \eeq commutes for all $n, \ell, k \in \NN$.  Then we call $\mc{G}$ an \emph{$\mc{R}^{\phi^m}$-module}.  For any $\sR^{\phi^m}$-module $\sG$ and $0 \leq i \leq 2$ we call $H^i(T, \sG)$ a \emph{cohomology module}.

An important special case is $\sG = \sh{F} \otimes \mc{R}^{\phi^m}$ for some quasi-coherent sheaf $\sh{F}$, with the maps $\mu_{n, \ell}$ induced by the multiplication maps on $\sR^{\phi^m}$. In this case we use the special notation
\[ \CM^i(\sh{F}, m) := \bigoplus_{n \in \NN} H^i(T, \sh{F} \otimes \sh{R}_n^{\phi^m})
\]
for the cohomology module $H^i(T, \mc{G})$.
\end{defn}

We must show that cohomology modules do in fact have an $R$-action, as the name suggests.  We prove this and other important formal properties of this construction in the next result.

\begin{lemma}\label{lem-cohom-module}
Let $m \in \NN$.

$(1)$ For any $\mc{R}^{\phi^m}$-module $\sG$, there is an $R$-module action on $H^i(T, \sG)$ induced by the maps $\mu_{n, \ell}$.

$(2)$ Let $\sG$ be an $\mc{R}^{\phi^m}$-module  such that for all $n > 0$ the map $\mu_{0, n}:  \sG_0 \otimes \sR_n^{\phi^m} \to \sG_n$ is surjective with 0-dimensional kernel.  Then $\mu_{0, \bullet}$
induces a surjective map $\CM^i(\sG_0, m) \to H^i(T, \sG)$ of $R$-modules that is an isomorphism for $i \geq 1$.

$(3)$ $\CM^i(\blank, m)$ is a functor from $\struct_X \lMod$ to $\rGr R$.  Moreover, for an exact sequence $0 \to \sF  \to \sG \to \sH \to 0$ of quasicoherent sheaves on $T$, there is a long exact sequence of $R$-modules
\begin{multline*}
 \CM^0(\sF, m) \to \CM^0(\sG, m) \to \CM^0(\sH, m) \to
 \CM^1(\sF, m) \to \CM^1(\sG, m) \to \CM^1(\sH, m) \to \\
  \CM^2(\sF, m) \to \CM^2(\sG, m) \to \CM^2(\sH, m) \to 0.
\end{multline*}

\end{lemma}
\begin{proof}
The proof of this lemma is routine, and so we leave some details to the reader.

$(1)$.  Fix $m$, $n, \ell \in \NN$.  Now,  $\mc{R}_{\ell}^{\phi^{n+m}}$ is globally generated,
and by Theorem~\ref{thm-generic}(3) we have $H^0(T, \sR_\ell^{\phi^{n+m}}) = (R_\ell t^{-\ell})^{\phi^{n+m}}$.  We write (loosely) $R_{\ell}^{\phi^{n+m}} = H^0(T, \sR_\ell^{\phi^{n+m}})$.  There is thus
  a surjective map \beq\label{globalsect}
 \struct_T \otimes R_{\ell} \to \struct_T \otimes R_{\ell}^{\phi^{n+m}} \to  \sh{R}_{\ell}^{\phi^{n+m}}.
\eeq Tensor \eqref{globalsect} with $\sh{G}_n$ and follow this by the multiplication map $\mu_{n, \ell}: \sG_n \otimes \sR_{\ell}^{\phi^{n+m}} \to \sG_{n+\ell}$; then applying $H^i(T, \blank)$ gives a map $H^i(T, \sG_n) \otimes R_{\ell} \to H^i(T, \sG_{n+\ell})$ which provides the desired $R$-action on $H^i(T, \sG)$.  Associativity of this action follows from \eqref{assoc}.

$(2)$.   Consider the action map $\mu_{0, \bullet}:  \sG_0 \otimes \sR^{\phi^m} \to \sG$; applying $H^i(T, \blank)$ induces the map
\[ \CM^i(\sG_0, m) = H^i(T, \sG_0 \otimes \sR^{\phi^m}) \to H^i(T, \sG). \]
It is an $R$-module map by a diagram chase, using \eqref{assoc} again.  From the long exact sequence in cohomology and our assumption on the kernel of $\mu_{0,n}$, we deduce that this map is surjective for $i \geq 0$ and is an isomorphism for $i \geq 1$.

$(3)$.  Given an exact sequence $0 \to \mc{F} \to \mc{G} \to \mc{H} \to 0$ of sheaves, there are exact sequences
\begin{equation}
\label{K-seq1} 0 \to \mc{K}_n \to \mc{G} \otimes \mc{R}_n^{\phi^m} \to \mc{H} \otimes \mc{R}_n^{\phi^m} \to 0\ \  \text{and}  \ \ 0 \to \sK'_n \to \sF \otimes \sR_n^{\phi^m} \to \sK_n \to 0,
\end{equation}
where the sheaves $\sK'_n$ have 0-dimensional support.

Since $\theta: \bigoplus_{n \geq 0} \mc{G} \otimes \mc{R}_n^{\phi^m} \to \bigoplus_{n \geq 0} \mc{H} \otimes \mc{R}_n^{\phi^m}$ is a morphism of $\mc{R}^{\phi^m}$-modules (in other words, $\theta$ commutes with the multiplication maps in the obvious sense), it is routine to check that $\ker \theta = \mc{K} = \bigoplus_{n \geq 0} \mc{K}_n$ obtains an induced $\mc{R}^{\phi^m}$-module structure, and so $H^i(T, \sK)$ is a right $R$-module.

For any $\ell \geq 0$, the multiplication map $\sR_n^{\phi^m} \otimes R_{\ell}^{\phi^{n+m}} \to \sR_{n+\ell}^{\phi^m}$ induces a morphism of exact sequences \beq\label{worldcup} \xymatrix{ 0 \ar[r]  & \sK_n \otimes R_{\ell}^{\phi^{n+m}} \ar[r]  \ar[d] &
 \sG \otimes \sR_n^{\phi^m} \otimes R_{\ell}^{\phi^{n+m}} \ar[r] \ar[d] &
 \sH\otimes \sR_n^{\phi^m} \otimes R_{\ell}^{\phi^{n+m}} \ar[d] \ar[r] & 0 \\
0 \ar[r] & \sK_{n+\ell} \ar[r] &
 \sG \otimes \sR_{n+\ell}^{\phi^{m}} \ar[r] &
 \sH \otimes \sR_{n+\ell}^{\phi^{m}} \ar[r] & 0.   }
\eeq
Consider the morphism of long exact sequences in cohomology induced from \eqref{worldcup}, which begins
\[ \xymatrix@C=6pt{
 0 \ar[r] & H^0(T, \sK_n)\otimes R_{\ell}^{\phi^{n+m}} \ar[r] \ar[d] &
H^0(T, \sG\otimes\sR_n^{\phi^m})\otimes R_{\ell}^{\phi^{n+m}} \ar[r] \ar[d] & H^0(T, \sH \otimes \sR_n^{\phi^{m}}) \otimes R_{\ell}^{\phi^{n+m}} \ar[r]\ar[d] &
 \cdots \\
0 \ar[r] & H^0(T, \sK_{n+\ell}) \ar[r] & H^0(T, \sG\otimes \sR_{n+\ell}^{\phi^m}) \ar[r] & H^0(T, \sH \otimes \sR_{n+\ell}^{\phi^m}) \ar[r]  & \cdots.}
\]
Because this diagram commutes, the cohomology long exact sequence \beq \label{LES} 0 \to H^0(T, \sK) \to \CM^0(\sG, m) \to \CM^0(\sH, m) \to H^1(T, \sK) \to \cdots. \eeq
 that we obtain by taking $\ell=0$ and summing over $n$ is in fact a long exact sequence of $R$-modules.

Note that $\sK_0 = \sF$. Now, from $\eqref{K-seq1}$ and part (2) of the lemma, we obtain for all $i \geq 0$ a surjective map of $R$-modules $\CM^i(\sF, m) \to H^i(T, \sK)$, which is an  isomorphism for $i \geq 1$.  Combining these maps with \eqref{LES},  we obtain the desired long exact sequence of $R$-modules.  Functoriality of $\CM^i(\blank, m)$ easily follows also.
\end{proof}

Cohomology modules allow us to make an important reduction.
\begin{proposition}\label{prop-3C}
To show that $R$ is right noetherian, it is enough to show that all cohomology modules $\CM^1(\sh{F}, m)$ are noetherian, where $\sh{F}$ is a coherent sheaf on $T$ and $m \in \NN$.
\end{proposition}
\begin{proof}
Suppose that $\CM^1(\sh{F},m)$ is noetherian for all $\sh{F}, m$.  Let $J^{(1)} \subseteq J^{(2)}\subseteq \cdots$ be an ascending chain of graded right ideals of $R$.  By Proposition~\ref{prop-3A}, there are an ideal sheaf $\sh{J}$ and  integers $m$ and $\ell_0$ so that for $n \geq m$ and $\ell \geq \ell_0$, the sections in $J^{(\ell)}_n$ generate $ \sh{J} \Lsh_m \cdot \sh{R}_{n-m}^{\phi^m}$. Without loss of generality, we may assume that $\ell_0 = 1$.

In particular, for all $j$, the sections in $J^{(j)}_m$ generate $\sh{H}:= \sh{J} \Lsh_m \subseteq \sh{R}_m.$ Clearly, it is enough to show that the chain $J^{(1)}_{\geq m} \subseteq J^{(2)}_{\geq m} \subseteq \cdots$ stabilizes, so we may assume that all $J^{(i)}$ are contained in the right ideal
\[ H := \bigoplus_{n \geq m} H^0(T, \sh{H} \sh{R}_{n-m}^{\phi^m}) \subseteq R_{\geq m}.\]

Let $H'$ be the shifted cohomology module
\[ H' := \CM^0(\sh{H}, m)[-m] = \bigoplus_{n \geq m} H^0(T, \sh{H} \otimes \sh{R}_{n-m}^{\phi^m}).  \]
By Lemma~\ref{lem-cohom-module}(2), there is a surjection $H'\to H$; since $\sR^{\phi^m}_0 = \sO_T$, this map is an isomorphism in degree $m$. Let $V$ be the preimage of $J^{(1)}_m$ in $H'_m$.  Since $H/J^{(1)}_{\geq m}$ is a factor of $H'/VR$, it suffices to show that $ H'/VR$ is noetherian.

Recall that $V$ generates $\sH$. Consider the exact sequence  $0 \to \sh{F} \to V \otimes \struct_T \to \sh{H} \to 0$, for the appropriate $\sh{F}$. By Lemma~\ref{lem-cohom-module}(3), there is a long exact sequence of cohomology modules that reads in part:
\[ \xymatrix{
V \otimes \CM^0(\sO_T,m) \ar[r]^f & \CM^0(\sH,m) \ar[r]^d & \CM^1(\sF,m) \ar[r] & V \otimes \CM^1(\sO_T, m).} \] By choice of $\tau$, $H^1(T, \sR_n^{\phi^m}) = 0$ for all $n \in \NN$, and the last term of the exact sequence above vanishes. The first term is $V \otimes R^{\phi^m}$. The cohomology module $\CM^0(\sH, m)$ is $H'[m]$.  The  connecting homomorphism $d$  thus induces an injection from $\bigl(H'/VR \bigr) [m]$ into $\CM^1(\sF, m)$.
Since $\CM^1(\sF, m)$ is noetherian by assumption, so is $H'/VR$.
\end{proof}

We observe that the cohomology modules $\CM^2(\sF, m)$ are easily seen to be finite-dimensional.
\begin{lemma}\label{lem-4}
For any coherent $\sh{F}$ and $m \in \NN$, the cohomology module $\CM^2(\sh{F}, m)$ is finite-dimensional.
\end{lemma}
\begin{proof}
Fix $\sh{F}$ and $m$.  There is a natural map $\sh{F} \otimes \sh{R}_n^{\phi^m} \to \sh{F} \otimes \Lsh_n^{\phi^m}$, whose kernel and cokernel have 0-dimensional support. In particular, taking cohomology we obtain that
\[ \CM^2(\sh{F}, m)_n = H^2(T, \sh{F} \otimes \sh{R}_n^{\phi^m}) \cong H^2(T, \sh{F} \otimes \Lsh_n^{\phi^m}).\]

Let $\mc{P}$ be a finite direct sum of invertible sheaves so that there is a surjection $\mc{P} \thra \mc{F}$. Then $H^2(T, \mc{P} \otimes \Lsh_n^{\phi^m})$ surjects onto $H^2(T, \mc{F} \otimes \Lsh_n^{\phi^m}) \cong \CM^2(\sh{F}, m)_n.$  Now, $H^2(T, \mc{P} (a,b)) = 0$ for all $a, b \gg 0$. Recall that $\sL_n^{\phi^m} \cong \sO(n, mn+\coeff)$.  As $n \to \infty$ so does $mn + \coeff$.
 Therefore, $H^2 (T, \mc{P}\otimes \Lsh_n^{\phi^m}) = 0$ for $n \gg 0$, and
$\CM^2(\sh{F}, m)$ is finite-dimensional, as claimed.
\end{proof}

Let $p: T \to \PP^1$ be projection onto the 2nd factor.  Our next goal is to show that the Leray spectral sequence associated to $p$ induces a decomposition of a cohomology module.
\begin{proposition}\label{prop-5}
Fix a coherent sheaf $\sh{F}$ on $T$ and $m \in \NN$.  Then there are natural $R$-actions on \beq\label{eq-a}
 \KM(\sh{F}, m) := \bigoplus_{n \in \NN} H^1(\PP^1, p_* (\sh{F} \otimes \sh{R}_n^{\phi^m}))
\eeq and \beq\label{eq-b} \QM(\sh{F}, m):= \bigoplus_{n \in \NN} H^0(\PP^1, R^1p_* (\sh{F} \otimes \sh{R}_n^{\phi^m})). \eeq
Further, there is a natural exact sequence of $R$-modules,
\[
0 \to \KM(\sh{F},m) \to \CM^1(\sh{F}, m) \to \QM(\sh{F}, m) \to 0 .
\]
\end{proposition}
\begin{proof}
The $R$-action on \eqref{eq-a} (respectively, on \eqref{eq-b}) is given by applying $H^1(\PP^1, p_* \blank)$ (respectively, $H^0(\PP^1, R^1 p_* \blank)$) to the multiplication map
\begin{equation}
\label{eq-mm} \sh{F} \otimes \sh{R}_n^{\phi^m} \otimes R_{\ell} \to \sh{F} \otimes \sh{R}_n^{\phi^m} \otimes R_{\ell}^{\phi^{n+m}} \to \sh{F} \otimes \sh{R}_{n+\ell}^{\phi^m},
\end{equation}
as in the proof of Lemma~\ref{lem-cohom-module}(1). By \cite[5.8.6]{Weibel}, for any quasi-coherent sheaf $\sh{M}$ on $T$, there is a convergent Leray spectral sequence $H^i(\PP^1, R^j p_* \sh{M}) \Rightarrow H^{i+j}(T, \sh{M})$. Further, by \cite[Theorem~5.8.3]{Weibel},  the exact sequence of low degree terms  is \beq\label{LSS}
 0 \to H^1(\PP^1,p_* \sh{M}) \to H^1(T, \sh{M}) \to H^0(\PP^1, R^1p_*\sh{M}) \to H^2(\PP^1, p_* \sh{M}) = 0,
\eeq
 and the maps in this exact sequence are natural in $\sh{M}$.

Fix $n, \ell \in \NN$.  We apply the exact sequence \eqref{LSS} to the multiplication map \eqref{eq-mm}. By naturality, the diagram
\[  \xymatrix@C=8pt{
0 \ar[r] & H^1(\PP^1, p_*(\sh{F} \otimes \sh{R}_n^{\phi^m})) \otimes R_{\ell} \ar[r] \ar[d] & H^1(T, \sh{F} \otimes \sh{R}_n^{\phi^m}) \otimes R_{\ell} \ar[r] \ar[d] &
H^0(\PP^1, R^1p_*(\sh{F} \otimes \sh{R}_n^{\phi^m})) \otimes R_{\ell} \ar[r] \ar[d] & 0 \\
0 \ar[r] & H^1(\PP^1, p_*(\sh{F} \otimes \sh{R}_{n+\ell}^{\phi^m})) \ar[r] & H^1(T, \sh{F} \otimes \sh{R}_{n+\ell}^{\phi^m}) \ar[r] & H^0(\PP^1, R^1p_*(\sh{F} \otimes \sh{R}_{n+\ell}^{\phi^m})) \ar[r] & 0 }
\]
 commutes.  This precisely says that the maps
\[ 0 \to \KM(\sh{F}, m) \to \CM^1(\sh{F}, m) \to \QM(\sh{F},m) \to 0\]
given by \eqref{LSS} preserve the $R$-module structure.
\end{proof}

Note that naturality of \eqref{LSS} also implies that $\QM(\blank, m)$ and $\KM(\blank, m)$ are functors from $\sO_T \lMod \to \rGr R$.

The strategy of the remainder of the proof that $R$ is noetherian will be to verify the hypotheses in the following corollary, which we do in the final two sections of the paper.
\begin{corollary}\label{cor-Leray}
To show that $R$ is right noetherian, it suffices to show that the modules $\KM(\struct(a,b), m)$ and $\QM(\struct(a,b), m)$ defined above are noetherian for all $a,b \in \ZZ$ and $m \in \NN$.
\end{corollary}
\begin{proof}
Since any invertible sheaf is isomorphic to some $\struct(a,b)$, the hypothesis together with Proposition~\ref{prop-5} shows that $\CM^1(\sh{H}, m)$ is noetherian for any invertible sheaf $\mc{H}$.

By Proposition~\ref{prop-3C}, it is enough to show that for any coherent $\sh{F}$ and $m \in \NN$, the cohomology module $\CM^1(\sh{F}, m)$ is noetherian. There is an exact sequence $0 \to \sh{F}' \to \sh{H} \to \sh{F} \to 0$ where $\sh{H}$ is isomorphic to a direct sum of invertible sheaves on $T$.  By Lemma~\ref{lem-cohom-module}(3),  there is an exact sequence of $R$-modules
\[ \CM^1(\sh{H}, m) \stackrel{\alpha}{\to} \CM^1(\sh{F}, m) \to \CM^2(\sh{F}', m).\]
By Lemma~\ref{lem-4}, the cokernel of $\alpha$ is finite-dimensional and is thus noetherian.  Since $\CM^1(\sh{H}, m)$ is noetherian by the first paragraph, $\CM^1(\sh{F}, m)$ is noetherian.
\end{proof}

\section{$\QM(\struct(a,b),m)$ is noetherian}\label{Q-MODULES}
In this section, we calculate the modules $\QM(\struct(a,b), m)$ and show that they are noetherian.  In fact, we given even more details of their structure:  these modules are finite extensions of point modules.

We continue to assume that $\tau = \tau(\rho, \theta) \in \mb{G}$ where the pair $(\rho, \theta)$ is algebraically independent over the prime subfield $\FF$, so that all of the properties in Theorem~\ref{thm-generic} hold. For  $j \geq 0$, recall that $F_j = \phi^{-j}(F)$ and $Q_j = \phi^{-j}(Q)$.  As in the last section, we write $p: T \to \mb{P}^1$ for the projection of $T$ onto the second factor.  Let $q_j := p(Q_j)$ and $f_j:= p(F_j)$, and write $f:=f_0$ and $q:=q_0$. Recall also  that $\sh{I}_n^{\phi^m}$ is defined as the base ideal of the subsheaf of $\Lsh_n^{\phi^m}$ generated by the rational functions in $(R_n t^{-n})^{\phi^m}$. We saw in Proposition~\ref{prop-0dim} that $\sh{I}_n^{\phi^m}$ defines the fat point subscheme that we may write as \beq\label{Inm-eq}
 n F_0 +  \cdots + n F_{m-1} + (n-1) F_m + \cdots +  F_{n+m-2} +
 n Q_0 +  \cdots + n Q_{m-1} + (n-1) Q_m + \cdots +  Q_{n+m-2}.
 \eeq
By choice of $\tau$, the points $F_j$ and $Q_i$ all lie on distinct fibers of $p$. We need a notation for a general (fat) fiber of $p$.  If $c \in \PP^1$ and $\ell\geq 1$, define $\ell T_c$ by the fiber square
\[ \xymatrix{
\ell T_c \ar[r] \ar[d] & T \ar[d]^{p} \\
\ell c \ar[r] & \PP^1. }\]

The main idea of this section is to reduce the calculation of $\QM(\struct(a,b), m)$ to the calculation of $\CM^1(\sh{F},m)$ for certain sheaves $\sh{F}$ supported entirely on a single fiber of $p$.  These latter cohomology modules are then computed directly with \v{C}ech cohomology and shown to be noetherian.  The result of this computation is given in the following main technical lemma.   The proof is somewhat sensitive because of the need to carefully track the $R$-action on the cohomology, and so we defer it until the end of the section.

\begin{lemma}
\label{lem-miracle} Let $\ell \geq 1$, $d \geq 0$, $a, b \in \ZZ$. For $n \geq \max(d, 1)$, let
\[ \sH(\ell)_n := \sI_Q^{n-d} \cdot \sO_{\ell Z}(aX+bW) \otimes \sL_n.\]
This is an $\mc{R}$-module since $\sH(\ell)_n\cdot \sR_k^{\phi^n} \subseteq \sH(\ell)_{n+k}$, and so
$H(\ell) := \bigoplus_{n \geq \max(d,1)} H^1(T, \sH(\ell)_n)$ is an $R$-module by Lemma~\ref{lem-cohom-module}(1).

Let $\ell_0 := \max(1, -a-1)$; let $n_0:=max(d, 1, -a-1)$.   For all $n \geq n_0$, $\ell \geq \ell_0$,  the natural restriction map $H^1(T, \mc{H}(\ell+1)_n) \to H^1(T, \mc{H}(\ell)_{n})$ is an isomorphism, and multiplication by $t \in R_1$ gives a bijection $\mu_t: H(\ell)_n \to H(\ell)_{n+1}$.

Moreover,  we have
\[
H^1(T, \mc{H}(\ell)_{n}) \cong  \begin{cases} 0 &  a \geq -d-1 \\
\struct_{q} \oplus \sO_{2q} \oplus \cdots \struct_{(-a-d-1)q} \ & a \leq -d-2.
\end{cases}
\]
(Note that since $\mc{H}(\ell)_{n}$ is supported along the fat fiber $\ell Z$, $H^1(T, \mc{H}(\ell)_{n}) = H^1(\ell Z, \mc{H}(\ell)_n) $ obtains an $\mc{O}_{\ell q}$-module structure from the base.)
In particular, $\dim_{\kk} H(\ell)_n = \binom{-a-d}{2}$ is constant for $n \geq n_0$, $ \ell \geq \ell_0$. Furthermore, $H(\ell)$ is noetherian and is, up to finite dimension, an extension of $\binom{-a-d}{2}$ point modules.
\end{lemma}

A symmetric result holds for the sheaves
$\sH^{\vee}(\ell)_n := \sI_F^{n-d} \cdot \sO_{\ell W} (aX+bW) \otimes \sL_n$, and we will use this without further comment.

We now note some consequences of the lemma above.  First, some special cohomology modules are noetherian.
\begin{lemma}
\label{lem-Nnoeth} Let $c=q_k$ or $c=f_k$ for some $k \in \NN$.  Let $ a, b \in \ZZ$.  For $\ell \geq 1$ let $\sF_{\ell}:= \sO_T(a,b)|_{\ell T_c}$.  Let $m \in \NN$.  Then
the cohomology module $\CM^1(\mc{F}_{\ell}, m)$ is noetherian and is, up to finite dimension, an extension of finitely many point modules.
\end{lemma}
\begin{proof}
We do the case that $c=q_k$; the case that $c=f_k$ is symmetric.    The idea is to show that the cohomology module in question can be shifted twice to obtain a tail of a module already studied in Lemma~\ref{lem-miracle}.  One must be careful, as it seems problematic to shift $\mc{R}^{\phi^m}$-modules in general, though it works for the special cases considered here; c.f. \cite[Lemma~5.5]{KRS}.

Fix $\ell \geq 1$, and let $\sG_n := \sO_{\ell T_c}(a,b) \cdot \sR^{\phi^m}_n$.  Then $\sG := \bigoplus_{n \in \NN} \sG_n$ is clearly an $\mc{R}^{\phi^m}$-module, and by Lemma~\ref{lem-cohom-module}(2),
\[ N := \bigoplus_{n \in \NN} H^1(T, \sG_n)\]
is an $R$-module isomorphic to $\CM^1(\sF_{\ell}, m)$.

Now, by \eqref{Inm-eq}, $\sI_n^{\phi^m}$ vanishes at $Q_k$ to order
\begin{equation}
\label{mult-cases-eq} \begin{cases}
0 & \text{if}\  k \geq n+m-1 \\
 n+m-k-1 & \text{if}\ m-1 \leq k \leq n+m-1 \\
 n  & \text{if}\ k \leq m-1.
 \end{cases}
 \end{equation}
For $r = \max(m, k+1)$,  we have for $n \geq 0$ that
\[ \sG_n \cong \sI_{Q_k}^{n+m-r} \cdot \sO_{\ell T_{c}}(a,b)\otimes \sL_n^{\phi^m}
 \cong \sI_{Q_k}^{n+m-r}\cdot \sO_{\ell T_{c}}(a,b) \otimes \sL_m^{-1} \otimes \sL_{n+m}.\]
For appropriate $a', b'$, there is thus an isomorphism $\sG_n \cong \sI_{Q_k}^{n+m-r}\cdot \sO_{\ell T_{c}}(a',b') \otimes \sL_{n+m}$ that respects the action by $\sR^{\phi^m}$ on each side.   Shifting by $(-m)$ and applying cohomology, we get an isomorphism
\[
\bigoplus_{n \in \NN} H^i(T, \mc{G}_n)[-m] \cong \bigoplus_{n \geq m} H^i(T, \mc{G}_{n-m})
\]
of $R$-modules.  In particular, letting $\mc{H}_n := \sI_{Q_k}^{n-r}\cdot  \sO_{\ell T_c} (a', b')\otimes \sL_n$, we have
\[ N[-m]_{\geq r} \cong \bigoplus_{n \geq r} H^1(T, \sH_n). \]

Next, given our assumption that $\tau$ is general, $\phi^{-k}:  T \dra T$ is defined and a local isomorphism from a neighborhood $U$ of $Z$ to a neighborhood $\phi^{-k}(U)$ of $\phi^{-k}(Z) = T_c$. The $\mc{R}$-module $\bigoplus  \sH_n$ has action maps $\mu_{n, \ell}: \mc{H}_n \otimes \mc{R}_{\ell}^{\phi^n} \to \mc{H}_{n+\ell}$ involving sheaves supported along $\phi^{-k}(U)$.  Thus we can pull back these action maps by $\phi^{-k}$ and reindex to get
$ \mu'_{n, \ell}: (\mc{H}_{n+k}^{\phi^{-k}}  \otimes  \mc{R}_{\ell}^{\phi^{n}}) \to \mc{H}_{n + k + \ell}^{\phi^{-k}}$. These action maps make $\mc{H}' = \bigoplus_{n \in \NN}  \mc{H}_{n+k}^{\phi^{-k}}$ into another $\mc{R}$-module. Explicitly, since $\mc{H}_n = \sI_{Q_k}^{n-r} \cdot \sO_{\ell T_c} (a', b') \otimes \sL_n$, we have an isomorphism
\[
\mc{H}_{n+k}^{\phi^{-k}} \cong (\sI_{Q_k}^{n-r+k} \cdot \sO_{\ell T_c}(a', b' ) \otimes \sL_{n+k})^{\phi^{-k}} \cong
 \sI_Q^{n-r+k} \cdot \sO_{\ell Z}(a'', b'') \otimes \sL_n
\]
for the appropriate $a'', b''$. Since $\phi^{-k}$ is an isomorphism on $U$, pulling back by $\phi^{-k}$  induces an isomorphism in cohomology of any sheaf supported on $\phi^{-k}(U)$.  Thus the pullback by $\phi^{-k}$ induces a bijection $H^1(T, \mc{H})[k]_{\geq 0} \cong H^1(T, \mc{H}')$, and this is an isomorphism of $R$-modules since the pullback commutes with the action maps by construction. Combining this with the first shift we calculated, we get
\[
N[-m+k]_{\geq r-k} \cong  \bigoplus_{n \geq r-k} H^1(T, \sI_Q^{n-r+k} \cdot \sO_{\ell Z}(a'', b'') \otimes \sL_n),
\]
where the right hand side is a tail of a module considered in Lemma~\ref{lem-miracle}.  Recalling that $N  \cong \CM^1(\mc{F}_{\ell}, m)$, the result thus follows immediately from Lemma~\ref{lem-miracle}, with $d = r-k >0$, $a = a''$.  \end{proof}

Another consequence of Lemma~\ref{lem-miracle} is that it tells us $R^1p_*$ of certain twists of a fat point on $T$.  We record this as:
\begin{lemma}\label{lem-pushfwd1}
Let $\sh{I}$ be the ideal sheaf of the point $z \in T$, and let $a, b \in \ZZ$.  Let $g = p(z)$.    If
 $a\leq -2$ and $n \geq -a-1$,  then
\[ R^1p_* \sh{I}^n (n+a,b) \cong \struct_{(-a-1) g} \oplus \cdots \struct_g.
\]
Further, for $\ell, n \geq -a-1$ the natural map
\[ \gamma: R^1p_* \sI^n(n+a,b) \to R^1p_* (\sI^n(n+a, b)|_{\ell T_g}) \]
is an isomorphism.

If $a\geq -1$, then $R^1 p_* \sh{I}^n(n+a,b) = 0$ for all $n \geq 0$.
\end{lemma}
\begin{proof}
The result of the lemma does not depend on the choice of point $z$, so we may assume that $z = Q$. Let $U:=  \mb{P}^1 \smallsetminus \{q\}.$ On $V := p^{-1}(U)$ we have
 $\sh{I}^n(n+a,b) \vert_V \cong \sO(n+a, b) \vert_V$.
Since $R^1 p_*$ is local on the base, we have
\[
\bigl(R^1p_* \sO(n+a, b) \bigr) \vert_U \cong H^1(V, \mc{O}(n + a, b) \vert_{V}) \cong H^1(\mb{P}^1_U, \mc{O}_{\mb{P}^1_U}(n + a)) = 0
\]
since $n+a \geq -1$ in all cases.  Thus $R^1p_* \sh{I}^n(n+a,b)$ is supported at $\{q\}$.

We apply the theorem on formal functions \cite[Theorem~III.11.1]{Ha}.  For some $\ell, n \geq 0$, consider the commutative diagram
\begin{equation}
\label{bigsquare-eq} \xymatrix{
 \widehat{(R^1 p_* \sh{I}^n(n+a,b))_q} \ar[d] \ar[r]^{\cong} &  \varprojlim_{\ell} H^1(\ell Z, \sh{I}^n(n+a,b)|_{ \ell Z})\ar[d] \\
R^1p_* \sI^n(n+a,b) \otimes \sO_{\PP^1,q}/\mf{m}^{\ell}_q \ar[r] &  R^1p_* (\sI^n(n+a, b)|_{\ell Z}). }
\end{equation}
Here, the top row is the isomorphism guaranteed by the theorem on formal functions, and the bottom row is the natural morphism between
the $\ell$th terms of the respective inverse limits.  Note that
\[
H^1(\ell Z, (\sh{I}^n(n+a,b))|_{ \ell Z}) \to H^1(\ell Z, \sh{I}^n \cdot (\mc{O}_T(n+a,b)|_{ \ell Z}))
\]
is an isomorphism, since moving the $\sh{I}^n$ outside of the restriction changes the sheaf on a $0$-dimensional set at most. Thus taking $d = 0$ in Lemma~\ref{lem-miracle} tells us exactly about the inverse limit $\varprojlim_{\ell} H^1(\ell Z, \sh{I}^n(n+a,b)|_{ \ell Z})$.
That lemma shows that the limit is trivially $0$ if $a \geq -1$; while if $a \leq -2$, then the maps in the limit stabilize for all $\ell, n \geq -a-1$, and thus the right-hand map in \eqref{bigsquare-eq} is an isomorphism for such $\ell, n$.  So the limit is isomorphic to
$\struct_{(-a-1) z} \oplus \cdots \struct_z$ as claimed.  This also shows that the limit is supported on $\sO_{\PP^1,q}/\mf{m}^{\ell}_q$
and so the left-hand map in \eqref{bigsquare-eq}, and in fact the natural map
\[ R^1p_* \sI^n(n+a,b) \to R^1p_* \sI^n(n+a,b) \otimes \sO_{\PP^1,q}/\mf{m}^{\ell}_q,\] are isomorphisms.  Then for $\ell, n \geq -a-1$ the bottom arrow of \eqref{bigsquare-eq} is an isomorphism, and finally it follows that the natural map $\gamma$ of the lemma statement is an isomorphism as well.
\end{proof}

We now apply the lemma to compute $R^1p_* \sO(a,b) \otimes \sR^{\phi^m}$.

\begin{corollary}\label{cor-R1p-R}
Let  $a, b \in \ZZ$  and $m, n \in \NN$. If $n \geq -a-1$, the length of $R^1p_* \sh{R}_n^{\phi^m}(a,b)$ is $2(m \binom{-a}{2} + \binom{-a}{3})$; in particular, $R^1p_* \sh{R}_n^{\phi^m}(a,b) = 0$ if $a \geq -1$. When $a \leq -2$, the (scheme-theoretic) support of $R^1p_* \sh{R}_n^{\phi^m}(a,b)$ is equal to
\[C(a,m) := (-a-1)f_0 + \cdots + (-a-1)f_{m-1} + (-a-2)f_m + \cdots + f_{m-a-3} + (-a-1) q_0 + \cdots +q_{m-a-3}.\]
In particular, this support does not depend on $n \geq -a-1$.
\end{corollary}
\begin{proof}
By our choice of $\tau$, the points $f_i, q_j$ are all distinct.  Let $n \geq -a-1$.  We have  $\sh{R}_n^{\phi^m}(a,b) \cong \mc{I}_n^{\phi^m}(n+a, b')$, where $b':= \binom{n+1}{2} +nm+b$.  For fixed $k \geq 0$, the multiplicity $e_k$ of $\mc{I}_n^{\phi^m}$ locally at $Q_k$ can be calculated as in \eqref{mult-cases-eq}.  In particular, if $\mc{I}$ is the ideal sheaf of $Q_k$, then in a small enough neighborhood $V$ of $q_k$ we have $(\mc{I}_n^{\phi^m}(n+a, b')) \vert_{p^{-1}(V)} \cong \mc{I}^{e_k}(e_k + [n + a -e_k], b') \vert_{p^{-1}(V)}$.  Here, $a' := n + a - e_k$ satisfies $a' = a$ if $k \leq m -1$, $a' = a - m + k + 1$ if $m-1 \leq k \leq n + m -1$, and $a' = n + a \geq -1$ if $k \geq n + m -1$.    It is now immediate from Lemma~\ref{lem-pushfwd1} that $R^1p_* \sh{R}_n^{\phi^m}(a,b)$ has multiplicity $\max(0,-a'-1)$ at $q_k$; since a symmetric result holds locally at the $f_k$, this shows exactly that $R^1p_* \sh{R}_n^{\phi^m}(a,b)$ is supported scheme-theoretically at $C(a,m)$.

Considering the lengths of the sheaves in Lemma~\ref{lem-pushfwd1} gives
\begin{multline*}
 \len(R^1p_* \sh{I}_n^{\phi^m}(n+a,b')) =
2 (m \binom{-a}{2} + \binom{-a-1}{2} + \binom{-a-2}{2} + \cdots + \binom{2}{2} + 0) \\
= 2 (m \binom{-a}{2} + \binom{-a}{3}).
\end{multline*}
This gives the first statement.
\end{proof}
\begin{remark}
\label{rem-notample} We can now justify the earlier assertion that the sheaves $\{\sR^{\phi^m}_n\}_{n \in \NN}$ do not form an ample sequence in the sense of \cite{VdB1997}.  By the corollary, if $m \geq 1$ and $a \leq -2$, then using Proposition~\ref{prop-5} we have
\[\dim_{\kk} H^1(T, \sO(a,b)\otimes \sR_n^{\phi^m}) \geq \dim_{\kk} H^0(\PP^1, R^1p_* \sO(a,b) \otimes \sR_n^{\phi^m}) > 0\]
for all $n \geq -a-1$.  If $m=0$ we must take $a \leq -3$.
\end{remark}
We are now ready to finish the proof that the modules $\QM(\sO(a,b),m)$ are noetherian.

\begin{theorem}\label{thm-Q}
Let $\sh{F}$ be an invertible sheaf on $T$  and let  $m \in \NN$.  Then the right $R$-module $\QM(\sF,m)$ is noetherian and is, up to finite dimension, an extension of finitely many point modules.
\end{theorem}
\begin{proof}
 Recall that
\[\QM(\sh{F}, m) = \bigoplus_{n\geq 0} H^0(\PP^1,R^1p_* (\sh{F} \otimes \sh{R}_n^{\phi^m})).\]
Let $\sF \cong \struct(a,b)$ for some $a,b \in \ZZ$. If $a \geq -1$, then $\QM(\sF,m)=0$ by Corollary~\ref{cor-R1p-R}. Thus without loss of generality we may assume that $a \leq -2$. Let $C:=C(a,m)$ as defined in Corollary~\ref{cor-R1p-R}.

For any $\ell \geq 1$, consider the natural map \beq\label{F-restrict} \sF \to \bigoplus_{c \in C} \sF|_{\ell T_c} \eeq of sheaves on $T$. This induces a natural map \beq\label{fn} \xymatrix{
 R^1p_*(\sF\otimes \sR_n^{\phi^m}) \ar[r] & \bigoplus_{c \in C} R^1p_* (\sF|_{\ell T_c} \otimes \sR_n^{\phi^m}). }
\eeq By Lemma~\ref{lem-pushfwd1}, this map is an isomorphism for $n, \ell \geq -a-1$.

Notice that if we sum \eqref{fn} over all $n$ and take global sections, we have precisely the module homomorphism $\QM(\sF, m) \to \bigoplus_{c \in C} \QM(\sF|_{\ell T_c}, m)$ induced by \eqref{F-restrict}.  Setting $n_0 = \ell_0 = -a-1$, then \[ \QM(\sF, m)_{\geq n_0} \cong \bigoplus_{c\in C} \QM(\sF|_{\ell T_c}, m)_{\geq n_0}\] for all $\ell \geq \ell_0$. Since each $\sF|_{\ell T_c}$ is supported on  a (fat) fiber of $p$, the modules $\KM(\sF|_{\ell T_c}, m)$ all vanish, and from Proposition~\ref{prop-5} we obtain that $\QM(\sF|_{\ell T_c},m) \cong \CM^1(\sF|_{\ell T_c}, m)$.
Thus
\[\QM(\sh{F}, m)_{\geq n_0} \cong \bigoplus_{c \in C} \CM^1( \sh{F} \vert_{\ell_0 T_c}, m)_{\geq n_0}.\]
The result follows from Lemma~\ref{lem-Nnoeth}.
\end{proof}

We now give the delayed proof of the main technical lemma, Lemma~\ref{lem-miracle}.
\begin{proof}[Proof of Lemma~\ref{lem-miracle}]

The proof is somewhat long, so we break it up into steps.

\noindent \textbf{Step 1:  Setting up identifications.}

Recall that $d \geq 0,  a, b \in \mb{Z}$ are fixed, and we have sheaves $\sH(\ell)_{n}:= \sI_Q^{n-d}\cdot \sO_{\ell Z}(aX+bW)\otimes \sL_n$ and a cohomology module $H(\ell):= \bigoplus_{n \geq \max(1,d)} H^1(T, \sH(\ell)_n)$.  Fix $\ell \geq \max(1, -a-1)$ and let $\mc{H}_n := \mc{H}(\ell)_{n}$ and $H:= H(\ell)$.

Let $U^+:= \ell Z \cap (T \smallsetminus Y)$ and let $U^- := \ell Z \cap (T \smallsetminus X)$ be charts on $\ell Z$. Let $S^+:= \kk[u, v]/(v^{\ell})$ and let $S^-:= \kk[u^{-1},v]/(v^{\ell})$.  We identify $U^+$ with $\spec S^+$ and $U^-$ with $\spec S^-$. Let $U^{\pm}:= U^+ \cap U^-$ and let $S^{\pm}:= \sO(U^{\pm}) = \kk[u,u^{-1}, v]/(v^{\ell})$.  We will use \v{C}ech cohomology on this open cover to compute $H_n = H^1(T, \mc{H}_{n})$ and to study the $R$-action on $H$.  We will show that $H$ is (up to finite dimension) an extension of $\binom{-a-d}{2}$ modules with Hilbert series $1/(1-s)$, each of which is a shifted point module.  It will follow that $H$ itself is noetherian.

Let $n_0 =  \max(d, -a-1)$ and let $n \geq n_0$ in the following calculations. We will regard both $\sH_n$ and $\sL_n(aX+bW)|_{\ell Z}$ as subsheaves of the constant sheaf $\big(\kk(u)[v]/(v^{\ell})\big) t^n$ on $\ell Z$, where $t^n$ is a placeholder to remind us that the action of $R$ is twisted by the appropriate power of $\varphi$. Let $V$ be an open subset of $\ell Z$.  If $s = \alpha(u,v) t \in R_1$ then the multiplication map
\[ \sH_n(V)  \otimes \kk s \to \sH_{n+1}(V)  \subseteq \big(\kk(u)[v]/(v^{\ell})\big) t^{n+1} \]
is given explicitly by $f(u,v) t^n \otimes s \mapsto f\alpha^{\phi^n} t^{n+1}$.

\noindent \textbf{Step 2:  Identifying a basis for the cohomology group.}

Since $\tau$ is general, $W^{\phi^i} \neq Z$ for any $i$.  Recalling that $\mc{L} = \mc{O}_T(Y + W)$ and so
\[\mc{L}_n = \mc{O}_T(Y + Y^{\phi} + \cdots+ Y^{\phi^{n-1}} + W + W^{\phi} + \cdots + W^{\phi^{n-1}}),\]
we get
\[ \sL_n(aX+bW)|_{\ell Z} = \sO_{\ell Z} (aX + Y + Y^{\phi} + \cdots+ Y^{\phi^{n-1}})t^n.\]
As a divisor, $(aX + Y + Y^{\phi} + \cdots+ Y^{\phi^{n-1}})|_Z$ is concentrated at $P$ and $Q$.   We thus have $\sH_n(U^\pm) = \sO_{\ell Z}(U^\pm) t^n = S^{\pm} t^n$ under these identifications.  Further, if $U= U^+$ or $U= U^-$, then under these identifications $\sH_n(U)$ is contained in $S^{\pm}t^n$. The image of the \v{C}ech differential
\[ \partial_n:  \sH_n(U^+) \oplus \sH_n(U^-) \to \sH_n(U^\pm)\]
is thus equal to $\sH_n(U^+)+ \sH_n(U^-) \subseteq \sH_n(U^\pm)$.

We claim that for $n \geq n_0$, the image of $\partial_n$ is equal to the vector space  $L t^n$, where
\begin{multline*}
  L = \\
 \spann(u^i v^j \st \max(0,-a-d) \leq j < \ell, \mbox{ or } 0 \leq j \leq  -a-d-1 \mbox{ and } i \geq - a, \mbox{ or }
  0 \leq j \leq -a-d-1 \mbox{ and } i\leq j+d) .
\end{multline*}
In particular, $L = S^{\pm}$ if $a+d  \geq -1$. Assume the claim for the moment.  This allows us to identify $H_n$ with $\displaystyle \left( S^{\pm}/L  \right)  t^n$, that is with the span of (the images of) the monomials
\[ \{ u^i v^j  \st 0 \leq j \leq i-d-1, d+1 \leq i \leq -a-1\} \cdot t^n.\]
If  we fix $i$ in the range $d+1 \leq i \leq -a-1$, we obtain the $i-d$ monomials $\{u^i, u^i v, \ldots, u^i v^{i-d-1}\} \cdot t^n$.   We see that $u^i \kk[v]/(v^{i-d})t^n$ is  a direct summand of $H^1(\ell Z, \mc{H}_n)$ as a $\kk[v]/(v^{\ell})$-module. Since $v$ is a local coordinate on the base $\PP^1$ at $q$, we have
\[ H^1(\ell Z, \mc{H}_n) \cong \sO_{ q} \oplus \sO_{2 q} \oplus \cdots \oplus \sO_{(-a-d-1)q}\]
as claimed in  the statement of the lemma.

For any $d \geq 0$, multiplication by $t \in R_1$ takes $\sH_n(U^\pm) = S^\pm t^n$ to $\sH_{n+1}(U^\pm) = S^\pm t^{n+1}$ and takes $\im
\partial_n = Lt^n$ to $\im \partial_{n+1}= L t^{n+1}$.  Since $H_n = \coker \partial_n$, multiplication by $t$ induces a bijection
$\mu_t : H_n \to H_{n+1}$.  The claim that restriction induces an isomorphism $H(\ell+1)_{n} \to H(\ell)_n$ is also immediate.

\noindent \textbf{Step 3:  Proving the claim that $\im \partial_n = Lt^n$.}

We drop the $t^n$ coefficients and identify $\sH_n(U)$  with a subspace of $S^{\pm}$, for $U = U^+, U^-, U^{\pm}$.  Let $J_k := v^k S^{\pm}$. We will show that $J_{ k} \cap L = J_{ k} \cap \im \partial_n$ for all $k\geq 0$.

By Proposition~\ref{prop-fund-pts}, $Q \in Y^{\phi^i} \cap Z$ for $i \geq 0$.  Since $(Y^{\phi^i}.Z) = 1$, the curve $Y^{\phi^i}$ meets $Z$ transversely at $Q$, and nowhere else.  On the other hand, $X \cap Z = \{P\}$.  Since $U^+ = U \smallsetminus \{ Q \}$, we have
\[ \sH_n(U^+) = \mc{O}_{\ell Z}(aX + Y + Y^{\phi} + \dots + Y^{\phi^{n-1}})(U^+) = u^{-a} S^+.\]

To compute $\sH_n(U^-)$, let $x\eta + y\xi$ be the $(1,i)$-form defining $Y^{\phi^i}$, where $\eta, \xi$ are homogeneous of degree $i$ in $\kk[w, z]$.   Since $Y^{\phi^i}$ meets $Z$ transversely at $Q$, the germ of $\eta$ in $\kk[v]/(v^{\ell})$ is contained in the maximal ideal of $\kk[v]/(v^{\ell})$, and the germ of $\xi$ is invertible.  Thus $Y^{\phi^i}$ is defined on $U^-$ by $r_i := u^{-1} + \alpha_i$, where $\alpha_i \in v \kk[v]/(v^{\ell})$.  In $S^{\pm}$ we have $r_i^{-1} = u(1-\alpha_i u + (\alpha_i u)^2 - \cdots \pm (\alpha_i u)^{\ell-1})$. Let $s:= r_0 r_1 \cdots r_{n-1}$,
and let $h:= s^{-1} u^{-n}$.
Thus
\[
\sL_n(aX+bW)|_{\ell Z}(U^-) = (aX + Y + Y^{\phi} + \cdots + Y^{\phi^{n-1}})|_{\ell Z}(U^-) = s^{-1} S^- = u^n h S^-.
\]

For $0 \leq i \leq \ell-1$, there are elements $\epsilon_i \in v^i\kk[v]/(v^{\ell})$ so that
 so that $h = 1 + \epsilon_1 u + \epsilon_2 u^2 + \cdots + \epsilon_{\ell-1} u^{\ell-1}$.
Multiplying by $\sI_Q^{n-d}$, we have
\[ \sH_n(U^-) = (u^{-1}, v)^{n-d} u^n h S^- = u^dhS^- + u^{d+1}vhS^- + \cdots + u^n v^{n-d} h S^-.\]

Now that we have a detailed description of $\sH_n(U^-)$ and $\sH_n(U^+)$, we make some observations.  First, $\im \partial_n \subseteq L$. To see this, note that since $S^-$ is spanned by monomials in $\kk[u^{-1}, v]$, therefore $\sH_n(U^-)$ is spanned by elements of the form
\[ u^iv^j h = u^i v^j + u^{i+1} v^j \epsilon_1 + u^{i+2} v^j \epsilon_2 + \cdots + u^{i + \ell  -1}v^j \epsilon_{\ell -1}  \]
where $i \leq \min(n, j + d)$ and  $0 \leq j \leq \ell$. Since $v^k | \epsilon_k$ for each $k$, $u^iv^j h$ is a sum of terms of the form $u^{i'}v^{j'}$ with $i' \leq j' + d$, and these are all in $L$ by definition.  Obviously $\sH_n(U^+) \subseteq L$, so the first observation follows.  Second, we observe that for each monomial $u^iv^j \in L$ there is an element in $\im \partial_n$ of the form $u^iv^j + s$ where $s \in J_{ j + 1} \cap L$.  If $i \geq -a$ this is obvious since  $u^{-a} S^+ = \sH_n(U^+) \subseteq \im \partial_n$; otherwise, we have $u^iv^j$ with $i \leq j + d$ and $i \leq -a-1$, so in particular $i \leq n$ by the choice of $n$.   The observation now follows  by the analysis above of the elements $u^iv^j h$ which span $\sH_n(U^-)$.

Given the two observations above, a simple downward induction on $k$ proves that $J_{ k} \cap L = J_{ k} \cap \im \partial_n$ for all $0 \leq k \leq \ell$.  Taking $k = 0$, we obtain that $L = \im \partial_n$ as claimed.

\noindent \textbf{Step 4:  Finding the filtration by point modules.}

We next show that there is a complete flag on the vector space $S^\pm/L$ that induces an $R$-module filtration of $H$ whose subfactors are point modules.
To see this, let us study further the action of $R$ on the sheaf $\sH := \bigoplus \sH_n$.
As we have already noted, any $s \in R_1$ acts on $\sH_n$ and its cohomology as the image of $s^{\phi^n}$ under the natural maps
\[   H^0(T, \sR_1^{\phi^n}) \to H^0(T, \sR_1^{\phi^n}|_{\ell Z}) \to H^0(\ell Z, \sI_Q \cdot \sL_1^{\phi^n}|_{\ell Z})\cdot t.
\]
Let $\sR':= \sI_Q \cdot \sL_1^{\phi^n}|_{\ell Z}$.

We use the open cover $U^+, U^-$ of $\ell Z$ to compute $H^0(\ell Z, \sR')$.  Recall that $Y^{\phi^n}|_{\ell Z}$ is defined on $U^-$ by $r_n = u^{-1} + \alpha_n$, for some $\alpha_n \in v \kk[v]/(v^{\ell})$.  Let $g: = 1 - u \alpha_n + (u \alpha_n)^2 - \cdots \pm (u \alpha_n)^{\ell-1}$, so $r_n^{-1} = u g$. Note that $g \in S^+ \smallsetminus u S^+$.  Since
\[
\mc{L}_1^{\phi^n}|_{\ell Z} = \mc{O}(W^{\phi^n} + Y^{\phi^n})|_{\ell Z} =  \mc{O}_{\ell Z}(Y^{\phi^n}),
\]
we compute that
\beq\label{action} H^0(\ell Z, \sR') = \Bigl( (u^{-1}, v) \cdot ug S^- \cap S^+ \Bigr)  = \Bigl( g \kk[v]/(v^{\ell}) + uv g \kk[v]/(v^{\ell}) \Bigr). \eeq

We now begin to construct the flag on $S^\pm/L$. For $d \leq e \leq -a-1$, define
\[ V(e) := \spann( u^i v^j + L \st i \leq j + e ) \subseteq S^\pm/L.\]
By \eqref{action}, if $f \in H^0(\ell Z, \sR')$, then $V(e) \cdot f \subseteq V(e)$.  Since $R_1$ acts on $H_n = (S^\pm/L) t^n$
 as its image in  $H^0(T, \sR')\cdot t$, the vector spaces $\underline{V}(e) := \bigoplus_{n \geq n_0} V(e) t^n$
are $R$-submodules of $H_{\geq n_0}$.  We therefore  have a  chain of $R$-modules
\[ 0 = \underline{V}(d) \subseteq \underline{V}(d+1) \subseteq \cdots \subseteq \underline{V}(-a-1) = H_{\geq n_0}.\]
Note from Step 2 that the multiplication-by-$t$ map $\mu_t$ induces a bijection from $\underline{V}(e)_n$ to $\underline{V}(e)_{n+1}$, for $n \geq n_0$.

Fix $d+1 \leq e \leq -a-1$ and consider the subfactor
\[ \frac{\underline{V}(e)}{\underline{V}(e-1)} = \bigoplus_{n \geq n_0} \frac{V(e)}{V(e-1)} t^n.\]
Now, $V(e)/V(e-1)$ has basis
\[ u^e + V(e-1), \ u^{e+1} v + V(e-1), \ \ldots, \ u^{-a-1}v^{-a-1-e} + V(e-1).\]
For $0 \leq c \leq -a-e$, let
\[ W(c):= \spann(u^{j+e} v^j + V(e-1) \st j \geq c) \subseteq V(e)/V(e-1).\]
By \eqref{action}, multiplication by $f \in H^0(\ell Z, \sR')$ preserves the spaces $W(c)$.  Thus $\underline{W}(c) := \bigoplus_{n \geq n_0} W(c) t^n$ is an $R$-submodule of $\underline{V}(e)/\underline{V}(e-1)$, and there is a  chain
\[ 0 = \underline{W}(-a-e) \subseteq \underline{W}(-a-e-1) \subseteq \cdots \subseteq \underline{W}(0) =
\underline{V}(e)/\underline{V}(e-1)\] of $R$-modules.  Again, the map $\mu_t$ gives a bijection from $\underline{W}(c)_n \to \underline{W}(c)_{n+1}$ for $n \geq n_0$.

For $0 \leq c \leq -a-e-1$ and $n \geq n_0$, the element $a_n := \bigl(u^{c+e} v^c + W(c+1)\bigr)t^n$ generates the 1-dimensional vector space $(\underline{W}(c)/\underline{W}(c+1))_n$.  Our analysis of $\mu_t$ above shows that  $a_n t = a_{n+1}$.  Thus $\underline{W}(c)/\underline{W}(c+1)$ is cyclic and torsion free, with Hilbert series $s^{n_0}/(1-s)$.  It is thus a (shifted) point module, and  is noetherian.  Since $H$ is, up to finite dimension, an extension of such modules, $H$ itself is noetherian.

Finally, since the filtration of $H$ is induced from a complete flag on $S^{\pm}/L$, the number of point modules appearing as subfactors is equal to $\dim_{\kk} S^\pm/L = \binom{-a-d}{2}$.
\end{proof}

\section{Completing the proof of Theorem~\ref{thm-main}}\label{K-MODULES}

In this section we prove that the modules $\KM(\sO(a,b),m)$ are noetherian (in fact, finite dimensional over $\kk$), and complete the proof of Theorem~\ref{thm-main}.

We will work with both $\sR$ and $\sA = \sR(\mathbb{1})$.  We begin by making some computations of  cohomology of the sheaves $\sh{A}_n^{\sigma^m}(a,b)$.

\begin{lemma}\label{lem-R1p-A}
Let $m \in \NN$ and $ a, b \in \ZZ$. For all $n \geq -a-1$, $R^1p_* \sh{A}^{\sigma^m}_n(a,b)$ is a 0-dimensional sheaf of length $2(m \binom{-a}{2} + \binom{-a}{3})$.
\end{lemma}
\begin{proof}
The proof uses similar methods as those in the last section:  we use the theorem on formal functions followed by an explicit \v{C}ech cohomology computation.  Recall that we write $\mc{A}_n^{\sigma^m} = \sh{I}_n^{\sigma^m} \mc{L}_n^{\sigma^m}$ for an ideal sheaf $\mc{I}_n^{\sigma^m}$, which in this case defines a $0$-dimensional subscheme supported at $F$ and $Q$.

First, a similar argument to the proof of Lemma~\ref{lem-pushfwd1} shows that the sheaf $R^1p_* \sA_n^{\sigma^m}(a, b)$ is 0-dimensional, and in fact is supported at $\{f, q\}$.  Of course the structure at those two points will be symmetric and in this proof we concentrate on the point $f$. For $\ell \geq 1$, let $\sh{F}_{\ell} :=  \sI_n^{\sigma^m} \cdot \sO_{\ell W}((a+n) Y )$. Note that the natural map $\sA_n^{\sigma^m}(a,b)|_{\ell W} \to \sF_\ell$ induces an isomorphism on $H^1$.

The theorem on formal functions shows that \beq \label{A-formal-eq} \big(R^1p_* \sA_n^{\sigma^m}(a, b)\big)_f \cong \varprojlim_{\ell} H^1(\ell W, \mc{F}_{\ell}). \eeq
Let $w:=v^{-1}$ and let $u, u^{-1}$ be local coordinates on $W$.  We define $S^-:= \kk[u^{-1},w]/(w^{\ell})$, $S^+:= \kk[u, w]/(w^{\ell})$, and $S^{\pm}:= \kk[u, u^{-1},w]/(w^{\ell})$. Let $U^+, U^-, U^\pm$ respectively be the spectra of $S^+, S^-$, and $S^{\pm}$, as open subsets of $\ell W$.

For brevity of notation, write $\mc{F} = \mc{F}_{\ell}$.  By Lemma~\ref{lem-pointchains2}(2), the ideal sheaf $\sI^{\sigma^i}_1$ is equal locally at $F$ to $(u, w^i)$, and so we get
\[
\sF(U^+) = (u, w^m)(u, w^{m+1}) \ldots (u, w^{m+n-1})S^+.
\]
Further, $\sF(U^-) = u^{n+a}S^-$, and $\sF(U^\pm) = S^\pm$. We need to calculate
\[ H^1(\ell W, \sF) \cong \sF(U^\pm)/(\sF(U^+)+\sF(U^-)).\]
If we take $\ell$ large enough, a basis of monomials for the factor $S^+/\sF(U^+)$  is the same as a basis of monomials for $\kk[u,w]/ (u, w^m)(u,w^{m+1}) \ldots (u, w^{m+n-1})$, and this was already calculated in Proposition~\ref{prop-0dim}.  Namely, setting $j_h = m + h-1$ for $1 \leq h \leq n$, for the indices $j$ with $\sum_{h = 1}^{k-1} j_h \leq j < \sum_{h = 1}^{k} j_h$ we get $u^iw^j$ for all $0 \leq i < n - k + 1$.  Thus we see that $\sF(U^\pm)/(\sF(U^+)+\sF(U^-))$ has the following basis of monomials:  for the indices $j$ with $\sum_{h = 1}^{k-1} j_h \leq j < \sum_{h = 1}^{k}j_h $, we get $u^iw^j$ for all $i$ satisfying $n + a < i < n - k + 1$.

There are
\[ m(-a-1)+(m+1)(-a-2) + \cdots + (m-a-2)(1)  = m \sum_{i = 1}^{-a-1} i  + \sum_{i=1}^{-a-2}i(-a-1-i)= m \binom{-a}{2} +
\binom{-a}{3}\]
 of these.
Thus $\len H^1(\ell W, \sF_{\ell}) = m \binom{-a}{2} + \binom{-a}{3}$ for all $\ell \gg 0$.  In addition, since we get the same $\kk$-basis for all large $\ell$, clearly the maps in the inverse limit in \eqref{A-formal-eq} are isomorphisms for all large $\ell$, and so we have calculated the length of $R^1p_* \sA_n^{\sigma^m}(a, b)$ at $f$.  By symmetry the length at $q$ is the same, and so we have
\[\len R^1p_* \sF = 2(m \binom{-a}{2} + \binom{-a}{3}),\]
as claimed.
\end{proof}

\begin{proposition}\label{prop-H1p*A}
For any $m \in \NN$ and $a,b \in \ZZ$, we have $H^1 (\PP^1, p_* \sh{A}_n^{\sigma^m}(a,b)) = 0$  for $n \gg 0$.
\end{proposition}
\begin{proof}
 Let $\sh{F} := \sh{I}_n^{\sigma^m}((n+a)Y)$, and suppose that $a \leq -1$.
We begin by computing $p_* \sF$. We let $u, u^{-1}, v, v^{-1}$ be local coordinates on $T$, where $u^{\pm}$ give coordinates along the fibers of $p$ and $v^{\pm}$ give coordinates along the base of $p$.  We denote the four open affines covering $T$ by $U^+_+$, $U^-_+$, $U^+_-$, and $U^-_-$, where for example $U^+_- = \spec\kk[u,v^{-1}]$. Note that in this coordinate system, $F$ is defined by $u=0, v^{-1}=0$, and $Q$ is defined by $u^{-1}=0, v=0$. Then by Lemma~\ref{lem-pointchains2}(2) we have:
\begin{gather*} \sh{F}(U^+_-) = (u,v^{-m})(u, v^{-(m+1)})\cdots (u,v^{-(m+n-1)}) \kk[u,v^{-1}]; \ \ \
\sh{F}(U^-_-) = u^{n+a} \kk[u^{-1},v^{-1}]; \\
\sh{F}(U^-_+) = u^{n+a} (u^{-1},v^m)(u^{-1},v^{m+1})\cdots (u^{-1},v^{m+n-1})\kk[u^{-1},v];\ \text{and}\ \ \sh{F}(U^+_+) = \kk[u,v].
\end{gather*}

Let $U_+ = \spec \kk[v]$ and $U_- = \spec \kk[v^{-1}]$ be charts for $\PP^1$. Now, if $n \geq -a$, then we have
\[
 (p_*\sh{F})(U_-) = \sh{F}(U^+_-) \cap \sh{F}(U^-_-)
 = \bigoplus_{i=0}^{n+a} u^i  v^{-(m(n-i) +\binom{n-i}{2})} \kk[v^{-1}]
\]
 and
\[
 (p_* \sh{F})(U_+) = \sh{F}(U^+_+) \cap \sh{F}(U^-_+)
 = \bigoplus_{i=0}^{n+a} u^i v^{m(i-a) + \binom{i-a}{2}} \kk[v].
\]
Note that the intersections
 $p_* \sh{F}(U_+) \cap u^i \kk[v, v^{-1}]$ and
 $p_* \sh{F}(U_{-}) \cap u^i \kk[v, v^{-1}]$ give a compatible direct sum decomposition of
 $p^* \sh{F}(U_+) $ and $p^* \sh{F}(U_{-})$.  Therefore, if $n \geq -a \geq 1$, then
 \begin{align*}
  p_* \sh{F} & \cong \bigoplus_{i=0}^{n+a} \struct(-(m(n-i) + \binom{n-i}{2} +m(i-a) + \binom{i-a}{2})) \\
   & \cong \bigoplus_{i=0}^{n+a} \struct(-(m(n-a) + \binom{n-i}{2} + \binom{i-a}{2})).
   \end{align*}

  By the projection formula,
 \[ p_* \sh{A}_n^{\sigma^m}(a,b) \cong p_* \sh{F}(0,nm +\coeff + b) \cong (p_* \sh{F})(nm+\coeff+b).\]
 To show that $H^1(\PP^1,p_*(\sh{A}_n^{\sigma^m}(a,b))) = 0$ for $n \gg 0$, it suffices to show that  there is some $n_0 \geq 0$ such that the sheaf
\[ \sh{G}_{n,i} := \struct(-(m(n-a) + \binom{n-i}{2} + \binom{i-a}{2}) +nm + \coeff + b)
 = \struct( \coeff - \binom{n-i}{2} - \binom{i-a}{2} + b')\]
 (where $b'=ma+b$ does not depend on $n$ or $i$) has vanishing $H^1$ for all $n \geq n_0$ and $0 \leq i \leq n+a$.

Consider the function $-\binom{n-i}{2} - \binom{i-a}{2}$ as $i$ varies.  On the interval $0 \leq i \leq n+a$ this is equal to
 \[ \frac{1}{2}[ -(n-i)(n-i-1) - (i-a)(i-a-1)]\]
 which is quadratic in $i$ with leading term $-i^2$.  It is symmetric around $i = (n+a)/2$ and thus attains its minimum at the
 endpoints
$i=0$
 and $i=n+a$.  This minimum value is $-\binom{n}{2} - \binom{-a}{2}$.
 Now we may choose $n_0$ so that $\coeff - \binom{n}{2} - \binom{-a}{2} + b' \geq -1$ for all $n \geq n_0$. Then for any $n \geq n_0$ and $0 \leq i \leq n+a$, we have
\[ \coeff - \binom{n-i}{2} - \binom{i-a}{2} + b' \geq \coeff - \binom{n}{2} - \binom{-a}{2} + b' \geq -1.\]
  Thus $H^1(\PP^1, \sh{G}_{n,i}) = 0 $ for all $i$, and so $H^1(\PP^1, p_* \sh{A}_n^{\sigma^m}(a,b)) = 0$.

If $n, a \geq 0$, then a similar computation as above gives that
\[ p_* \sF \cong \bigoplus_{i=0}^a \struct(-(m(n-i) + \binom{n-i}{2}))^{\oplus 2} \oplus
    \bigoplus_{i=a+1}^{n-1} \struct(-(m(n-a)  + \binom{n-i}{2} + \binom{i-a}{2})).\]
The proof that $H^1(\PP^1, p_* \sA_n^{\sigma^m}(a,b)) = 0$ for $n \gg 0 $ is similar to that for $a\leq -1$, and we leave the details to the reader.
\end{proof}

\begin{theorem}\label{thm-K}
  Let $\tau = \tau(\rho, \theta)$ for a pair $(\rho, \theta)$ algebraically independent over $\FF$.  For any $a,b \in \ZZ$ and $m \in
  \NN$, the $R$-module $\KM(\sO(a,b), m)$
is finite-dimensional and therefore noetherian.
\end{theorem}
\begin{proof}
Fix $a,b$.  Recall that $\KM(\sO(a,b), m) = \bigoplus_{n \in \NN} H^1(\PP^1,p_* \sh{R}_n^{\phi^m}(a,b))$. By Theorem~\ref{thm-generic},
  \[ \dim_{\kk} H^1(T, \sh{R}_n^{\phi^m}(a,b) ) \leq \dim_{\kk}H^1(T, \sh{A}_n^{\sigma^m}(a,b)).\]
    By Corollary~\ref{cor-R1p-R} and Lemma~\ref{lem-R1p-A}, if $n \geq -a-1$ we have
  \[ \len R^1p_* \sh{R}_n^{\phi^m}(a,b) = 2(m \binom{-a}{2} + \binom{-a}{3}) = \len R^1 p_* \sh{A}_n^{\sigma^m}(a,b).\]

But now, by the Leray spectral sequence we have
\begin{align*}
  \dim_{\kk} H^1(\PP^1, p_* \sh{R}_n^{\phi^m}(a,b))
   & = \dim_{\kk} H^1( T,  \sh{R}_n^{\phi^m}(a,b)) - \dim_{\kk} H^0(\PP^1,R^1 p_* \sh{R}_n^{\phi^m}(a,b)) \\
   & \leq \dim_{\kk} H^1(T, \sh{A}_n^{\sigma^m}(a,b)) - \dim_{\kk} H^0(\PP^1, R^1 p_* \sh{A}_n^{\sigma^m}(a,b)) \\
   & = \dim_{\kk} H^1(\PP^1, p_* \sh{A}_n^{\sigma^m}(a,b)).
   \end{align*}
By Proposition~\ref{prop-H1p*A}, this last term vanishes for $n \gg 0$.  So $\KM(\sO(a,b), m)$ is finite-dimensional, as claimed.
   \end{proof}

We are finally ready to give the proof of the main theorem.
 \begin{proof}[Proof of Theorem~\ref{thm-main}]

Let $\tau = \tau(\rho, \theta)$ for a pair $(\rho, \theta)$ algebraically independent over $\FF$.  Let $R:= R(\tau)$.  For any $a, b \in\ZZ$ and $m \in \NN$, by Theorem~\ref{thm-Q} $\QM(\struct(a,b),m)$ is a noetherian right $R$-module.  By Theorem~\ref{thm-K}, $\KM(\struct(a,b),m)$ is noetherian.
By Corollary~\ref{cor-Leray}, $R$ is therefore right noetherian.

Now, $R^{op} \cong R(\tau^{-1})$ by Theorem~\ref{prop-2}.  Since $\tau^{-1} = \tau(\rho^{-1}, \theta^{-1})$, $R$ is also left noetherian. By Proposition~\ref{prop-Koszul}, $R$ is defined by the relations \eqref{rels}; setting
$c := \zeta/\epsilon = \frac{\theta-1}{\theta+1}$ and $d := \delta/\gamma = \frac{\rho-1}{\rho+1}$
we obtain the given presentation of $R$. The remaining properties of $R$ are given by Proposition~\ref{prop-Koszul}.
   \end{proof}

We conclude with a brief discussion of some questions suggested by the results in this paper, which we hope to address in further work.
We still do not have a very deep understanding of the category of graded $R$-modules; for instance, is it closely related to
some more geometrically defined category?  Some other important questions are whether $R$ satisfies the Artin-Zhang $\chi_1$ condition, and what the structure of the point modules over $R$ is.

Now that we know that noetherian GK-4 birationally commutative surfaces exist, this naturally raises the question of whether they can be  classified, thus completing the classification of noetherian birationally commutative surfaces.  The work of Diller and Favre in \cite{DF2001} shows that the GK-4 growth type arises in a fairly limited situation:  the field automorphism $\varphi$ must be induced by a birational self-map $\phi$ of some ruled surface which preserves the ruling. Thus there is some hope that the general GK-4 birationally commutative surface is not too different in behavior from the examples we consider in this paper, although it probably will not have such special homological properties. 

Finally, what are the implications of the fact that connected graded noetherian Koszul algebras of finite global dimension are not automatically AS-Gorenstein?  Let $R:= R(\tau)$, for general $\tau$.  We note that  $R$ is not {\em strongly noetherian}:  by a similar proof as that of \cite[Theorem~9.2]{KRS}, $R$ is not generically flat, and therefore there is a commutative noetherian $\kk$-algebra $C$ so that $R\otimes C$ is not noetherian.  If a connected graded  strongly noetherian $\kk$-algebra $R$ has finite global dimension and is Koszul, must it be AS-Gorenstein?  Conversely, a counterexample to this question would be extremely interesting.

\bibliographystyle{amsalpha}

\providecommand{\bysame}{\leavevmode\hbox to3em{\hrulefill}\thinspace}
\providecommand{\MR}{\relax\ifhmode\unskip\space\fi MR }
\providecommand{\MRhref}[2]{%
  \href{http://www.ams.org/mathscinet-getitem?mr=#1}{#2}
}
\providecommand{\href}[2]{#2}

\end{document}